\documentclass[12pt]{amsart}
\pdfoutput=1
\usepackage{graphicx}
\usepackage{amsfonts}
\usepackage{amssymb}
\usepackage{latexsym}
\usepackage{amscd}

\newtheorem{thm}{Theorem}[section]
\newtheorem{cor}[thm]{Corollary}
\newtheorem{lem}[thm]{Lemma}
\newtheorem{prop}[thm]{Proposition}
\newtheorem{empirical}[thm]{Empirical Observation}

\newtheorem*{riemann-mapping-theorem}{Riemann Mapping  Theorem}

\newenvironment{pf}{\proof[\proofname]}{\endproof}
\newenvironment{pf*}[1]{\proof[#1]}{\endproof}
\usepackage{euscript}

\newcommand{\const}{\operatorname{const}}

\newcommand{\gin}{\Gamma^{\text{in}}}
\newcommand{\gout}{\Gamma^{\text{out}}}
\newcommand{\lra}{\longrightarrow}

\newcommand{\cal}[1]{{\mathcal #1}}

\newcommand{\ixp}{\operatorname{ixp}}

\newcommand{\Dom}{\operatorname{Dom}}

\newcommand{\beq}{\begin{equation}}
\newcommand{\eeq}{\end{equation}}

\newcommand{\Reals}{\RR}

\theoremstyle{definition}
\newtheorem{defn}{Definition}[section]

\theoremstyle{remark}
\newtheorem{rem}{Remark}[section]

\renewcommand{\deg}{\operatorname{deg}}

\newcommand{\diam}{\operatorname{diam}}

\renewcommand{\mod}{\operatorname{mod}}
\newcommand{\tl}{\tilde}

\newcommand{\eps}{\epsilon}

\newcommand{\crit}{\operatorname{Crit}}
\newcommand{\sing}{\operatorname{Sing}}
\newcommand{\asym}{\operatorname{Asym}}

\newcommand{\aaa}[1]{{{\mathbf{#1}}}}

\newcommand{\pr}{{\cal P}}

\newcommand{\cla}{{{\aaa P}_0}}
\newcommand{\claa}{{{\aaa P}}}
\newcommand{\clab}{\widetilde{{\aaa P}}} 
\newcommand{\clac}{\widehat{{\aaa P}}} 
\newcommand{\clad}{{{\aaa P}_1}}

\newcommand{\bm}[1]{{\mathbf #1}}
\newcommand{\bin}[1]{{\underline{\mathbf #1}}_2}

\renewcommand{\Re}{\operatorname{Re}}

\renewcommand{\Im}{\operatorname{Im}}

\numberwithin{equation}{section}
\newcommand{\thmref}[1]{Theorem~\ref{#1}}
\newcommand{\propref}[1]{Proposition~\ref{#1}}
\newcommand{\secref}[1]{\S\ref{#1}}
\newcommand{\lemref}[1]{Lemma~\ref{#1}}

\newcommand{\figref}[1]{Figure~\ref{#1}}

\newcommand{\cW}{{\cal W}}

\newcommand{\cV}{{\cal V}}

\newcommand{\cP}{{\cal P}}
\newcommand{\cC}{{\cal C}}

\newcommand{\cL}{{\cal L}}
\newcommand{\cD}{{\cal D}}
\newcommand{\cE}{{\cal E}}
\newcommand{\cS}{{\cal S}}

\newcommand{\CC}{{\mathbb C}}
\renewcommand{\SS}{{\mathbb S}}
\newcommand{\RR}{{\mathbb R}}

\newcommand{\TT}{{\mathbb T}}
\newcommand{\ZZ}{{\mathbb Z}}
\newcommand{\NN}{{\mathbb N}}
\newcommand{\DD}{{\mathbb D}}
\newcommand{\HH}{{\mathbb H}}
\newcommand{\QQ}{{\mathbb Q}}

\newcommand{\bA}{{\mathbf A}}

\newcommand{\ignore}[1]{{}}

\newcommand{\cir}[1]{\overset{\circ}{#1}}

\newcommand{\finv}[1]{{g_{#1}}}
\newcommand{\finvloc}{{g}}

\newcommand{\Arg}{\operatorname{Arg}}

\newcommand{\tlphi}{\phi}

\begin{document}

\title[Parabolic renormalization]{The fixed point of the parabolic renormalization operator}

\author{Oscar Lanford III, Michael Yampolsky}

\maketitle

\section{Introduction}
In a recent paper \cite{IS}, H. Inou and M. Shishikura demonstrated that the successive parabolic renormalizations $\cP^n(f_0)$
of the quadratic polynomial $f_0(z)=z+z^2$ converge to an analytic map $f_*$ defined in a neighborhood of the origin,
which satisfies the fixed point equation
$$\cP(f_*)=f_*.$$
Conjecturally, the solution of the above functional equation is unique in a suitably restricted class of maps.
 
In this paper we present a class of analytic maps $\cla$ which have a maximal analytic extension to 
a Jordan domain,  satisfying the invariance property
$$\cP:\cla\to\cla.$$
The covering properties of a map $f\in\cla$ admit an explicit topological model.
We prove that the Inou-Shishikura fixed point $f_*$ of $\cP$ is contained in $\cla$, and conjecture that successive
renormalizations of any map $f\in\cla$ converge to $f_*$. 

The boundary of the maximal domain of analyticity of $f_*$ has a 
 a highly degenerate geometry.
It is this bad geometry that makes the study of $f_*$ so challenging. In contrast, consider the parabolic renormalization of 
critical circle maps with a parabolic fixed point on the circle, which both of the authors have studied \cite{Lan3,Ya2,EY}.
The corresponding renormalization fixed point also has a maximal analytic extension, whose covering properties are similar
to those of $f_*$. The geometry of its domain of analyticity, however, is rather tame, which permits both a numerical and
an analytic study.

 We present a numerical method for computing the Taylor's expansion of $f_*$ with a high accuracy. 
Our approach also allows us to compute the domain $\text{Dom}(f_*)$, and the reader will see the first computer-generated images of it.
Finally, we obtain a numerical estimate of the leading eigenvalue of $D\cP|_{f_*}$.

\section{Local dynamics of a parabolic germ}

\subsection{Fatou coordinates}
We briefly review the local dynamics an analytic function $f$ in the
vicinity of a parabolic fixed point at $0$: 
$$f(z)=e^{2\pi i p/q}z+O(z^2).$$
We consider first the case $q=1$, that is, $f'(0)=1$, and we write
\begin{equation}
\label{parabolic-normal-form-1}
f(z)=z+az^{n+1}+O(z^{n+2}),
\end{equation}
for some $n\in\NN$ and $a\neq 0$.
The integer $n+1$ can be recognized as the {\it multiplicity} of $0$
as the solution of $f(z)-z = 0$.

A complex number $\nu$ of modulus one is called an {\it attracting direction}
 if the product $a\nu^n<0$,
and a {\it repelling direction}  if the same product is positive.
The terminology has the following meaning: 

\begin{prop}
\label{attr dir}
Let $\{f^k(z)\}$ be an orbit in $\Dom(f) \setminus \{0\}$ which converges
to the parabolic fixed point $0$. Then the sequence of unit vectors
$f^{k}(z)/|f^{k}(z)|$ converges as $k \to \infty$ to one of the
attracting directions.
\end{prop}

We say in this case that the orbit converges to $p$ {\em from the
direction of $\nu$.}

If $f$ has a parabolic fixed point at $0$, it admits a {\em local
  inverse} there, by which we mean a function $g$, defined and
analytic on a neighborhood of $0$, so that $g(f(z)) = z = f(g(z))$ for
$z$ near enough to $0$. The germ at $0$ of a local inverse is unique,
but its domain of definition typically has to be chosen. A local
inverse also has a parabolic fixed point at $0$; attracting directions for
$f$ are repelling for the inverse and vice versa.

\begin{defn}
\label{defn-petal}
Let $\nu$ be an attracting direction for $f$.  An {\em attracting
  petal for $f$ (from the direction $\nu$)} is a Jordan domain $P$
with closure in $\Dom(f)$ such that:
\begin{enumerate}
\item $0 \in \partial P$
\item $f$ is injective on $P$;
\item  $f(\overline{P} \setminus \{ 0 \})\subset P$;
\item for any $z \in P$, the orbit $f^k(z)$ converges to $0$ from the
  direction $\nu$, and the convergence of $f^k$ to $0$ is uniform on $P$.
\item conversely, any orbit $f^k(z)$ which converges to $z$ from the
  direction $\nu$ is eventually in $P$.
\end{enumerate}
\end{defn}

\noindent
Similarly, $U$ is a {\it repelling petal} for $f$ if it is an
attracting petal for some local inverse $g$ of $f$.

Judiciously chosen petals can be organized into a {\it Leau-Fatou
  Flower} \index{Leau-Fatou Flower} at $0$:

\begin{thm}
\label{thm-flower}
There exists a collection of $n$ attracting petals $P^a_i$, and $n$ repelling petals $P^r_j$
such that the following holds. Any two repelling petals do not intersect, and every repelling petal
intersects exactly two attracting petals. Similar properties hold for attracting petals.
The union
$$(\cup P^a_i)\cup (\cup P^r_j)\cup\{0\}$$
forms an open simply-connected neighborhood of $0$.
\end{thm}

\noindent
The proof of this statement relies on some changes of
coordinates. First of all:  Every germ of the form
(\ref{parabolic-normal-form-1}) can be brought into the form 
\begin{equation}
\label{parabolic-normal-form-2}
f(z)=z+z^{n+1}+\alpha z^{2n+1}+O(z^{3n+1})
\end{equation}
in a suitable conformal local coordinate change at $0$.  In fact, a
straightforward induction shows the following:

\begin{prop}[cf. \cite{Mil} Problem 10-d,\cite{BE}]
\label{formal conjugacy}
For every germ of the form (\ref{parabolic-normal-form-1}) there
exists a unique $\alpha\in\CC$ such that for every $N\in\NN$ greater
that $2n+1$ there is a locally conformal change of coordinates $\psi$,
with $\psi(0)=0$, such that
$$\psi\circ f\circ \psi^{-1}(z)=z+z^{n+1}+\alpha z^{2n+1}+O(z^N).$$

\noindent
Further, there exists a formal power series $\Psi(z)=\sum_{k=1}^\infty p_kz^k$
which formally conjugates 
$$\Psi\circ f\circ \Psi^{-1}(z)=z+z^{n+1}+\alpha z^{2n+1}.$$
Thus, the number $\alpha\in\CC$ is a formal conjugacy invariant of $f$, and specifies its formal conjugacy class uniquely.

\end{prop} 

For the next few paragraphs, we will take $f$ to have the special form
(\ref{parabolic-normal-form-2}). The attracting directions are then
the $n$th roots of $-1$. We will describe some ways of constructiong
attracting petals for the attracting direction $\nu$; the adjustments
necessary to deal with repelling petals are routine. The reader is
reminded that (\ref{parabolic-normal-form-2}) is {\em not} the general
form of a mapping with a parabolic fixed point of order $n$; it has
been cleaned up by making a preliminary analytic change of coordinates
to eliminate some powers of $z$ in its Taylor series. 

The behavior of orbits of such an $f$ near $0$ is greatly clarified by
making the coordinate change
\[
w = \kappa(z) := -\frac{1}{n z^n}.
\]
We are considering a particular attracting direction, and we take
$\kappa$ to be defined on the sector between the two adjacent
repelling directions;  it opens up this sector to
the complex plane cut along the positive real axis. With its domain of
definition restricted in this way, $\kappa$ is bijective, and its
inverse is given by
\[
\kappa^{-1}(w) = \left( -\frac{1}{n w} \right)^{1/n}
\]
where the branch of the $n$-th root is the one cut along the positive
real axis and taking the value $\nu$ at $-1$.


Then
\begin{eqnarray*}
F(w)&:=& \kappa \circ f \circ \kappa^{-1}(w)\\
&=&-\frac{1}{n}\left( f\left(
    \left(-\frac{1}{nw}\right)^{\frac{1}{n}} \right)\right)^{-n}=\\ 
&=&-\frac{1}{n}\left(\left(-\frac{1}{nw}\right)^{\frac{1}{n}}+\left(-\frac{1}{nw}\right)^{\frac{n+1}{n}}+\alpha\left(-\frac{1}{nw}\right)^{\frac{2n+1}{n}}
  + O\left(\left(\frac{-1}{nw}\right)^{\frac{3n+1}{n}}\right)\right)^{-n}=\\
&=&w\cdot\left(1+\left(-\frac{1}{nw}\right)+\alpha\left(\frac{1}{nw}\right)^2+
  O\left(\frac{1}{w^3}\right)\right)^{-n}=\\
&=&w\cdot\left(1+\frac{1}{w}-\frac{\alpha}{n}\cdot\frac{1}{w^2}+\frac{n+1}{2n}\cdot\frac{1}{w^2}+O(\frac{1}{w^3})\right).
\end{eqnarray*}
We thus obtain:
\[
F(w)=w+1+\frac{A}{w}+O(1/w^2)\quad\text{ as $w\to\infty$}
\]
where
\[
A=\frac{1}{n}\left(\frac{n+1}{2}-\alpha\right).
\]
Selecting a right half-plane $H_r=\{\Re z>r\}$ for a sufficiently
large $r>0$, we have
$$\Re F(w)>\Re w+1/2,\text{ and hence }F(\overline{H_r})\subset H_r.$$
The domain $\kappa^{-1}(H_r)$ is then an attracting petal for the
attracting direction $\nu$. In the case of a simple parabolic point,
what we have just shown simplifies to the assertion that any disk of
sufficiently small radius tangent to the imaginary axis from the left
at $0$ is an attracting petal.

The petals just discussed -- pullbacks of half-planes under $\kappa$
-- have boundaries tangent at the origin to the directions $e^{\pm i
  \pi/(2n)} \nu$.  For many purposes -- such as the proof of Theorem
\ref{thm-flower} -- we will need petals with strictly larger opening
angle. There are many ways to construct such petals; here is one which
is convenient for our purposes. Let $\pi/2 < \alpha < \pi$, $R > 0$,
and let
\begin{equation}\label{eq:Delta-alpha-R}
\Delta(\alpha, R) := \{ w : - \alpha < \Arg(w - R ) < \alpha \}
\end{equation}
(i.e., $\Delta(\alpha, R)$ is the sector $\{ -\alpha < \Arg(w) <
\alpha \}$ translated right by $R$.)
From
\[
F(w) = w + 1 + O(1/w),
\]
there exists a $R_0 = R_0(\alpha)$ so that
\begin{equation}\label{eq:move-right}
\vert \Arg(F(w) - w) \vert < \pi - \alpha\quad\text{and}\quad
\Re(F(w)) > \Re(w) + 1/2
\end{equation}
for $\vert w \vert \geq R_0$.
If $R$ is large enough, the domain $ \Delta(\alpha, R)$
does not intersect the disk $\{ \vert w \vert \leq R_0 \}$, so
(\ref{eq:move-right}) holds for $w \in \Delta(\alpha, R)$. For
such $R$'s, by elementary geometric considerations,
\[
F \left(\overline{\Delta(\alpha, R)}\right) \subset \Delta(\alpha, R),
\]
and $F^n(w) \underset{n \to \infty}{\longrightarrow} \infty$ for all
$w \in \Delta(\alpha, R)$. Further any $\Delta(\alpha, R)$ contains a
right half-plane and hence eventually contains any $F$-orbit
converging to $\infty$. Finally, it can be verified that the sequence
of iterates $F^n$ converges uniformly to $\infty$ on
$\Delta(\alpha,R)$. We omit this verification; it uses simplified
versions of the ideas used in the proof of
\lemref{asym-expansion-3}. Thus, sets of the form $\kappa^{-1}\left(
\Delta(\alpha, R)\right)$ are attracting petals, symmetric about the
attracting direction under consideration, with tengents at the origin
in directions $e^{\pm i \alpha/n}\nu$. It will be useful to have a
general term for behavior for this: We will say that a petal $P$ with
attracting or repelling direction $\nu$ is {\em ample} if
\[
P \supset \{ z : \vert \Arg(z/\nu) \vert < \alpha/n, \vert z \vert < r \}
\]
for some $\alpha > \pi/2$ and sufficiently small $r$.  

\medskip\noindent
The dynamics inside a petal is described by the following:

\begin{prop}
\label{fatou cyl}
Let $P$ be an attracting petal for $f$. Then there exists
a conformal change of coordinates $\tlphi$ defined on $P$, conjugating
$f(z)$ to the unit translation $T:z\mapsto z+1$. 
\end{prop}

\begin{proof}
For a traditional proof, see e.g. \cite{Mil} \S 10. We cannot resist
giving a proof based on quasiconformal surgery, which
  probably originated in the work of Voronin \cite{Vor}. For
  definiteness, we discuss the case of an attracting petal with
  attracting direction $\nu$, and let 
\[
F(w)=w+1+O(1/w)
\]
be as above. Also as above, we select a right half-plane $H_r$. The main
step will be to prove the existence of $\phi$ for the special petal
$\kappa^{-1}H_r$, which we provisionally denote by $P_0$. The case of
a general petal will then follow by an easy extension argument.

As we know, $F(\overline{H_r})\subset H_r,$ so let us denote $S$ the
closed strip
$$S=\overline{H_r\setminus F(H_r)}.$$
Setting $\SS=\{\Re z\in [0,1]\}$, let $h$ be any diffeomorphism
$$h:S\to \SS,$$
which on the boundary of the strip conjugates $F$ to $T$:
$$T\circ h(w)=h\circ F(w),\text{ for all }w\text{ with }\Re w=r.$$
We will further require that the first partial
derivatives of $h$ and $h^{-1}$ be uniformly bounded in $S$.
Verifying the existence a diffeomorphism with these properties is an
elementary exercise which we leave to the reader.

The diffeomorphism $h$ defines a new complex structure
$\mu=h_*\sigma_0$ on $\SS$, which we extend to the left half-plane
$\{\Re z<0\}$ by
$$\mu|_w=(T^n)^*\mu\text{ for }T^n(w)\in\SS.$$
Gluing together $H_r$ with the standard complex structure and the
half-plane $\{ \Re z<1\}$ with structure $\mu$ via the homeomorphism
$h$ (which is now analytic), and using the Measurable Riemann Mapping
Theorem, we obtain a new Riemann surface $X$.  By the Uniformization
Theorem, $X$ is conformally isomorphic either to $\CC$ or to the
disk. By construction, $X$ is quasiconformally isomorphic to
$\CC$ and therefore cannot be conformally isomorphic to the disk. We
can specify a conformal isomorphism $\Phi:X\to \CC$ uniquely by
imposing normalization conditions $\Phi(0)=0$ and $\Phi(-1)=-1$.

The pair of maps $T|_{\{ \Re z<0\}}$ and $F|_{H_r}$ induces a conformal
automorphism of $X$, which we denote by $\tilde{F}$. Then $\Phi \circ
\tilde{F} \circ \Phi^{-1}$ is a conformal automorphism of $\CC$
with no fixed point. It is a standard fact that the only such
automorphisms are translations, and our choice of normalization for
$\Phi$ implies that
\begin{equation}
\label{phi-at-infty}
\Phi\circ \tilde{F}\circ \Phi^{-1}\equiv T.
\end{equation}
But $\tilde{F} = F$ on $H_r \subset X$, so we get
\[
\Phi \circ F = T \circ \Phi\quad\text{on $H_r$.}
\]
Moreover, the restriction of $\Phi$ to $H_r$ is analytic in the standard
sense.  Thus, we set $\phi := \Phi \circ \kappa$ on $P_0 :=
\kappa^{-1} H_r$ and obtain
\[
\phi \circ f = \phi + 1\quad\text{on $P_0$,}
\]
as desired.  Since $\Phi$ is a conformal  isomorphism from
$X$ to $\CC$, the map $\phi$ is univalent on $P_0$.

This proves the existence of $\phi$ on the particular petal
$P_0$. We provisionally denote the above $\phi$, which is
defined on $P_0$, by $\phi_0$. We define
\[
B^f_\nu := \{ z : f^n(z) \to 0 \quad\text{from the direction $\nu$}\}.
\]
If $z_0 \in B^f_\nu$, then $f^n(z_0)
\in P_0$ for sufficiently large $n$. If $f^{n_0}(z_0)
\in P_0$, then $(f^{n_0})^{-1}P_0$ is an open set containing $z_0$ and
contained in $B^f_\nu$, so $B^f_\nu$ is open. Since $P_0$ is mapped
into itself by $f$, and since
\[
\phi_0(f(z)) = \phi_0(z) + 1\quad\text{on $P_0$,}
\]
$\phi_0(f^n(z_0)) - n$ takes the same value for all $n$ for which
$f^n(z_0) \in P_0$. We denote this common value by $\phi(z_0)$, thus
obtaining a function $\phi$ defined on all of $B^f_\nu$ and extending
$\phi_0$ defined on $P_0 \subset B^f_\nu$. Tautologically, 
$$\phi(f(z))= \phi(z) + 1.$$
 If $f^{n_0}(z_0) \in P_0$,
then $f^{n_0}(z) \in P_0$ on a neighborhood of $z_0$, so $\phi(z) =
\phi_0(f^{n_0}(z))-n_0$ on this neighborhood, which shows that the
extended $\phi$ is analytic, but not necessarily univalent, on all of $B^f_\nu$. 

Now let $P$ be a general attracting petal with the same attracting
direction $\nu$. By definition of petal, $P \subset B^f_\nu$, so we
can restrict $\phi$ to $P$, thus obtaining an analytic function
satisfying
\[
\phi(f(z)) = \phi(z) +1 \quad\text{on $P$.}
\]
It remains to show that the restriction of $\phi$ to $P$ is
univalent. To see this, let $z_1$, $z_2$ be points of $P$ with
$\phi(z_1) = \phi(z_2)$. For sufficiently large $n$, $f^n(z_1)$
and $f^n(z_2)$ are both in $P_0$, so
\[
\phi_0(f^n(z_1) = \phi(z_1)+n = \phi(z_2)+n = \phi_0(f^n(z_2)).
\]
But, by construction, $\phi_0$ is univalent, so $f^n(z_1) =
f^n(z_2)$. The argument so far works for any pair $z_1$, $z_2$ in
$B^f_\nu$ with $\phi(z_1) = \phi(z_2)$. Now, however, we use that
facts that $z_1$ and $z_2$ are both in the petal $P$, which is mapped
into itself by $f$ and on which $f$ is univalent.
Hence, from $f^n(z_1) = f^n(z_2)$ it follows that $z_1 = z_2$, proving
univalence of $\phi$ on $P$.


\end{proof}

We note for future reference a simple result which was proved in the
course of the preceding argument:

\begin{prop}\label{th:phi-extended}
Let $\nu$ be an attracting direction for $f$, $P$ be an attracting
petal from the direction $\nu$, $\phi$ a univalent analytic function
defined on $P$ and satisfying the function equation
\[
\phi(f(z)) = \phi(z) +1.
\]
Then $\phi$ has a unique extension to $B^f_\nu$ satisfying this equation.
\end{prop}


We define {\em attracting Fatou coordinate} (for the attracting petal
$P$ with attracting direction $\nu$) to be a function $\phi_A$ defined,
analytic and univalent on $P$ and satisfying
\[
\phi_A(f(z)) = \phi_A(z) + 1\quad\text{on $P$.}
\]
As we have seen, such a function extends uniquely, via the above
functional equation, to all of $B_\nu^f$, and the extension restricts
to an attracting Fatou coordinate on any other petal with attracting
direction $\nu$. It is clear that, if $\phi_A$ is an attracting Fatou
coordinate then so is $\phi_A + c$ for any constant $c$. We will see
shortly that any two attracting Fatou coordinates differ only in this
way.

Any
attracting Fatou coordinate can be written in the form $\phi_A =
\Phi_A \circ \kappa$, where $\Phi_A$ satisfies the functional equation
\[
\Phi_A(F(w)) = \Phi_A(w) + 1\quad\text{(with $F = \kappa \circ f \circ
  \kappa^{-1}$ as above,)}
\]
on an appropriate $F$-invariant domain ``near infinity''. We will refer
to such $\Phi_A$'s as {\em Fatou coordinates at infinity.}

A {\it repelling Fatou coordinate} $\phi_R$ for $f$ means an attracting Fatou
coordinate for an analytic local inverse $g$ of $f$. If
$$f(z) = z + z^{n+1} + \ldots\text{, then }g(z) = z - z^{n+1} + \ldots,$$
which can be brought back into the standard form by conjugating with
$z \mapsto -z$ ($n$ odd) or $z \mapsto i \cdot z$ ($n$ even). The above considerations then
apply to define $\phi_R$, repelling petals, etc.  We note that:
\begin{itemize}
\item  a repelling Fatou coordinate $\phi_R$ satisfies {\em the same}
  functional equation
\[
\phi_R(f(z)) = \phi_R(z) + 1
\]
as does an attracting one, but the domains of definition are different, and
\item the image of a repelling petal by a  repelling Fatou coordinate is mapped into itself
  by the unit {\em left} translation $w \mapsto w-1$; 
 the image of an ample repelling petal under a repelling Fatou
  coordinate contains a {\em left} half-plane.
\end{itemize}
Again, it is useful to consider also {\em repelling Fatou coordinates
  at infinity}: If $\phi_R$ is a repelling Fatou coordinate, the
corresponding one at infinity is
\[
\Phi_R(w) = \phi_R(\kappa^{-1}(w))
\]
(but the appropriate branches of $\kappa^{-1}$ are different from the ones
in the attracting case).

Our next step is to prove a crude asymptotic formula for a Fatou
coordinate at infinity. It is advantageous here to deviate from what
we have been doing. We consider a mapping $f$ of the form
\[
f(z) = z + z^{n+1} + f_{n+2} z^{n+2} + \cdots ,
\]
i.e., we do not assume we have made a preliminary change of variable
to eliminate, e.g., the terms $z^j$ for $j$ between $n+1$ and $2
n+1$. We introduce $F = \kappa \circ f \circ \kappa^{-1}$ as before;
this time, the behavior of $F$ near infinity is
\[
F(w) = w + v(w)\quad\text{where $v(w) = 1 + v_1 w^{-1/n} + v_2
  w^{-2/n} + \ldots$};
\]
the series converges for sufficiently large $\vert w \vert$. Let
$\Phi_A$ be an attracting Fatou coordinate at infinity. By what
we have already proved: For any $\alpha < \pi$, there is an $R$ so
that $\Phi_A$ extends analytically to a univalent function on the set $\{ \vert
w \vert > R, - \alpha < \Arg(w) < \alpha \}$.

\begin{prop}
\label{crude asymptotics}
\[
\Phi_A'(w) \to 1\quad\text{and}\quad \Phi_A(w)/w \to 1
\]
uniformly as $w \to \infty$ in any sector $\{ -\alpha < \Arg(w)
< \alpha \}$ with $\alpha < \pi$. The same limits hold for $\Phi_R$,
but with $w \to \infty$ in the opposite sector $\{ -\alpha < \Arg(-w) <
\alpha \}$
\end{prop}

This proposition is a less-precise version of Lemma A.2.4 of
\cite{Sh}, and the argument we give is the first part of Shishikura's
proof of that lemma. Shishikura carries the analysis further and is
able to identify, in favorable cases, the first correction to the
indicated asymptotic behaviors. We do not give his full argument here, as
we will prove  \thmref{asym-expansion-2}, which gives more precise
information about the asymptotic behavior of Fatou coordinates.

\begin{proof}
Fix $\alpha$ with $\pi/2 < \alpha < \pi$, and let $\alpha < \alpha_1 <
\pi$. Take $R_1$ so that $\Phi_A$ is defined and univalent in
\[
S_1 := \{ \vert w \vert > R_1, - \alpha_1 < \Arg(w) < \alpha_1
\}
\]
and also so that $\vert v(w) - 1 \vert < 1/4$ on $S_1$. For $w_0
\in S_1$, denote by $\rho = \rho(w_0)$ the distance from $w_0$ to the
boundary of $S_1$. We will investigate limits as $\vert w_0 \vert \to
\infty$ in the strictly smaller sector $-\alpha < \Arg(w_0) < \alpha$;
then there is a constant $k > 0$ so that, asymptotically, $\rho(w_0)
\geq k \cdot \vert w_0 \vert$. In the following, we will frequently
assume silently that $\rho(w_0)$ is ``large enough''. We will also use
$C$ to denote a generic ``universal'' constant; different instances of
$C$ need not denote the same constant.

For the first step, we use the Koebe Distortion Theorem: If $\vert w -
w_0 \vert < \rho - 2$ -- so the disk of radius 2 about $w$ is in $S_!$
-- the mapping
\[
a \mapsto \frac{\Phi_A(w + a) - \Phi_A(w)}{\Phi_A'(w)}
\]
is analytic and univalent on $\{ \vert a \vert < 2 \}$ and has unit
derivative at the origin. A simple rescaling of the Koebe Theorem to
adapt it to the disk of radius 2 gives a universal constant $C > 0$ so that
\[
\left \vert \frac{\Phi_A(w + a) - \Phi_A(w)}{\Phi_A'(w)} \right\vert >
C^{-1}
\quad\text{for $ 3/4 < \vert a \vert < 5/4$}
\]
We insert $a = v(w)$ into this estimate, use $\vert 1- v(w) \vert < 1/4$
to ensure that $3/4 < \vert v(w) \vert < 5/4$, and use also the
functional equation
\[
\Phi_A(w+v(w)) = \Phi_A(F(w)) = \Phi_A(w)+1
\]
to get
\[
\vert \Phi'_A(w) \vert < C \quad\text{for $\vert w - w_0 \vert <
  (\rho - 2)$.}
\]
Applying the Cauchy estimates gives a bound
\[
\vert \Phi''_A(w) \vert \leq \frac{C}{\rho}\quad\text{for $\vert w -
  w_0 \vert < \rho/2$}
\]
(with a different $C$).

Next we apply Taylor's Formula with Integral Remainder to write
\[
\Phi_A(w_0 + a) = \Phi_A(w_0) + a \cdot \Phi'(w_0) + a^2 \cdot
\int_{s=0}^1 (1 -s) \Phi''(w_0 + s \cdot a) ds.
\]
Again, we set $a = v(w_0)$ and use $\Phi_A(w_0 + v(w_0)) = \Phi_A(w_0)
+ 1$ to get
\[
1 - v(w_0) \Phi_A'(w_0) = (v(w_0))^2 \int_{s=0}^1 (1 -s) \Phi''(w_0 +
s \cdot a) ds.
\]
Since the estimate $\vert \Phi_A''(w)\vert \leq C/\rho$ holds for all
$w$ appearing in the integral on the right, we get
\[
\vert 1 - v(w_0) \cdot \Phi_A'(w_0) \vert \leq C/\rho(w_0).
\]
We have already remarked that $\rho(w_0)\geq k \vert w_0 \vert$ as $w_0
\to \infty$ in the sector $\{ - \alpha < \Arg(w_0) < \alpha \}$
so
\[
\vert \Phi_A'(w_0) - v(w_0)^{-1} \vert = O(\vert w_0
\vert^{-1})\quad\text{in that sector.}
\]
This establishes the asserted convergence of $\Phi_A'$; the assertion
about $\Phi(w)/w$ follows by integration.

\end{proof}


Equipped with this information about the asymptotic behavior of Fatou
coordinates, we can now show that the image of a Fatou coordinate is
large enough. As usual, it suffices -- up to insertion of some
minus signs -- to consider the attracting case. Let
$\Phi_A$ be an attracting Fatou coordinate, and let $0 < \alpha <
\pi$. Then, for sufficiently large $R$, $\Phi_A$ extends analytically
to a univalent function on
\[
S := \{ w : \vert w \vert > R, -\alpha < \Arg(w) < \alpha \}
\]

\begin{prop}
\label{fatou coords cover}
Let $0 < \alpha_0 < \alpha$. Then, for sufficiently large $R_0$,
\[
\Phi_A(S) \supset S_0 := \{ \vert w \vert < R_0, -\alpha_0 < \Arg(w) <
\alpha_0 \}.
\]
\end{prop}

\begin{proof}
Let $w_0 \in S_0$; we want to investigate solutions to the equation
$\Phi_A(w) = w_0$ which we rewrite as
\[
 w = w_0 + w -
\Phi_A(w) =: \Psi_0(w)
\]
The idea is to apply the Contraction Mapping Principle to $\Psi_0$,
using the fact that $\Psi_0'(w) = 1 - \Phi_A'(w)$ which is small for
$w$ large. To do this, we need to find a domain mapped into itself by
$\Psi_0$ and on which $\Psi_0$ is contractive. Suppose we can find a
$\delta > 0$ so that
\begin{itemize}
\item $\vert \Psi_0'(w) \vert < 1/2$ for $\vert w - w_0 \vert < \delta$
\item $\vert \Psi_0(w_0) - w_0 \vert < \delta/2$
\end{itemize}
Then, for $\vert w - w_0 \vert < \delta$,
\[
\begin{split}
\vert \Psi_0(w) - w_0 \vert &\leq \vert \Psi_0(w) - \Psi_0(w_0)\vert +
\vert \Psi_0(w_0) - w_0 \vert\\
&< (1/2) \cdot \vert w - w_0 \vert + \delta/2 \leq \delta/2 + \delta/2 = \delta,
\end{split}\]
so the disk of radius $\delta$ about $w_0$ will be mapped to itself, and
$\Psi_0$ will have a unique fixed point in this disk.

We implement this strategy as follows: First of all we arrange, by
making $R$ larger if necessary, that $\vert \Phi_A'(w) - 1 \vert <
1/2$ on $S$. We write
\[
\epsilon := \sin(\alpha - \alpha_0),
\]
and we note that, by elementary geometry, $\vert w - w_0 \vert <
\epsilon \cdot w_0$ implies $\vert \Arg(w/w_0)\vert < \alpha -
\alpha_0$. If we further take
$R_0 \geq (1-\epsilon)^{-1} R$, then the disk of radius $\delta := \epsilon
\cdot \vert w_0 \vert$ about $w_0$ is contained in $S$, for any $w_0
\in S_0$. Recall that we have already arranged that $\vert
\Psi_0'\vert < 1/2$ on $S$. Finally, we apply \propref{crude
  asymptotics} to see that, by taking
$R_0$ large enough we can arrange that
\[
\vert \Phi_A(w_0) - w_0 \vert < (1/2)\cdot \epsilon\cdot \vert w_0
\vert\quad\text{for all $w_0 \in S_0$.}
\]
All the element for the above contraction argument are now in place,
and we can conclude that, for every $w_0 \in S_0$, there is a unique $w$
with $\Phi_A(w) = w_0$ in $\{ \vert w - w_0 \vert < \epsilon \vert w_0
\vert \} \subset S$.  This proves the assertion $\Phi_A(S) \supset S_0$.

\end{proof} 

It follows from this proposition that:

\begin{prop}
\label{covershalfplane}
The image under $\phi_A$ of any
ample petal of $f$ contains  a right half-plane.
\end{prop}

\noindent
A similar assumption holds for ample repelling petals, but with the image under $\phi_R$ covering a left
half-plane.

Let $P$ be an attracting petal. We define a relation on $P$ by $z_1
\sim z_2$ if $z_1$ and $z_2$ are on the same orbit, that is, if either $z_2  =
f^j(z_1)$ or $z_1 = f^j(z_2)$ (with $j \geq 0$.) It is easy to check
that this is an equivalence relation. 

Consider the quotient $P/_\sim$.  The canonical projection $\pi: P \to
P/_\sim$ is locally injective, and it is straightforward to verify
that there is a unique way to give $P/_\sim$ a Riemann surface
structure in such a way as to make $\pi$ analytic and therefore a
local conformal isomorphism.

Let $P_1$ be another attracting petal contained in $P$. Since the orbit of every point $z\in P$ eventually lands in $P_1$,
the inclusion $P_1\hookrightarrow P$ induces a conformal homeomorphism 
$$P_1/\sim\underset{\simeq}{\longrightarrow} P/_\sim.$$
Now if $P_2$ is any attracting petal with the same attracting direction as
$P$, the intersection $P\cap P_2$
is also a petal. Hence
$$P_2/_\sim\simeq P/_\sim\simeq (P\cap P_2)/_\sim.$$
Thus, the quotient $P_\sim$ does not depend on the choice of the petal $P$ but only on the
choice of the attracting direction $\nu$ corresponding to $P$. We will write 
$$P/_\sim\equiv \cC_A^\nu.$$
We will omit $\nu$ from the notation when the choice of the attracting direction is clear from the context
(for instance, when there is only one attracting direction).

Let $\phi_A$ be an attracting Fatou coordinate defined on some petal
$P$. It is easy to verify, using the injectivity of $\phi_A$ on $P$,
that two points $z_1$ and $z_2$ are equivalent if and only if
$\phi_A(z_1) - \phi_A(z_2) \in \ZZ$. Hence, $\phi_A$ defines by
passage to quotients an injective mapping $\tl\phi_A$ from $\cC_A^\nu$
to $\CC/\ZZ$. If we take $P$ to be an ample petal, then it follows
from \propref{covershalfplane} that $\tl\phi_A$ takes on all values in
$\CC/\ZZ$. Thus:
\begin{prop}
\label{conformal isomorphism of cylinders}
The map $\tl\phi_A$ is a conformal isomorphism from the Riemann surface
$\cC_A^\nu$ to the $\CC/\ZZ$.
\end{prop}

In light of the preceding proposition, we will call $\cC_A^\nu$ the
{\it attracting cylinder} corresponding to the direction $\nu$.

\ignore{
The quotient $P/\sim$ is called
the {\em attracting cylinder} (for reasons which will become apparent
shortly); we will denote it by $\cC_A$.

Although the construction of $\cC_A$ starts from the choice of a petal
$P$, the dependence on $P$ is inconsequential: Consider first two
petals $P_1$, $P_2$ with $P_2 \subset P_1$. Then any orbit in $P_1$ is
eventually in $P_2$, so any $z_1$ in $P_1$ is equivalent (in the sense
of $P_1$) to an element of $P_2$. On the other hand, two elements of
$P_2$ are equivalent in the sense of $P_2$ if and only if they are
equivalent in the sense of $P_1$. In short: The inclusion mapping $P_2
\hookrightarrow P_1$ induces a bijection $P_2/\sim_2 \to
P_1/\sim_1$. It is also easy to see that this bijection is an
isomorphism of Riemann surface structures. In short: The attracting
cylinder built on $P_1$ is canonically isomorphic to the one built on
$P_2$. For $P_1$, $P_2$ general attracting petals, $P_1 \cap P_2$ is
again an attracting petal, so
\[
(P_1/\sim_1) \approx (P_1 \cap P_2/ \sim_{1\cap 2}) \approx (P_2/\sim_2).
\]
Loosely: The attracting cylinder is independent of the choice of
petal. In fact, in the same spirit: The attracting cylinder depends
only on the {\em germ} of $f$ at $0$. (On the other hand: It does, I
believe, depend on the choice of an {\em attracting direction.})

Let $\phi_A$ be an attracting Fatou coordinate, which we take to be
defined on $P$. If $z_1 \sim z_2$, say $z_2 = f^j(z_1)$, then
$\phi_A(z_2) = \phi_A(z_1) + j$. Conversely, if $\phi_A(z_2) =
\phi_A(z_1) + j$, then $\phi_A(z_2) = \phi_A(f^j(z_1)$, so, by
univalence of $\phi_A$ on petals, $z_2 = f^j(z_1)$. In short: $z_1
\sim z_2$ if and only if $\phi_A(z_1) - \phi_A(z_2) \in \ZZ$. Thus,
$\phi_A$ defines by passage to quotients an injective mapping
$\tl\phi_A: \cC_A \to \CC/\ZZ$. If we form $\cC_A$ as a quotient of an
ample petal, \propref{fatou cyl} says that $\phi_A$ takes on all
complex values mod $\ZZ$, i.e., that $\tl\phi_A: \cC_A \to \CC/\ZZ$ is
surjective and hence bijective. It is easy to check that $\tl\phi_A$
is also analytic, so

\begin{prop}\label{conformal isomorphism of cylinders}
  The map $\tlphi_A$ induces a conformal isomorphism $\tl\phi_A$ from
  the space of orbits $\cC_A$ defined above to the cylinder
  $\CC/\ZZ$.
\end{prop}

Alternatively, we can skip the direct verification that there is a
Riemann surface structure on $\cC_A$ making the canonical projection
analytic and simply transport the Riemann surface structure on
$\CC/\ZZ$ to the quotient set $\cC_A$ by the bijection $\tl\phi_A$.

}
The {\em repelling cylinder} $\cC^\nu_R$ for $f$ is the attracting
cylinder for a local inverse of $f$, fixing the origin.

If $P$ is an attracting petal, we will call the half-open domain 
$$C_A=P\setminus fP$$
a {\em fundamental attracting crescent}, the name reflecting its
shape. A fundamental repelling crescent means a fundamental attracting
crescent for a local inverse of $f$.

\begin{prop}\label{th:fundamental-crescent}
For any attracting petal $P$ in the direction $\nu$, the fundamental crescent $C_A$ projects bijectively onto the
attracting cylinder $\cC^\nu_A$. More concretely:
Every point of $P$ lies on a forward orbit starting in $C_A$, and
distinct points of $C_A$ have disjoint forward orbits.
\end{prop}

\begin{proof}
From the requirement that $f^n$ converges uniformly to 0 on $P$ --
condition (4) in of our definition of petal --
it follows that no point of $P$ admits an arbitrarily long backward orbit in
$P$. Thus, every point of $P$ lies on the forward orbit starting 
outside $P$; the first point on this orbit inside $P$ is in $C_A$. Let
$z_1$, $z_2$ be points of $C_A$ whose forward orbits intersect. From
the injectivity of $f$ on $P$, we have -- possibly after interchanging
$z_1$ and $z_2$ -- $f^n(z_1) = z_2$ for same $n \geq 0$. If $n > 0$,
then $z_2 = f^n(z_1) \in fP$, contradicting $z_2 \in C_A$, so the only
possibility left is $n=0$, i.e., $z_1 = z_2$.
\end{proof}

We note the following standard fact:
\begin{prop}
\label{liouville}
Let $h:\CC^*\to\CC^*$ be an injective holomorphic map. Then either $h(z)=cz$ or $h(z)=c/z$ for a non-zero constant $c$.
\end{prop}
\begin{proof}
  Such an $h$ is in particular an analytic function with isolated
  singularities at $0$ and $\infty$ (and nowhere else.) By
  injectivity, neither singularity can be essential, so $h$ extends
  to a meromorphic mapping of the Riemann sphere to itself, i.e. to a rational
  function. Injectivity on the sphere with two points deleted implies
  that this rational function has degree one, i.e., is a M\"obius
  transformation. In particular, the extended function maps the sphere
  bijectively to itself, so either $h(0)=0$, in which case
  $h(\infty) = \infty$, or $h(0) = \infty$. In the
  first case, $h(z)/z$ is bounded at $\infty$ and has a removable
  singularity at $0$, so, by Liouville's Theorem, $h(z)/z = c$. In the case
  $h(0)=\infty$, applying the above to $z \mapsto h(1/z)$ gives $h(z) = c/z$.
  
\end{proof}

A corollary of the above result is a uniqueness statement for Fatou coordinates:

\begin{prop}
\label{uniqueness-fatou}
Let $P$ be an attracting petal of $f$, and let $\phi_1$ and $\phi_2$
be attracting Fatou coordinates on $P$, i.e., univalent analytic
functions satisfying $\phi_i(f(z)) = \phi_i(z) + 1$. Then $\phi_2(z) -
\phi_1(z)$ is constant on $P$.
\end{prop}
\begin{proof}
By \propref{conformal isomorphism of cylinders}, $\phi_1$ and $\phi_2$
both induce conformal isomorphisms from $\cC_A$ to $\CC/\ZZ$. It will be
more convenient to work with with the punctured plane $\CC^*$ instead of
$\CC/\ZZ$. The function
\[
\ixp(z) := \exp(2 \pi i z)
\]
induces -- again by passage to quotients -- a conformal isomorphism
from $\CC/\ZZ$ to $C^*$, so $\ixp \circ \phi_1$ and $\ixp \circ
\phi_2$ both induce conformal isomorphism $\cC_A \to \CC^*$. Hence,
the prescription
\[
h: \exp(2 \pi i \phi_1(z))) \mapsto  \exp(2 \pi i \phi_2(z))\quad\text{for all
  $z \in P$}
\]
defines a conformal isomorphism $h: \CC^* \to \CC^*$.
By \propref{liouville}, there are two possibilities:
\begin{enumerate}
\item there is a non-zero constant -- which we write as $\exp(2 \pi i
  c)$ -- so that
\[
\exp(2 \pi i c )\cdot \exp(2 \pi i \phi_1(z) = \exp(2 \pi i
\phi_2(z))\quad\text{for all $z \in P$}
\]
or
\item there is a constant $c$ so that
\[
\exp(2 \pi i c )\cdot \exp(- 2 \pi i \phi_1(z)) = \exp(2 \pi i
\phi_2(z))\quad\text{for all $z \in P$}
\]
\end{enumerate}
In the first case,
\[
\phi_2(z) - \phi_1(z) - c \in \ZZ\quad\text{for all $z \in P$.}
\]
But the expression on the left is continuous, and an integer-valued
continuous function on a connected set must be constant, so
\[
\phi_2(z) = \phi_1(z) + c + n_0\quad\text{for all $z \in P$, for some
  $n_0 \in \ZZ$}
\]
which is what we wanted to prove.

In the second case, similarly,
\[
\phi_2(z) = - \phi_1(z) + c + n_0\quad\text{for all $z \in P$,}
\]
and this contradicts
\[
\phi_1(f(z)) - \phi_1(z) = 1 = \phi_2(f(z)) - \phi_2(z)\quad\text{for
  $z \in P$,}
\]
so this case is excluded.
\end{proof}

The critical values of Fatou coordinates are simply related to those
of $f$:

\begin{prop}\label{th:fatou-crit}
Let $\phi_A$ be an attracting Fatou coordinate defined on
$B^f_\nu$. Then critical values of $\phi_A$  all have the form
\[
\phi_A(v) - n\quad\text{ with $v$ a critical value of $f$ in $B^f_\nu$
  and $n
  \geq 1$}
\]
If $f: B^f_\nu \to B^f_\nu$ is surjective, all numbers of this form
are critical values.
\end{prop}

\begin{proof}
Iterating the functional equation for $\phi_A$, differentiating, and
applying the chain rule gives
\[
\phi_A'(z) = \frac{d}{dz}\phi_A(f^n(z)) = \phi_A'(f^n(z))
\prod_{j=0}^{n-1} f'(f^j(z)).
\]
For given $z$ and sufficiently large $n$, $f^n(z)$ is in an attracting
petal, which implies $\phi_A'(f^n(z)) \neq
0$. Thus: $z$ is a critical point of $\phi_A$ if and only if there
exists $j\geq 0$ such that $f^j(z)$ is a critical point of $f$. For
such a $j$, $f^{j+1}(z) = v$ is a critical value of $f$. Since
\[
\phi_A(z) = \phi_A(f^{j+1}(z)) - (j+1) = \phi_A(v) - (j+1),
\]
the assertion follows.
\end{proof}


For completeness, let us note how the situation changes if $f'(0)$ is
a $q$-th root of unity $e^{2\pi i p/q}$ with $q\neq 1$.  A fixed petal
for the iterate $f^q$ corresponds to a cycle of $q$ petals for $f$.
It thus follows that $q$ divides the number $n$ of
attracting/repelling directions of $0$ as a fixed point of $f^q$.

\subsection{Asymptotic expansion of a Fatou coordinate at infinity}


We will now specialize to the case $q=1$ and $n=1$. By rescaling
we can then bring the coefficient of $z^2$ to $1$,
and the normal form (\ref{parabolic-normal-form-2}) becomes
\begin{equation}
\label{parabolic-normal-form-3}
f(z)=z+z^2+\alpha z^3+O(z^4).
\end{equation}
We will say in this case that $0$ is a {\em simple} parabolic fixed
point of $f$. 
There is one attracting direction  ($-1$) and one repelling
direction ($+1$). 

If the domain of definition $\Dom(f)\ni 0$ is fixed, we let $B^f\subset\Dom(f)$, as before, to denote the basin of the
parabolic point at the origin. The {\it immediate basin} of $0$, which we denote $B_0^f$, is the connected
component of $B^f$ which contains an attracting petal.

The change of variables $\kappa$ moving the parabolic point to
$\infty$ becomes simply
\[
\kappa(z) = - \frac{1}{z},\quad \kappa^{-1}(w) = - \frac{1}{w},
\]
and we have
\[
F(w) = -\left(f(-\frac{1}{w})\right)^{-1} = w + 1 + \frac{A}{w} + O(w^{-2}),
\quad\text{with $A = 1 - \alpha$,}
\]
and $F(w) - w$ is analytic at $\infty$.

We showed earlier (\propref{crude asymptotics}) that any attracting
Fatou coordinate at infinity $\Phi_A$ for such an $f$ satisfies
\[
\Phi_A(w) = w + o(w)\quad\text{as $w \to \infty$ appropriately}
\]
We will prove shortly a much more precise result -- an asymptotic
expansion giving $\Phi_A$ up to corrections of order $w^{-n}$ for any
$n$. Before we do this, we investigate formal solutions to the
functional equation
\[
\Phi(F(w)) = \Phi(w) + 1
\]
satisfied by both attracting and repelling Fatou coordinates.

\begin{prop}
\label{formal-fatou}
There is a unique sequence $b_1, b_2, \ldots$ of complex coefficients
such that
\begin{equation}
\label{formal-fatou-coord-1}
\Phi_{\text{ps}} (w) = w - A \log(w) + \sum_{j=1}^\infty b_j
w^{-j}
\end{equation}
satisfies
\begin{equation}
\label{formal-fatou-coord-2}
\Phi_{\text{ps}} \circ F(w) = \Phi_{\text{ps}}(w) +1\quad\text{in the sense
  of formal power series.}
\end{equation}
Furthermore, if we set
\begin{equation}
\label{def-phi-n}
\Phi_n(w) := w - A \log w + \sum_{j+1}^n b_j w^{-j}
\end{equation}
then
\begin{equation}
\label{estimate-phi-n}
\Phi_n(F(w)) - \Phi_n(w) - 1 = O(w^{-(n+2)})
\end{equation}
\end{prop}

\noindent 
The logarithm appearing in (\ref{formal-fatou-coord-1}) -- and all other
logarithms in this 
section -- are to be understood as the {\em principal branch,} i.e.,
the branch with a cut along the negative real axis and real values on
the positive axis.  Because of the logarithmic term,
$\Phi_{\text{ps}}$ as written is not exactly a formal power series in
$w^{-1}$. To work around this, we rewrite the equation
$\Phi_{\text{ps}} \circ F = \Phi_{\text{ps}} + 1$ formally as
\begin{equation}
\label{eq:22}
F(w) - w - 1 - A \log(F(w)/w) = \sum_{j=1}^\infty b_j 
  \left( F(w)^{-j} - w^{-j} \right)
\end{equation}
Since $F(w)/w$ is analytic at $\infty$ and takes the value $1$
there, $\log(F(w)/w)$ is analytic at $\infty$ and vanishes
there. Furthermore, the formal identity
\[
\log(F(w)) - \log(w) = \log(F(w)/w)
\]
holds literally on $\{ - \alpha < \Arg(w) < \alpha, \vert w \vert > R \}$
for sufficiently large $R$, for any $\alpha < \pi$. It is equation
(\ref{eq:22}) which we really solve.

The left-hand side of (\ref{eq:22}) is analytic at $\infty$, and vanishes to
second order there:
\[
F(w)-w-1 = -A w^{-1} + O(w^{-2})\quad\text{and}\quad \log(F(w)/w) =
w^{-1} + O(w^{-2})
\]
Further, $F(w)^{-j} - w^{-j}$ is analytic at infinity and
vanishes to order $j+1$ there, so
\[
\sum_{j=1}^\infty b_j   \left( F(w)^{-j} - w^{-j} \right)
\]
is indeed a formal power series in $w^{-1}$ which begins with a term
in $w^{-2}$. Furthermore, the coefficient of $w^{-j-1}$ in the
expression on the right in (\ref{eq:22}) can be
written as 
\[
- j b_j  + \text{a function of $b_1$, \ldots, $b_{j-1}$}
\]
Thus, since the left-hand side of (\ref{eq:22}) is known, the $b_j$'s can be
determined successively, and, by induction on $j$, they are uniquely
determined. The assertion about the order of the error term
$\Phi_n(F(w)) - \Phi_n(w) - 1$ also follows, since
\[
\sum_{j=n+1}^\infty b_j \left( F(w)^{-j} - w^{-j} \right)
\]
begins with a term in $w^{-(n+2)}$

\begin{thm}
\label{asym-expansion-2}
Let $b_1$, $b_2$, \dots be as in \propref{formal-fatou}, and let
$\Phi_A$ be an attracting Fatou coordinate at infinity. Then, for any $n$,
\begin{equation}\label{eq:21}
\Phi_A(w) = w - A \log w + C_A + \sum_{j=1}^n b_j w^{-j} + O(\vert w
\vert^{-(n+1)})
\end{equation}
uniformly as $w \to \infty$ in any sector $-\alpha < \Arg(w) < \alpha$
with $\alpha < \pi$
\end{thm}

\noindent
We collect the main estimates needed for the proof of
\thmref{asym-expansion-2} in the following lemma:

\begin{lem}
\label{asym-expansion-3}
\[
\sum_{j=0}^\infty |F^j(w)|^{-m}= O(\vert w \vert^{-(m-1)})
\tag{I}\]
and
\[
\left\vert (F^j)'(w)\right\vert \quad\text{is bounded uniformly in $j$}\tag{II}
\]
both estimates holding uniformly for $w \to \infty$ in any sector $\{
-\alpha < \Arg(w) < \alpha \}$ with $\alpha < \pi$.
\end{lem}

\begin{pf}
We fix an $\alpha < \pi$, and we choose
an $\alpha_1$ with $\alpha < \alpha_1 < \pi$; it saves trouble later
if we also require that $\pi - \alpha_1 <
\pi/6$. Next we fix an $R_0$ large enough so that
\begin{equation}\label{eq:20}
\vert F(w) - (w + 1) \vert < \sin(\pi - \alpha_1)
\end{equation}
holds for $\vert w \vert > R_0$;
then an $R_1$ so that the translated sector
\[
\Delta(\alpha_1, R_1) := \{ w : -\alpha_1 < \Arg(w - R_1) < \alpha_1 \}
\]
does not intersect the disk of radius $R_0$ about $0$. Then
(\ref{eq:20}) holds on $\Delta(\alpha_1, R_1)$, so $F$ maps
$\Delta(\alpha_1, R_1)$ to itself; also, since $\pi - \alpha_1 <
\pi/6$, we also have -- again from (\ref{eq:20}) -- 
\begin{equation}
\label{less-than-half}
\vert F(w) - w - 1 \vert < \frac{1}{2}\quad\text{on $\Delta(\alpha_1,
  R_1)$.}
\end{equation}
from which it follows that
\begin{equation}
\label{greater-than-half}
\Re(F(w) \geq \Re(w)+1/2\quad\text{for $w \in \Delta(\alpha_1, R_1)$.}
\end{equation}
and also
\begin{equation}
\label{angle-less-than}
-(\pi - \alpha_1) < \Arg(F(w) - w) < (\pi - \alpha_1)
\end{equation}
We will prove estimates (I) and (II) for $w \to \infty$ in $\Delta(\alpha_1,
R_1) \cap \{ - \alpha < \Arg(w) < \alpha\} $; this does what we want
since every $w$ with $-\alpha < \Arg(w) < \alpha$ and sufficiently
large modulus is in $ \Delta(\alpha_1, R_1)$.

Changing notation slightly: We want then to estimate
\[
\sum_{j = 0}^\infty \vert F^j(w_0)\vert^{-m},
\]
for large $w_0$, given that
\[
-\alpha_1 < \Arg(w_0 - R) < \alpha_1 \quad\text{and}\quad - \alpha <
\Arg(w_0) < \alpha
\]
where -- crucially -- $\alpha_1 > \alpha$.
We write
\[
F^j(w_0) =: w_j =: u_j + i v_j
\]
In the calculation which follows, we adopt the convention that $K$
denotes some constant depending only on $\alpha$,
$\alpha_1$ and $m$. Different instances of $K$ are not necessarily the same
constant. All inequalities involving $w_0$ are only asserted to hold
for $\vert w_0 \vert$ large enough.

We first treat the case $-\pi/4 \leq \Arg(w_0) \leq \pi/4$, i.e.,
$\vert v_0 \vert \leq u_0$. Then, by (\ref{greater-than-half}),
\[
\vert w_j \vert \geq u_j \geq u_0 + j/2,
\]
so
\[
\sum_{j=0}^\infty \vert w_j \vert^{-m} \leq \sum_{j=0}^\infty (u_0 +
j/2)^{-m} \leq K u_0^{-(m-1)} \leq K \vert w_0 \vert^{-(m-1)};
\]
in the last step, we used $\vert v_0 \vert \leq u_0$ to estimate
$u_0 \geq \vert w_0 \vert/\sqrt{2}$.
This proves the desired estimate in this case.

There remain the possibilities $\pi/4 < \Arg(w_0) < \alpha$ and $-\alpha <
\Arg(w_0) < - \pi/4$. The estimates in the two cases are essentially
the same; for definiteness we assume that the first holds, i.e., that
$w_0$ is in the upper half-plane. By (\ref{angle-less-than}) the $w_j$
are all contained in the translated sector $\{ -(\pi - \alpha_1) < w -
w_0 < +(\pi -
\alpha_1) \}$. This sector intersects the diagonal line $\{ \Re(w) =
\Im(w) \}$ in a segment $(\hat{w}_{-}, \hat{w}_{+})$ (labeled so that
$\vert\hat{w}_{-}\vert  <  \vert\hat{w}_{+}\vert$.) It is easy to
show, either by elementary geometry or by writing explicit formulas, that
\[
\vert \hat{w}_{-}\vert \geq K^{-1} \vert w_0 \vert\quad\text{and}\quad
\vert \hat{w}_{+}\vert \leq K \vert w_0 \vert.
\]
Since $u_{j+1} \geq u_j + 1/2$, $w_j$, which start out above the
diagonal line, will get and stay below it after finitely many
steps. Let $j_0$ be the first index for which $w_j$ is below the
diagonal. Using the above upper bound on $\vert \hat{w}_{+} \vert$ we
get a bound
\[
j_0 \leq K \vert w_0 \vert
\]
The line segment from $w_{j_0 - 1}$ to $w_{j_0}$
intersects the diagonal line between $\hat{w}_{-}$ and
$\hat{w}_+$. Since 
\[
\vert w_{j_0} - w_{j_0-1} \vert = \vert F(w_{j_0-1}) - w_{j_0 - 1}
\vert < 3/2
\]
we get
\[
\vert w_{j_0} \vert \geq \vert \hat{w}_{-} \vert - 3/2 \geq K^{-1}
\vert w_0 \vert - 3/2 \geq K^{-1} \vert w_0 \vert
\]
By the first case treated
\begin{equation}\label{last-part}
\sum_{j=j_0}^\infty \vert w_j \vert^{-m} \leq K \vert w_{j_0}
\vert^{-(m-1)} \vert \leq K \vert w_0 \vert^{-(m-1)}
\end{equation}

All the $w_j$ lie above (or on) the line through $w_0$ with
direction $\alpha_1 - \pi$ (and the origin lies below this line.) Using
$\Arg(w_0) < \alpha < \alpha_1$, the distance from this line to the
origin satisfies a lower bound $K^{-1} \vert w_0 \vert$. Hence each
$\vert w_j \vert \geq K^{-1} \vert w_0 \vert$. Since $j_0 \leq K \vert
w_0 \vert$, 
\[
\sum_{j=0}^{j_0 - 1 } \vert w_j \vert^{-m} \leq j_0 \cdot \left(
  K^{-1} \vert w_0 \vert \right)^{-m} \leq (K \vert w_0 \vert) \cdot
(K^m \vert w_0 \vert^{-m}) \leq K \vert w_0 \vert^{-(m-1)}
\]
Combining this estimate on the sum of the first $j_0$ terms with the
estimate (\ref{last-part}) on the sum of the rest gives
\[
\sum_{j=0}^\infty \vert F^j(w_0)\vert^{-m} \leq K \vert w_0
\vert^{-(m-1)}
\]
(for sufficiently large $\vert w_0 \vert$), so (I) is established.
(II) follows easily from (I) together with the estimate
$$F'(w)=1+O(|w|^{-2}),$$
the chain rule, and standard manipulations for reducing estimates on
products to estimates on sums.

\end{pf}

\begin{proof}[Proof of \thmref{asym-expansion-2}]
Let $\pi/2 < \alpha <
\pi$. By an argument already used several times, we can choose $R$
sufficiently large so that
\[
\vert F(w) - w - 1 \vert < \sin(\pi - \alpha)\text{ and }
\Re(F(w) >  \Re{w} + 1/2
\]
for all $w \in \Delta(\alpha, R) := \{ -\alpha < \Arg(w-R) < \alpha \}$
As usual, it follows from the first of these inequalities that $F
(\overline{\Delta(\alpha, R)}) \subset \Delta(\alpha, R)$. We are
going to prove
\[
\Phi_A(w) = \Phi_n(w) + O(\vert w
\vert^{-(n+1)})\quad\text{as $w \to \infty$ in $\Delta(\alpha, R)$}
\]
(where $\Phi_n$ is defined by (\ref{formal-fatou-coord-2})); this
assertion for all $\alpha$ implies the assertion of the theorem for
all $\alpha$.

We set
\[ 
c_n(w) :=\Phi_A(w)-\Phi_n(w)\quad\text{and}\quad
u_n(w):= \Phi_n\circ F(w)-\Phi_n(w)-1.
\]
By \propref{formal-fatou},
$u_n$ is analytic at infinity with
\begin{equation}
\label{estimate-phi-n-1}
u_n(w)=d_{n+2}w^{-(n+2)}+\cdots.
\end{equation}
A simple calculation gives
\[
c_n\circ F(w) = c_n(w)- u_n(w);
\]
iterating gives
\[
c_n\circ F^k(w) = c_n(w) - \displaystyle\sum_{j=0}^{k-1}u_n(F^j(w));
\]
reorganizing and differentiating gives
\begin{equation}
\label{eq-as1}
c_n'(w)= \sum_{j=0}^{k-1}u_n'(F^j(w))\cdot(F^j)'(w)+ c_n'(F^k(w))\cdot
(F^k)'(w).
\end{equation}
By differentiating (\ref{estimate-phi-n-1})
\[
u_n'(w) = O(|w|^{-(n+3)})\text{ as }w\to\infty,
\]
so, by \lemref{asym-expansion-3},
\[
\sum_{j=0}^\infty \left\vert u_n'(F^j(w))\cdot(F^j)'(w) \right\vert <
\infty
\]
By \propref{crude asymptotics}
\[
\Phi_A'(w) \to 1\quad\text{as $\Re(w) \to +\infty$},
\]
and the same is true for $\Phi_n$ by an elementary calculation; hence
\[
c'_n(w) \to 0 \quad\text{as $\Re(w) \to +\infty$}.
\]
Thus, we can let $k \to \infty$ in (\ref{eq-as1}) to get
\[
c_n'(w) = \sum_{j=0}^{\infty}u_n'(F^j(w))\cdot(F^j)'(w)
\]
Applying both parts of \lemref{asym-expansion-3} to this representation,
\begin{equation}
\label{eq-as3}
|c_n'(w)|\leq \const|w|^{-(n+2)}
\end{equation}
for all $w \in \Delta(\alpha, R)$.  It follows from this
estimate that the limit
\[
\lim_{u \to \infty} c_n(u+iv)
\] 
exists and is independent of $v$. We denote this limit
by $C_A$. Then
\[
c_n(u+iv)=C_A -\int_{u}^{\infty}c_n'(\sigma+iv)d\sigma,
\]
so by integrating (\ref{eq-as3}) we get
\[
\vert c_n(w) - C_A \vert \leq \const \vert w \vert^{-(n+1)}
\]
for all $w \in \Delta(\alpha, R)$, which is what we set out to prove.

\end{proof}

\ignore{

It is not difficult to get an asymptotic development
\begin{equation}
\label{asym-fatou-1}
\Phi(w)=w-A\log w +\text{const}+O\left(\frac{1}{|w|}\right)
\end{equation}
for $\Phi$ in (\ref{phi-at-infty}),
with an appropriate choice of the branch of the logarithm (see \cite{Sh}).
Indeed, denote $$v(w)\equiv F(w)-w=1+\frac{A}{w}+O\left(\frac{1}{|w|^2}\right).$$
Using the Taylor's formula with a remainder, we have
$$1=\Phi(w+v(w))-\Phi(w)=v(w)\Phi'(w)+v(w)^2\int_0^1(1-t)\Phi''(w+v(w)t)dt.$$
Using Koebe Distortion theorem to bound $\Phi'$ and $\Phi''$, we 
get an asymptotic estimate
$$\left|\Phi'(w)-\frac{1}{v(w)}\right|<\frac{\text{const}}{|w|^2},$$
from which (\ref{asym-fatou-1}) follows by a simple integration.

We will now show that the remainder term of (\ref{asym-fatou-1})
admits an asymptotic expansion valid for $w \to \infty$ in any sector of
the form $\{ -\alpha < \Arg(w) < \alpha \}$, $\alpha < \pi$. 
Denote by
$$\Phi_{\mathbf b^n}(w)=w+A\log w+\displaystyle{\sum}_{k=1}^{n-1}b_k w^{-k},\text{ where }\mathbf b^n=(b_0,b_1,\ldots,b_{n-1})\in\CC^n.$$
As a preliminary fact, let us note that:

\begin{lem}
\label{asym-expansion-1}

\begin{itemize}
\item[(I)] The expression $\Phi_{\mathbf b^n}\circ F - \Phi_{\mathbf b^n}(w)-1$ 
extends to be analytic at infinity and to vanish there.
\item[(II)] There is one and only one choice of $\hat {\mathbf b}^n=(\hat b_0,\ldots,\hat b_{n-1})$ for which 
$$\Phi_{\hat {\mathbf b}^n}\circ F - \Phi_{\hat {\mathbf b}^n}(w)-1$$
vanishes to order $n+1$ at infinity.
\end{itemize}
\end{lem}

\begin{proof}
The function $F(w)/w$ is analytic at infinity and takes the value $1$ there, hence $\log(F(w)/w)$ is also analytic at 
infinity and vanishes there. The claim (I) follows immediately.

The claim (II) is verified by a straightforward algebra.
\end{proof}

\noindent
For the choice of  $\hat {\mathbf b}^n=(\hat b_0,\ldots,\hat b_{n-1})$ as in \lemref{asym-expansion-1} (II), we will write
$$\Phi_n\equiv \Phi_{\hat {\mathbf b}^n}.$$
We then have:


(I) Let us select $r>0$ large enough so that
$$|F(w)-w-1|\leq 1/2\text{ for }\Re w\geq r.$$
The right half-plane $H_r=\{\Re w\geq r\}$ is mapped into itself, and for $w\in H_r$
$$\Re F^j(w)\geq \Re w+j/2\text{, and }|\Im F^j(w)-\Im w|\leq j/2.$$
We now set out to estimate 
$$\displaystyle\sum_{j=0}^\infty |F^j(w)|^{-m}.$$
We write $w=s+it$ and consider separately $s\geq |t|$ and $s<|t|.$ In the first case,
we estimate
$$F^j(w)|\geq |\Re F^j(w)|\geq |s+j/2|,$$
and thus
$$\displaystyle\sum_{j=0}^\infty |F^j(w)|^{-m}\leq \displaystyle\sum_{j=0}^\infty |s+j/2|^{-m}\leq \const s^{-(m-1)}\leq \const |w|^{-(m-1)}.$$
In the case $s<|t|$, we split the sum
$$\displaystyle\sum_{j=0}^\infty |F^j(w)|^{-m}=\displaystyle\sum_{0\leq j<|t|} |F^j(w)|^{-m}+\displaystyle\sum_{j\geq |t|} |F^j(w)|^{-m}.$$
In the first sum, we use 
$$|t-\Im F^j(w)|\leq j/2\leq |t|/2$$
to estimate $|F^j(w)|\geq |t|/2.$ Since the number of terms is $\leq t$, this part of the sum is majorized
by $\const(|t|/2)^{-m}|t|\leq \const|w|^{-(m-1)}.$

In the second sum, we estimate $|F^j(w)|\geq |\Re F^j(w)|\geq |s+j/2|\geq |j/2|.$
Thus, the second sum is majorized by
$$\const\displaystyle\sum_{j\geq |t|}j^{-m}\leq\const |t|^{-(m-1)}\leq\const |y|^{-(m-1)},$$
so the desired estimate is proved.

\begin{proof}[Proof of \thmref{asym-expansion-2}]
Set 
$$h_n=\Phi-\Phi_n\text{ and }u(z)=u(-1/w)=\Phi_n\circ F(w)-\Phi_n(w)-1.$$
Note that $u_n$ is an analytic function at the origin, with
$$u_n(z)=d_{n+1}z^{n+1}+\cdots$$
We calculate
$$h_n(w)=h_n\circ F(w)+u_n(-1/w).$$
Iterating, we get
$$h_n(w)=h_n\circ F^k(w)+\displaystyle\sum_{j=0}^{k-1}u_n(-1/F^j(w)).$$
Differentiation yields
\begin{equation}
\label{eq-as1}
h_n'(w)=h_n'(F^k(w))\cdot (F^k)'(w)+\displaystyle\sum_{j=0}^{k-1}v(F^j(w))\cdot(F^j)'(w),
\end{equation}
where 
$$v(w)=u'(-1/w)\cdot w^{-2}=O(|w|^{-(n+2)})\text{ as }w\to\infty.$$
Trivially,
$$h_n'(w)\to 0,\text{ as }\Re w\to +\infty.$$
By \lemref{asym-expansion-3} (II), we can let $k\to\infty$ in (\ref{eq-as1}) to  get
\begin{equation}
\label{eq-as2}
h_n'(w)= \displaystyle\sum_{j=0}^{\infty}v(F^j(w))\cdot(F^j)'(w).
\end{equation}
By \lemref{asym-expansion-3} (II), there exists $r_2>0$ such that
$$|v'(F^j(w))(F^j)'(w)|\leq \const|F^j(w)|^{-(n+2)}$$
for all $j$ and for all $w$ with $\Re w\geq r_2.$

By \lemref{asym-expansion-3} (I),
\begin{equation}
\label{eq-as3}
|h_n'(w)|\leq \const|w|^{-(n+1)}
\end{equation}
for all $w$ sufficiently far to the right. It follows from this estimate, that the limit
$$\underset{s\to\infty}{\lim}h_n(s+it)$$
exists, and is independent of the value of $t$. Let us denote this limit $H$. Then
$$h_n(s+it)=H-\displaystyle\int_{s}^{\infty}h_n'(\zeta+it)d\zeta,$$
so by integrating (\ref{eq-as3}) we get
$$|h_n(w)-H|\leq \const|w|^{-n}\text{ for }\Re w\geq r_2,$$
which is what we set out to prove.
\end{proof}

}

We insert here a simple remark which we will want to refer to
repeatedly. We say that an analytic function $f:U\to\CC$ is {\it real-symmetric} if 
the Taylor coefficients of $f$ are real 
at some  point $x\in U\cap\RR$. Note, that we do not require that $U$ itself is a real-symmetric domain.
Similarly, if $x$ is a point in $\RR$, we say that an analytic germ $f(z)$ at $x$ is real-symmetric if its
coefficients are real.


\begin{prop}\label{th:formally-real}
Let $f$ be a real-symmetric analytic germ of the form (\ref{parabolic-normal-form-3}), and let $\phi_A$ be
an attracting Fatou coordinate for $f$.
Then there is a pure imaginary constant $c$ so that
\[
\overline{\phi_A(z)} = \phi_A(\overline{z}) + c\quad\text{on $B^f_0$.}
\]
Furthermore,  the coefficients $A$, $b_1$, $b_2$, \dots of
  \thmref{asym-expansion-2} are real.

Similar assertions hold for a repelling Fatou coordinate.
\end{prop}

\begin{proof} Let $P$ be a small attracting petal invariant under
  complex conjugation (e.g., a small disk tangent to the imaginary
  axis at the origin.) Since $f$ commutes with complex conjugation,
\[
z \mapsto \overline{\phi_A(\overline{z})}
\]
is another univalent analytic function defined on $P$ and satisfying
the usual functional equation $\phi(f(z)) = \phi(z) + 1$. By
uniqueness up to an additive constant of the Fatou coordinate
(\propref{uniqueness-fatou}), there is a constant $c$ so that
\[
\overline{\phi_A(\overline{z})} = \phi_A(z) + c \quad\text{on $P$.}
\]
In the usual way, this identity extends to all of $B^f_0$ by repeated
application of the functional equation. Applying the identity at any
real point $x$ shows that $c$ is pure imaginary. We omit
the proofs of the other assertions, which are even simpler.
\end{proof}

\subsection{A note on resurgent properties of the asymptotic expansion of the Fatou coordinates}
Let us briefly mention a very different approach to the construction of the asymptotic series
(\ref{eq:21}) originating in the works of J.~{\'E}calle \cite{Ec}.

Recall (see e.g. \cite{Ram}) that a formal power series
$\sum_{m=1}^\infty a_mx^{-m}$ is of {\it Gevrey order }$k$ if
$$|a_m|<CA^n(n!)^{\frac{1}{k}}\text{ for some choice of positive constants }C, A.$$

Consider the asymptotic expansion (\ref{formal-fatou-coord-1}) for the Fatou coordinates, and denote
$$\nu_*(w)\equiv \sum_{j=1}^\infty b_jw^{-j}\text{, so that }\Phi_{\text{ps}}(w)=w-A\log(w)+\nu_*(w).$$
As was shown by {\'E}calle for the case $A=0$ \cite{Ec} and A. Dudko and D. Sauzin in the general case \cite{DuSa}:

\begin{thm}
\label{gevrey-order}
The asymptotic series $\nu_*=\sum_{j=1}^\infty b_jw^{-j}$ is of Gevrey order $1$. 
\end{thm}


\thmref{gevrey-order} is a part of {\'E}calle's theory of {\it resurgence} as applied specifically to Fatou coordinates
(see \cite{Sau} for an account). Recall, that the {\it Borel transform} of a formal power series
$$h_{*}=\sum_{m=1}^\infty a_mx^{-m}$$ consists in applying the termwise inverse Laplace transform:
$$a_m x^{-m}\mapsto \frac{a_m\zeta^{m-1}}{(m-1)!}.$$
In the case when the formal power series is of Gevrey order $1$, this yields a series
$$\sum_{m=1}^\infty \frac{a_m\zeta^{m-1}}{(m-1)!},$$
which converges to an analytic function $\hat h(\zeta)$ in a neighborhood of the origin.

The following theorem describes the phenomenon of {\it resurgence} associated with the asymptotic series $\nu_*$,
discovered by {\'E}calle \cite{Ec}. {\'E}calle presented the proof for the case when $A=0$, and so the logarithmic term is 
absent in (\ref{formal-fatou-coord-1}), and outlined an approach to it in the general case. An independent proof in the general case
was recently given by A.~Dudko and D.~Sauzin \cite{DuSa}:

\begin{thm}
The Borel transform $\hat\nu$ of the formal power series $\nu_*$ analytically extends from the neighborhood of the origin
along every path which avoids the points $2\pi i\ZZ^*$. Furthermore, let $S_{\theta,\eps}$ be any sector
$$S_{\theta,\eps}=\{|Arg(\zeta)-\theta|<\eps\}\text{ such that }\overline{S_{\theta,\eps}}\cap \{2\pi i\ZZ^*\}=\emptyset,$$
and let $\gamma$ be any path as above which eventually lies in $S_{\theta,eps}$ (see \figref{fig-sector}). Denote $\hat\nu^\gamma$ the analytic continuation
along $\gamma$. Then $\hat\nu^\gamma$ is a function of exponential type:
$$|\hat h^\gamma(\zeta)|<C\exp(D|\zeta|),$$
where the constant $C$ and $D$ depend only on $\theta$ and $\eps$.
In particular, $\hat h$ has an analytic continuation $\hat h^+$ to the right half plane $\{\Re\zeta>0\}$
and an analytic continuation $\hat h^-$ to the left half plane $\{\Re\zeta<0\}$.
\end{thm}
 
\begin{figure}
\centerline{\includegraphics[height=0.3\textheight]{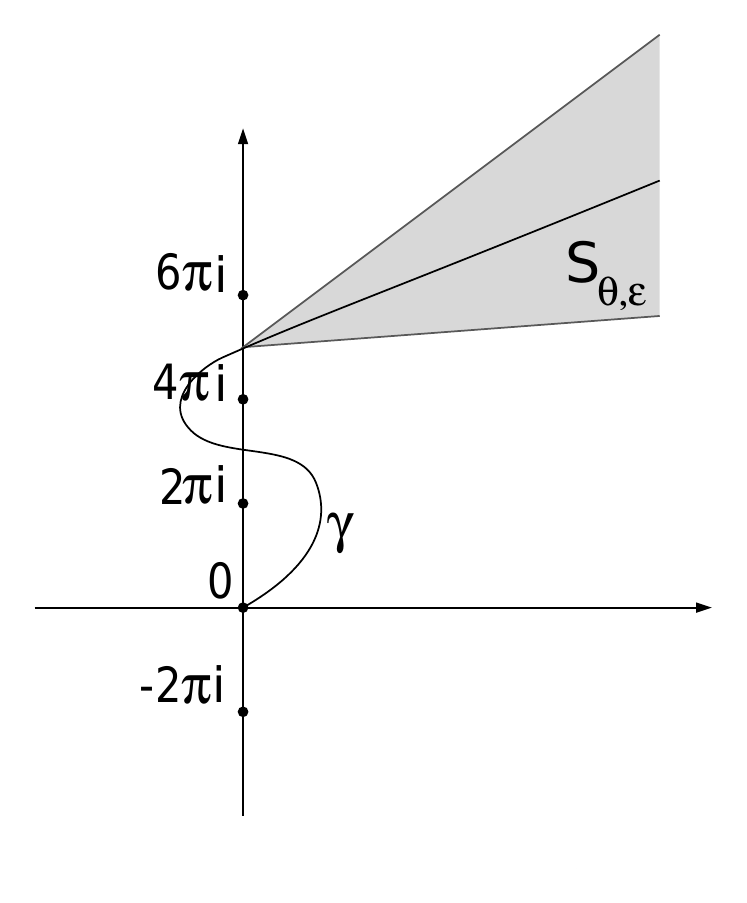}}
\caption{\label{fig-sector}An analytic continuation along a path $\gamma$ has an exponential type in a sector $S_{\theta,\eps}$ whose
closure does not contain any of the points $2\pi i\ZZ^*$.
}
\end{figure}

\noindent
Consider the standard Laplace transforms 
$$\cL^+(w)=\int_0^\infty e^{-w\zeta}\hat h(\zeta)d\zeta\text{ and }\cL^-(w)=\int_{-\infty}^0 e^{-w\zeta}\hat h(\zeta)d\zeta.$$
Note that 
$$\nu^+\equiv \cL^+\hat\nu^+\text{ and }\nu^-\equiv \cL^-\hat\nu^-$$
are defined for $\Re w$ sufficiently large.
The resurgent properties of $\nu_*$ are completed by the following refinement of \thmref{asym-expansion-2}:
\begin{thm}\cite{DuSa}
\label{gevrey-asym}
The analytic functions 
$$\Phi_A(w)=w-A\log w+\nu^+(w)\text{ and }\Phi_R(w)=w-A\log w+\nu^-(w).$$
Furthermore, let $\alpha<\pi/2$. Then there exist positive constants $B=B(\alpha)$, $C=C(\alpha)$ such that
$$\left| \Phi_A(w)-\left( w -A\log w+\sum_{j=1}^n b_j w^{-j}\right)\right|<C^{n+1}(n+1)!|w|^{-(n+1)}$$
uniformly in a sector $\{|\text{Arg}(w)|<\pi/2+\alpha\text{ and }|w|>B\}$,
and similarly for $\Phi_R$.
\end{thm}

\noindent
Thus $\nu_*$ is a {\it Gevrey asymptotic series of order $1$.} It follows from \thmref{gevrey-asym} and a Stirling formula
estimate that the first $n$ terms of the asymptotic series $\nu_*$ are useful for numerically estimating the Fatou
coordinates for $|w|>\text{const}\cdot n$. 

The values of the constants in \thmref{gevrey-asym} can be estimated explicitly. Dudko \cite{Dudko} shows the following.
Let us write $$F(w)=w+1+a(w),\text{ where }a(w)=Aw^{-1}+O(w^{-2}),$$
and introduce functions
$$b(w)=a(w-1)=\sum_{k=1}^\infty c_k w^{-k},\text{ and }$$
$$m(w)=-A\log\left(\frac{1+w^{-1}b(w)}{1-w^{-1}}\right)+b(w)=\sum_{k=1}^\infty d_k w^{-k}.$$
Let $C_0, \beta>0$ be such that for all $k\in\NN$
$$|c_k|\leq C_0\beta^{k-1}\text{ and }|d_k|\leq C_0\beta^{k-1}.$$
Finally, let $S$ be such that $|\hat\nu(\zeta)|\leq S$ for $|\zeta|\leq 2$.
Set
$$B=\frac{(\beta_0+\frac{C_0}{\cos\alpha})(\frac{\alpha}{8}+1)+1}{\sin\frac{\alpha}{8}}.$$
Then for all $n\in\NN$
$$\left| \Phi_A(w)-\left( w -A\log w+\sum_{j=1}^n b_j w^{-j}\right)\right|<\left(S+\frac{C_0}{\cos\alpha}\left(\frac{8}{\alpha}\right)^{n+1}\right)(n+1)!|w|^{-(n+1)}$$
uniformly in a sector $\{|\text{Arg}(w)|<\pi/2+\alpha\text{ and }|w|>B\}$,
and similarly for $\Phi_R$.

\subsection{{\'E}calle-Voronin invariants and definition of parabolic
  renormalization.}\

The Riemann surface $\CC/\ZZ$ 
has two punctures at the upper end ($\Im z\to+\infty$), and at the lower end ($\Im z\to-\infty$).
Filling them with points $\oplus$ and $\ominus$ respectively, we obtain the Riemann sphere.
The mapping
$$\ixp(z)\equiv \exp(2\pi i z),$$
conformally transforms $\CC/\ZZ\mapsto \CC^*$, sending $\oplus\to 0$ and $\ominus\to\infty$.

Consider a germ $f$ with a simple parabolic fixed point at $0$, normalized as in (\ref{parabolic-normal-form-3}). 
Let $P_A$ and $P_R$
be a pair of ample petals for $f$, and denote $f^{-1}$ the local branch of the inverse which fixes the origin. Note that $f^{-1}$ extends
univalently to $P_R\cup f(P_A)$.
Fix a choice of the Fatou coordinates $\phi_A$ and $\phi_R$. 

The forward orbits originating in $P_A$ are parametrized by points in the attracting cylinder $\cC_A$. Similarly, $f^{-1}$-orbits
in $P_R$ are parametrized by points in $\cC_R$. By the definition of an ample petal, $P_A\cap P_R\neq \emptyset$. Let $z$ be any point
in the intersection of the petals. It is trivial to see that the correspondence
$$\tl\phi_R(z)\mapsto \tl\phi_A(z)$$
defines a mapping from a subset of $\cC_R$ to $\cC_A$. We denote this mapping by $h$. It is more convenient for us to pass to $\CC^*$
from $\CC/\ZZ$ via the exponential, and consider the mapping $\bm{h}$ formally defined as 
\begin{equation}\label{eq:renorm-heuristic}
\bm{h} = (\ixp \circ \phi_A) \circ (\ixp \circ \phi_R)^{-1}.
\end{equation}
We note:
\begin{lem}\label{th:renorm-analytic}
The mapping $\bm{h}$ is analytic.
\end{lem}

\noindent
Furthermore,
\begin{lem}\label{th:W-nbhd-of-0}
The domain of definition of $\bm{h}$ contains a punctured neighborhood of $0$, and 
a punctured neighborhood at infinity. The singularities of
  $\bm{h}$ at $0$ is removable, the analytic extension taking the
  value $0$ at $0$. Similarly,  $\bm{h}$ has a pole at $\infty$.
\end{lem}

\begin{proof}
By local theory, $P_A \cap P_R$ contains the (upward-facing)
circular sector $\{ w : 0 < \vert w \vert < r, \pi/2 -
\delta < \Arg(w) < \pi/2 + \delta\}$ for sufficiently small $r$, for
any $\delta < \pi/2$. By \propref{fatou coords cover} and the crude
asymptotic estimate $\phi_R(z) \approx -1/z$ as $z \to 0$ staying away
from the positive real axis (\propref{crude
  asymptotics}),
 it follows that the image under $\phi_R$
of this sector contains the upper half-strip
\begin{equation}\label{eq:upper-strip}
\phi_R(P_A\cap P_R ) \supset \{ x + i y : 0 \leq x \leq 1, y > R \}
\end{equation}
for sufficiently large $R$, so the domain of definition of $\bm{h}$ contains a punctured neighborhood
\[
\{ z: 0< \vert z \vert < \exp(- 2 \pi R).
\]
By \propref{crude
  asymptotics}) the infimum
of $\Im(\phi_A(x+iy))$ over the strip on the right of
(\ref{eq:upper-strip}) goes to $+\infty$ as $R \to \infty$, so
$\bm{h}(z) \to 0$ as $z \to 0$, as claimed. The proofs for the
assertions about $\infty$ are similar.
\end{proof}

\noindent
Let us analytically continue $\bm{h}$ to the origin and to infinity. We denote $h^+(z)$ the analytic germ
of $\bm{h}$ at $0$, and $h^-(z)$ the analytic germ at $\infty$. When necessary to emphasize the dependence on the germ $f$
we will write $\bm{h}_f$ and $h^\pm_f$.

It is easy to see that:
\begin{prop}
The germs $h^+$, $h^-$ do not depend on the choice of the petals $P_A$ and $P_R$.
\end{prop} 

\noindent
By \propref{uniqueness-fatou}, an attracting (repelling) Fatou coordinate  differs from our choice of $\phi_A$ ($\phi_R$)
by an additive constant. A trivial
verification shows that replacing $\phi_A$ by $\phi_A + c_A$ and
$\phi_R$ by $\phi_R + c_R$ changes $h^\pm$ to
\begin{equation}\label{eq:change-of-normalization}
w \mapsto \lambda_A h^\pm(\lambda_R^{-1} w)
\end{equation}
with
\begin{equation}\label{eq:rescaling-factors}
\lambda_A = \exp(2 \pi i c_A),\quad \lambda_R = \exp(2
\pi i c_R).
\end{equation}
The scale change factors $\lambda_A$ and $\lambda_R$ can be given arbitrary
non-zero values by the appropriate choices of $c_R$ and $c_A$.

We are going to show that the zero of $h^+$ at $0$ and the pole of
$h^{-}$ at $\infty$ are simple. We will see this by deriving a
useful explicit formula for the respective leading coefficients. To
write this formula, we need to introduce some notation. Writing just
the first few terms in the asymptotic approximations to the Fatou
coordinates (\propref{asym-expansion-2}):
\begin{equation}\label{eq:asym-expansion-4}
\begin{split}
\phi_A(z) &= - \frac{1}{z} + A \log(-z) + C_A + O(z)\\
\phi_R(z) &= - \frac{1}{z} + A \log(z) + C_R + O(z)
\end{split}\end{equation}
where
\begin{itemize}
\item $A = 1 - \alpha$, with $\alpha$ the coefficient of $z^3$ in
  the Taylor expansion for $f$ at $0$.
\item the logarithms mean the {\em standard branch,} i.e., the branch
  with a cut along the negative real axis and real values on the
  positive real axis.
\item $C_R$ and $C_A$ are complex constants (specifying the
  normalization of the Fatou coordinates)
\item the term $O(z)$ in the first equation means a quantity which goes
  to zero at least as fast as $\vert z \vert$ as $z \to 0$
  inside any sector of the form $\{ z \neq 0 : - \alpha < \Arg(z) < \alpha
  \}$ with $\alpha < \pi$ (a {\em left-facing} sector), and the $O(z)$
  in the second means as $z \to 0$ in any similarly defined right-facing sector.
\end{itemize} 

\begin{prop}\label{thm:simple-zero}
\[
(h^+)'(0) = \exp\bigl( -2 A \pi^2 + 2 \pi i (C_A - C_R)\bigr)
\]
and
\[
\lim_{z \to \infty} \frac{h^{-}(z)}{z} = \exp\bigl( + 2 A \pi^2 + 2
\pi i (C_A - C_R)\bigr)
\]
In particular: $h^+$ has a simple zero at $0$, $h^-$ a simple pole
at $\infty$, and
\begin{equation}
\label{product derivatives}
(h^+)'(0) \cdot (h^-)'(\infty) = \exp(-4 \pi^2 A)
\end{equation}

\end{prop}

\begin{proof}
To prove the formula for $(h^+)'(0)$, we look at points $z(t)$ of the
form $it$, with $t$ small and positive; such $w$'s are in $B^f_0 \cap
P_R$ for any ample petal $P_R$. Inserting into (\ref{eq:asym-expansion-4})
and using $\log(\pm it) = \log(t) \pm i \pi/2$:
\[
\phi_R(z(t)) = i t^{-1} + A \log t +  i A \pi/2 + C_R +
O(t)
\]
\[
\phi_A(z(t)) = i t^{-1} + A \log t - i A \pi/2 + C_A + O(t).
\]
We put $w(t) := \exp(2 \pi i \phi_R(z(t)))$; then $h^+(w(t)) = \exp( 2
\pi i \phi_A(z(t)))$ so
\[
\frac{h^{+}(w(t))}{w(t)} = \exp(2\pi i ( - i A \pi + C_A - C_R +
O(t))),
\]
so
\[
(h^+)'(0) = \lim_{t \to 0^+} \frac{h^{+}(w(t))}{w(t)} = \exp\bigl( -2 A
\pi^2 + 2 \pi i (C_A - C_R)\bigr),
\]
as asserted. The assertion about $h^-$ is proved by a similar calculation.

\end{proof}

Let us say that
two pairs of germs at zero and infinity $(h^+_1,h^-_1)$ and $(h^+_2, h^-_2)$ are {\em equivalent} if there exist
non-zero constants $\lambda_A$ and $\lambda_R$ so that
\[
h^{\pm}_2(z) = \lambda_A h^{\pm}_1( \lambda_R^{-1} z).
\]
In view of the following, let us call the equivalence class of the germs $\bm{h}_f=h^\pm_f$ the {\it {\'E}calle-Voronin invariant} of $f$.
By (\ref{product derivatives}) the {\'E}calle-Voronin invariant
determines the {\it formal} conjugacy class of the simple parabolic germ $f$. 
A much stronger result is due to Voronin \cite{Vor} (an equivalent
version was formulated by {\'E}calle \cite{Ec}):

\begin{thm}
Two analytic germs $f_1$ and $f_2$ of the form
(\ref{parabolic-normal-form-3}) are conjugate by a local conformal
change of coordinates $\varphi(z)$ with $\varphi(0)=0$ if and only if
their \'Ecalle-Voronin invariants are equal.
\end{thm}

\begin{proof}[Sketch of proof]
The proof of the ``only if'' direction is very easy -- essentially a
diagram-chase. We assume
\[
f_1 = \varphi^{-1} \circ f_2 \circ \varphi\quad\text{on a neighborhood
  of $0$,}
\]
where $\varphi$ is analytic at $0$ with $\varphi(0) = 0$, $\varphi'(0)
>0$. We label various objects attached to $f_1$ and $f_2$ with
indices $1$ and $2$ respectively.  Because the assertion concerns
germs, we can cut down the domains of the various functions appearing
as convenient. We set things up as follows:
\begin{itemize}
\item We fix a domain for $\varphi$ so that it is univalent.
\item We fix a domain for $f_2$ contained in the image of $\varphi$,
  on which $f_2$ is univalent, and so that the image of $f_2$ is
  contained in the image of $\varphi$.
\item We take for the domain of $f_1$ the preimage under $\varphi$ of
  the domain of $f_2$. Then the equation
\[
f_1 = \varphi^{-1} \circ f_2 \circ \varphi
\]
is exact, including domains.
\end{itemize}
It is then obvious that:
\begin{itemize}
\item $\varphi$ maps any ample petal for $f_1$ to an
  ample petal for $f_2$. We fix any ample repelling petal $P^{(1)}_R$ to
  use in the construction of $\bm{h}_1$, and we use $P^{(2)}_R :=
  \varphi P^{(1)}_R$ as repelling petal to construct $\bm{h}_2$.
\item If $\phi^{(2)}_A$ is an attracting Fatou coordinate for $f_2$
  defined on $B^{f_2}$ then $\phi^{(2)}_A \circ \varphi$
  is an attracting Fatou coordinate for $f_1$, and similarly for
  repelling Fatou coordinates. To construct \'Ecalle-Voronin pairs for
  the two mappings, we choose any attracting Fatou coordinate
  $\phi^{(2)}_A$ for $f_2$ and use $\phi^{(1)}_A := \phi^{(2)}_A \circ
  \varphi$ for $f_1$ (and
  similarly for repelling Fatou coordinates.)
\end{itemize}

With things organized this way, an entirely mechanical verification
shows that $\bm{h}_1$ and $\bm{h}_2$ are identical pairs of germs.

Conversely, assume that the \'Ecalle-Voronin invariants of $f_1$ and
$f_2$ are equal. By fixing the constants in Fatou coordinates for
the two mapping appropriately, we can arrange that
\begin{equation}\label{eqmod-1}
h_1^+ = h_2^+\quad\text{and}\quad h_1^{-} = h_2^{-}
\end{equation}
on neighborhoods of $0$ and $\infty$ respectively. Fix also small
attracting and repelling petals, for instance,
\[
P_R = -P_A = \{ z: -1/z \in \Delta(\alpha, R)\},
\]
with $\Delta(\alpha, R)$ as defined in (\ref{eq:Delta-alpha-R}),
$\alpha$ some number in $(\pi/2, \pi)$, and $R$ large enough. With
these choices, $P_R \cap P_A$ has only two components, an upper one
contained in $U^+$ and a lower one contained in $U^-$. Then
$\bigl(\phi^{(2)}_A\bigr)^{-1} \circ \bigl( \phi^{(1)}_A \bigr)$
conjugates $f_1$ to $f_2$ on $P_A$, and, correspondingly,
$\bigl(\phi^{(2)}_R\bigr)^{-1} \circ \bigl( \phi^{(1)}_R \bigr)$
conjugates $f_1$ to $f_2$ on $P_R$. From (\ref{eqmod-1}) it follows --
possibly after adding appropriate integers to the various Fatou
coordinates -- that the two conjugators agree on $P_A \cap
P_R$. Putting them together, we obtain a conjugator on $P_A \cup P_R$,
a punctured neighborhood of $0$. It is easy to see, using the
asymtotic estimates for Fatou coordinates, that this conjugator extends
analytically through $0$ with the value $0$ there.

\end{proof}

\noindent In fact, all equivalence classes of pairs of
germs actually occur as \'Ecalle-Voronin invariants:

\begin{thm}
\label{realization}
Let $h^+$ be an analytic germ at $0$, with a simple zero there, and let
$h^-$ be a meromorphic germ at $\infty$ with a simple pole there. 
Denote $\lambda^+=(h^+)'(0)$, $\lambda^-=(h^-)'(\infty)$, and let 
$$\lambda^+\cdot\lambda^-=e^{-4\pi^2 A}.$$
Then there exists a simple parabolic germ $f(z)$ at the origin of the form 
\[
f(z) = z + z^2 + (1-A)z^3+\cdots
\]
whose \'Ecalle-Voronin invariant is the equivalence class of $(h^+,
h^-)$. 
\end{thm}

\begin{proof}[Sketch of proof]
We follow the argument given in \cite{BH}. Let us choose the lifts $H^\pm(w)$ of the germs $h^\pm$ via the exponential map:
$$e^{2\pi iH^\pm(w)}=h^\pm(e^{2\pi i w}).$$
For a sufficiently large value of $R>0$, the maps $H^\pm$ are defined in the half-planes $V^\pm\equiv \{\pm\Im w>R\}$.
Denote $U^\pm\equiv H^\pm(V^\pm)$ -- these domains are invariant under the  translation $w\mapsto w+1$ and contain half-planes
$\{\pm\Im w>R'\}$.
We have
$$H^\pm(w)=w+\frac{1}{2\pi i}\log(\lambda^\pm)+o(1),$$
let us select the branches of the logarithm so that $\log(\lambda^+\lambda^-)=4\pi^2A.$
Fix $w^\pm\equiv H^\pm(R\pm iR).$
We define $V$ as the union of $V^\pm$ and the left half-plane $\{\Re(w)<-R\}$ and $U$ as the union of $U^\pm$ and the half-plane to the right
of the line passing through $w^+$ and $w^-$. Let $\cV$ denote the Riemann surface obtained by guing 
$V$ and $U$ via $H^+$, $H^-$ (see Figure \ref{fig-gluing}). The result follows from the following claim:

\begin{figure}
\centerline{\includegraphics[height=0.3\textheight]{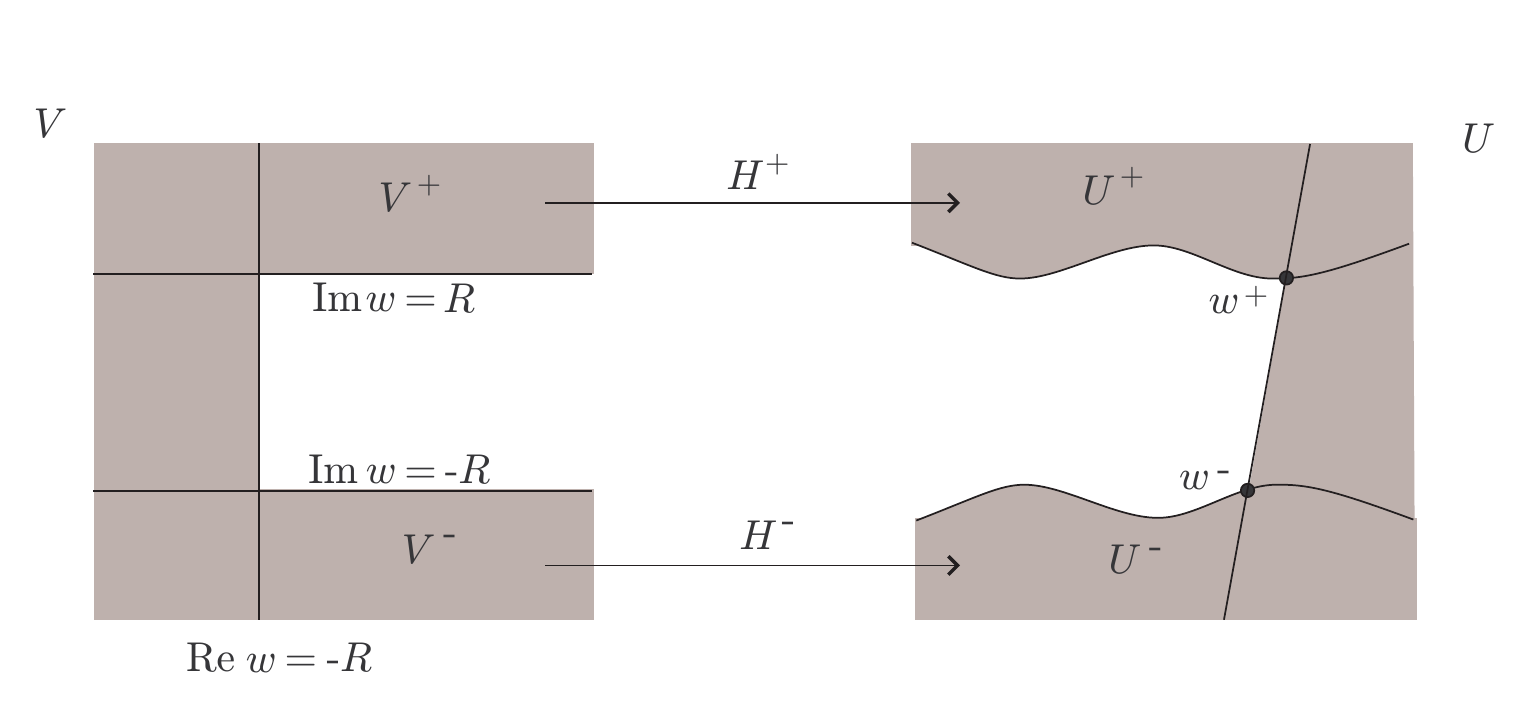}}
\caption{\label{fig-gluing}An illustration of the proof of \thmref{realization}}

\end{figure}

\medskip
\noindent
{\bf Claim.} {\sl For sufficiently large values of $R$ the Riemann surface $\cV$ is conformally isomorphic to $\CC\setminus \DD$.}

\medskip
\noindent
The proof of the claim is fairly straightforward, so we leave it to the reader. Let us show how to complete the argument assuming the claim.
Let us select large enough value of $R$, and let $\chi:\cV\to\CC\setminus\DD$ be the conformal isomorphism with $\chi'(\infty)>0$. 
Denote 
$$V'=V\setminus\{-R-1\leq \Re(w)<-R\},$$
and let $\cV'$ be a subsurface of $\cV$ obtained by replacing $V$ with $V'$. The unit translation $w\mapsto w+1$ maps $V'\to V$ and 
$U\to U$ and commutes with $H^\pm$. Hence it induces an analytic map $F:\cV'\to \cV$. It is easy to see that
$$\hat F(w)\equiv \chi \circ F\circ \chi^{-1}(w)=w+1+\frac{A}{w}+o(\frac{1}{w}).$$
Indeed, the projections to $V$ and $U$ are the Fatou coordinates at infinity for $\hat F$, and $H^\pm$ are 
the changes of coordinates between $V$ and $U$. By construction, 
$$f(z)\equiv -\frac{1}{\hat F(-\frac{1}{z})}$$
is an analytic germ at the origin with the desired properties.

\end{proof}

We are finally ready to define parabolic renormalization:

\begin{defn}
\label{renormalizable_bis}
We say that a parabolic germ $f(z)$ of the form
(\ref{parabolic-normal-form-2}) is {\it renormalizable} if for some --
and hence for all -- choices of normalization of $\phi_A$ and
$\phi_R$, the coefficient of $z^2$ in the Taylor series of $h^+$ at
$0$ does not vanish.
\end{defn}

Then, by the rescaling formulas (\ref{eq:change-of-normalization}) and
(\ref{eq:rescaling-factors}), then exist choices of the normalizations
of the attracting and repelling Fatou coordinates so that the
corresponding $h^+$ has the form
\begin{equation}\label{eq:normalized-h}
h^+(w) = w + w^2 + \cdots
\end{equation}
i.e., so that $h^+$ has a normalized simple parabolic fixed point at
$0$. The rescaling factors $\lambda_A$ and $\lambda_R$ which
accomplish this are uniquely determined; the corresponding additive
constants $c_A$ and $c_R$ are uniquely determined modulo $\ZZ$.

\begin{defn}
\label{defn-par-ren_bis}
For a renormalizable germ $f(z)$ of the form
(\ref{parabolic-normal-form-3}) we will call the analytic germ of the
unique rescaling of $h^+$ with the form (\ref{eq:normalized-h}) the
{\it parabolic renormalization} of $f$. We will use the notation
$\cP f$ for the parabolic renormalization. 

\end{defn}

We thus have
\begin{equation}
\label{eqn-par-ren}
\pr f =\ixp\circ \phi_A\circ (\phi_R)^{-1}\circ\ixp^{-1},
\end{equation}
with $\phi_A$, $\phi_R$ the  appropriately
normalized Fatou coordinates and with suitably selected branches of
the inverses.

\subsection{Analytic continuation of parabolic renormalization}
Parabolic renormalization, as defined above, maps a renormalizable analytic germ $f$ at the origin of the form
(\ref{parabolic-normal-form-3}) to a germ $\cP(f)$ of the same form. 
We will change the point of view now, and will talk about an analytic map  $f$, defined in a domain $\Dom(f)\ni 0$,
whose germ at $0$ is of the form (\ref{parabolic-normal-form-3}). Note that we do not impose any conditions on the
naturality of the domain $\Dom (f)$ at this point. The map $f$ may analytically extend beyond $\Dom(f)$, however,
when considering orbits of points under $f$, we restrict ourselves only to the orbits which do not leave $\Dom(f)$.

As before, we denote $B^f\subset \Dom(f)$  the basin of $0$, and $B^f_0\subset B^f$  the immediate basin of $0$,
that is, the connected component of $B^f$ which contains an attracting petal.
Let us fix an attracting Fatou coordinate $\phi_A$ and extend it to all of $B^f$ via the functional equation.
We also choose a repelling petal $P_R$ and fix a repelling Fatou coordinate $\phi_R$ on it.

We make the following simple observation:
\begin{lem}\label{th:renorm-well-defined}
Let $z_1 \in B^f\cap P_R$ and let $z_2 \in P_R$, and assume
\[
\ixp(\phi_R(z_1)) = \ixp(\phi_R(z_2))
\]
Then $z_2 \in B^f$ and
\[
\ixp(\phi_A(z_1)) = \ixp(\phi_A(z_2))
\]
\end{lem}

\begin{proof}
From the assumption:
\[
\phi_R(z_2) - \phi_R(z_1) =: m \in \ZZ.
\]
We consider separately $m \geq 0$ and $m < 0$. In the first case,
$\phi_R(f^{-m}(z_2)) = \phi_R(z_1)$. Since $f^{-m}(z_2)$ and $z_1$ are
both in $P_R$, and since $\phi_R$ is univalent on $P_R$, $f^{-m}(z_2)
= z_1$, so $z_2 = f^{m}(z_1)$, so $z_2 \in B^f$ and $\phi_A(z_2) =
\phi_A(z_1) + m$ and the asserted equality holds.

Now suppose $m < 0$, and let $p := -m > 0$. Then $\phi_R((f^{-p}(z_1)
= \phi_R(z_2)$, so, arguing as above, $z_1 = f^p(z_2)$, so $z_2 \in
B^f$ so, again, the asserted equality holds.

\end{proof}
The above lemma implies:
\begin{cor}
\label{extension-h}
The pair of analytic germs $\bm{h}$ extends to 
$$\cD(\bm{h}) \equiv \{w=\ixp\circ \phi_R(z)\;|\text{ where }z\in B^f\cap P_R\}\cup\{0,\infty\}.$$
\end{cor}

\noindent
Furthermore,
\begin{prop}
The domain $\cD(h)$ does not depend on the choice of the repelling petal $P_R$.
\end{prop}
\begin{proof}

Suppose $P^{(1)}_R$ is another repelling petal, and write
$\bm{h}^{(1)}$ for the corresponding function. If $w \in
\cD(\bm{h})$, then $w$ can be written as $\ixp \circ \phi_R(z)$, with
$z \in B^f \cap P_R$. For large enough $n$, $f^{-n}(z) \in
P^{(1)}_R$, and, since $f^n(f^{-n}(z)) = z \in B^f$, $f^{-n}(z) \in
B^f$. Hence $w = \ixp \circ \phi_R(f^{-n}(z)) \in \cD(\bm{h}^{(1)})$
and
\[
\bm{h}^{(1)}(w) = \ixp \circ \phi_A(f^{-n}(z)) = \ixp \circ \phi_A(z)
= \bm{h}(w).
\]
This shows that $\cD(\bm{h})=\cD(\bm{h}^1)$.
\end{proof}

Since the derivative of the local inverse never vanishes, we have: 

\begin{lem}\label{th:critical-values}
The only possibilities for critical values of $\bm{h}$ in $\cD(\bm{h})$ are the images
under $\ixp$ of the critical values of $\phi_A$. More explicitly: $v$
is a critical value of $\bm{h}$ if and only if it can be written
\[
v = \ixp \circ \phi_A(z_c)
\]
where $z_c$ is a critical point of $f$ belonging to $B^f$ and
admitting a backward orbit converging to $0$.
\end{lem}


We further prove:

\begin{thm}\label{th:W-Jordan-domain}
Suppose  $B^f_0$ is a Jordan domain, such that $\phi_A$ cannot be analytically continued through any point
of $\partial B^f_0$. Then the pair of germs $\bm{h}=(h^+,h^-)$ has a maximal domain of analyticity $\Dom(\bm{h})$,
which is a union of two Jordan domains $W^+\ni 0$ and
$W^-\ni\infty$. Furthermore, let $P_R$ be a repelling petal for $f$. Then 
$$\partial\Dom(\bm{h})\subset\ixp\circ \phi_R(\partial B_0^f\cap P_R).$$
\end{thm}

\begin{proof}
Let $P_R$ be a repelling petal which maps under $\phi_R$ to a left
half-plane. After shrinking $P_R$ if necessary, we further assume that
$f(P_R)$ 
 is also
a petal. For purposes of this proof, $f^{-1}$ and $\phi_R^{-1}$ will
mean the (two-sided) inverses of the respective restrictions to
$P_R$. By local theory, $P_R \cup \{0\}$ contains a neighborhood of
$0$ in $\partial B^f_0$. The component of $\partial B^f_0 \cap (P_R
\cup \{0\})$ containing $0$ is therefore a Jordan arc, which we denote
by $\widetilde{\sigma}$. The intersection $\overline{P_R} \cap (\partial B^f_0 \setminus
\widetilde{\sigma})$ is compact and does not contain $0$, so
$\Re(\phi_R(\,.\,))$ is bounded below on it. Hence, for $\beta$
sufficiently negative, the arc $$\widetilde{\gamma} := \phi_R^{-1}\{
\Re{w} = \beta \} $$ is disjoint from $\partial B^f_0 \setminus
\widetilde{\sigma}$. We give the arc $\widetilde{\gamma}$ the
clockwise orientation, which means that $\Im(\phi_R(\,.\,))$ goes to
$+\infty$ at the beginning of $\widetilde{\gamma}$.  The arc
$\widetilde{\gamma}$ must intersect $B^f_0$; otherwise, it (with $0$
appended) would bound a repelling petal contained in $B^f_0$, which is
impossible by local theory. Thus
\[
\emptyset \neq \widetilde{\gamma} \cap \partial B^f_0 \subset
\widetilde{\sigma}. 
\]

Let $\gamma$ denote the initial segment of $\widetilde{\gamma}$ up to
its first intersection with $B^f_0$, and denote the first intersection
point by $z_1$. We have set things up so as to guarantee that $z_1 \in
\widetilde{\sigma}$. Then $f^{-1}\gamma$ is another
arc in $B^f_0 \cap P_R$, running from $0$ to $f^{-1}(z_1)$, and
disjoint from $\gamma$. Let $\sigma$ denote the subarc of $\partial
B^f_0$ running from $z_1$ to $f^{-1}(z_1)$ (with end-points included
this time.) Then
\[
\delta := \gamma + \sigma - f^{-1}\gamma
\]
(meaning: first traverse $\gamma$, then $\sigma$, then
$f^{-1} \gamma$ backwards) is a Jordan curve. Since $\delta\setminus\{0\}$ 
is contained in the Jordan domain $P_R$, the  domain
$\Delta$ bounded by $\delta$ is also contained in $P_R$. Since
$\overline{P_R}$ is contained in the petal $f(P_R)$, the Fatou coordinate $\phi_R$ is
analytic on $\overline{\Delta} \setminus \{ 0 \}$, and
\begin{equation}\label{eq:defn-of-W}
W := \ixp \circ \phi_R(\Delta \cup \gamma \cup f^{-1}\gamma)
\end{equation}
is contained in $\cD(\bm{h})$.
We are going to argue that $W \cup \{0\}$ is a Jordan domain, and that
it is equal to $W^+$. The main steps
are to show
\begin{enumerate}
\item $W$ is a connected open punctured neighborhood of $0$.
\item $\ixp \circ \phi_R$ maps $\sigma$ to a Jordan curve
  $\widehat{\sigma}$.
\item The boundary of $W$ is $\widehat{\sigma} \cup \{ 0 \}$
\end{enumerate}
Since $W$ is the image under $\ixp \circ \phi_R$ of a subset of $P_R
\cap B^f_0$, $W \subset \cD(\bm{h})$. On the other hand,
$\widehat{\sigma}$ is contained in the image of $\partial B^f_0$ and
hence disjoint from $\Dom(f)$. By the Jordan curve theorem, $\CC
\setminus \widehat{\sigma}$ has exactly two connected components. As
$W \cup \{ 0 \}$ is connected, open, and disjoint from
$\widehat{\sigma}$, it is contained in one of these components; we
temporarily denote the containing component $U$. If $W$ were not all
of $U$, then its relative boundary in $U$ would have to be non-empty,
contradicting (3). Thus, $W \cup \{0\}$ is a Jordan domain. Further,
$W \subset \cD(\bm{h})$ and $\widehat{\sigma}$ does not intersect
$\cD(\bm{h})$, so $W$ must be a component of $\cD(\bm{h})$, i.e.,
must be $W^+\setminus \{ 0 \}$.

\medskip\noindent To prove (1): By definition (\ref{eq:defn-of-W}) is the
continuous image of a connected set and so connected. The image of
$\Delta \cup \gamma$ under $\phi_R$ contains an upper half-strip $\{ u
+ i v : \beta - 1 < u \leq \beta, v > R \}$ for sufficiently large
$R$, and this half-strip maps under $\ixp$ to a punctured disk about
$0$, so $W$ is a punctured neighborhood of $0$.

To show that $W$ is open, it suffices to show that the image of
$\gamma$ under $\ixp \circ \phi_R$ -- which is the same as the image
of $f^{-1}\gamma$ -- is in the interior of $W$. This follows from the
way $\ixp \circ \phi_R$ glues together the two edges. Roughly, any
sufficiently small disk about a point of the image of $\gamma$ is the
disjoint union of three parts, respectively the images under $\ixp
\circ \phi_R$ of
\begin{itemize}
\item a differentiably-distorted half-disk in $\Delta$ with diameter along
  $\gamma$
\item another differentiably-distorted half-disk in $\Delta$ with diameter along
  $f^{-1} \gamma$
\item a subarc of $\gamma$
\end{itemize}
Thus, any such disk is contained in the interior of $W$. 

\medskip\noindent To prove (2), we need

\begin{lem}\label{sigma-hat-Jordan}
$ixp \circ \phi_R$ is injective on $\sigma \setminus \{ z_1,
  f^{-1}(z_1) \}$ (that is, on $\sigma$ with its end-points deleted.)
\end{lem}

\noindent We postpone the proof and proceed to deduce (2) from the
lemma. Composing $\ixp \circ \phi_R$ with a continuous
paramet\-rization of $\sigma$ gives a continuous mapping from the
parameter interval $[0,1]$ to $\CC$ which is injective except for
sending $0$ and $1$ to the same point. In an obvious way, this
produces a continuous injective mapping of the circle $\TT$ to $\CC$,
that is, a parametrized Jordan curve in $\CC$.

\medskip\noindent To prove (3): it is nearly obvious that
$\widehat{\sigma}$ and $0$ are contained in the boundary of $W$. To
prove the converse: Let $w$ be a boundary point of $W$; let $(w_n)$ be
a sequence in $W$ converging to $w$; and, for each $n$, let $z_n$ be a
point of $\Delta \cup \gamma$ with $w_n = \ixp \circ \phi_R(z_n)$. By
compactness of $\overline{\Delta}$, we can assume -- by passing to a
subsequence -- that $z_n \to z \in \overline{\Delta}$. If $z = 0$,
then $z_n \to 0$ inside $\Delta \cup \gamma$, which implies
$\Im(\phi_R(z_n)) \to \infty$, which implies $w_n \to 0$, i.e., $w =
0$. Otherwise, $w = \ixp \circ \phi_R(z)$. Since $w \notin W$, $z$
cannot be in $\Delta$, or in $\gamma$, or in $f^{-1}\gamma$, and we
have already dealt with the possibility $z=0$, so we are left only with $z
\in \sigma$, which implies $w = \ixp \circ \phi_R(z) \in
\widehat{\sigma}$.

Modulo the proof of \lemref{sigma-hat-Jordan}, this shows that $W^+$
is a Jordan domain and also gives a useful representation for
$\partial{W}$ as the image under $\ixp \circ \phi_R$ of a fundamental
domain for $f$ in $\partial B^f_0$. From this latter representation,
it is evident that, if $\ixp \circ \phi_A$ cannot be
analytically continued through $\partial B^f_0$, then  $h^+ = \bm{h}|_{W^+}$ cannot
be analytically continued through $\partial W$. Thus, all the
assertions about $h^+$ are proved; the proofs of those about $h^-$ are
similar. 

\end{proof}

\begin{proof}[Proof of \lemref{sigma-hat-Jordan}]
We use the same notation as in the proof of
\propref{th:W-Jordan-domain}. Recall that $\widetilde{\sigma}$ denotes
the component of $B^f_0 \cap P_R$ which contains the parabolic point
$0$. Deleting $0$ splits
$\widetilde{\sigma}$ into two subarcs, which we denote by
$\widetilde{\sigma}_{\pm}$; we will say later which is which.
We parametrize $\widetilde{\sigma}$ as $t \mapsto
\widetilde{\sigma}(t): t_- < t < t_+$, with parameter $t = 0$
corresponding to the parabolic point $0$ and with
$\widetilde{\sigma}_+$ corresponding to $t > 0$.
Since $f^{-1}$ maps $P_R$ into itself and $\widetilde{\sigma}$
into $\partial B^f_0$, it maps $\widetilde{\sigma}$ into
itself. Using the parametrization, we
conjugate $f^{-1}$ on $\widetilde{\sigma}$ to a one-dimensional
mapping:
\[
f^{-1}(\widetilde{\sigma}(t)) = \widetilde{\sigma}(j(t)),
\]
where $j(\,.\,)$ is a continuous injective mapping of the parameter
interval into itself, with $j(0) = 0$.

\medskip\noindent{\em{\bf Claim:} $j(\,.\,)$ is increasing, i.e, (loosely)
  $f$ is orientation-preserving on $B^f_0$.}

\medskip\noindent 
We assume the claim for the
moment. Since $f^{-n}$ converges uniformly to $0$ on $P_R$,
\[
0 < j(t) < t\quad\text{for $0 < t < t_{+}$.}
\]

Recall that $\widetilde{\gamma}$ is the image under $\phi_R^{-1}$ of
an appropriate vertical line and that $\gamma$ is the initial segment
of $\widetilde{\gamma}$, up to $z_1$, its first point of intersection
with $B^f_0$. We choose the labelling of the components of
$\widetilde{\sigma} \setminus \{0\}$ so 
that $z_1 \in \widehat{\sigma}_+$. Then $f^{-1}(z_1)$ is also in
$\widetilde{\sigma}_1$ and -- from the conjugacy to $j(\,.\,)$, the
subarc of $\widetilde{\sigma}$ from $z_1$ to $f^{-1}(z_1)$ with one
end-point included and the other not is a fundamental domain for
the action of $f^{-1}$ on $\widetilde{\sigma}_{+}$. In particular, two
distinct points on this arc have disjoint $f^{-1}$ orbits and hence
distinct images under $\ixp \circ \phi_R$.

It remains to prove the claim. We know that $f^{-1}$ maps
$\widetilde{\sigma}$ into itself and hence maps $\widetilde{\sigma}_+$
either into itself, or into $\widetilde{\sigma}_-$. From injectivity of
$f^{-1}$ on $\widetilde{\sigma}$, we will be done if we show that that
the first alternative holds:
\[
f^{-1} \widetilde{\sigma}_+ \subset \widetilde{\sigma}_+.
\]
By connectivity, this will follow if we show that
$f^{-1}\widetilde{\sigma}_+ \cap \widetilde{\sigma}_+$ is non-empty.

The labelling of the components of $\widetilde{\sigma} \setminus \{
0\}$ is chosen by requiring that $\widetilde{\gamma}$ meets
$\widetilde{\sigma}_+$ before $\widetilde{\sigma}_-$. The closed path
made by following $\widetilde{\gamma}$ from 0 to its first meeting
point $z_1$ with $\widetilde{\sigma}$, then $\widetilde{\sigma}_+$
back to $0$, is a Jordan curve; denote the domain it bounds by
$U$. The first place $\widetilde{\sigma}_-$ meets $\widetilde{\gamma}$
is outside $\overline{U}$; since $\widetilde{\sigma}_{-}$ does not
intersect the boundary of $U$, all of $\widetilde{\sigma}_{-}$ is
outside of $U$. Thus, any continuous path in $P_R$ which starts in $U$
and reaches $\widetilde{\sigma}_{-}$ must intersect $\widetilde{\sigma}_+$
first. It is easy to see, using local theory, that $f^{-1}\gamma$
starts out in $U$. The first place where $f^{-1}\gamma$ intersects
$\widetilde{\sigma}$ is $f^{-1}(z_1)$, so 
\[
f^{-1}(z_1) \in \widetilde{\sigma}_+ \cap f^{-1}\widetilde{\sigma}_+,
\]
completing the proof.

\end{proof}

Let us introduce a model for the dynamics of a map on its immediate
basin. We use the notation $B:\DD\to\DD$ for the quadratic Blaschke product
\begin{equation}
\label{eq:blaschke}
B(z) = \frac{3 z^2 + 1}{3 + z^2}
\end{equation}
We prove:
\begin{thm}\label{th:blaschke-model}
Let $f$ be an analytic function with a normalized simple parabolic
point at the origin. Assume that the immediate basin $B_0^f$ is simply-connected,
and  $f: B_0^f \to B_0^f$ is a degree-2
branched covering map. Then there is a conformal isomorphism $\varphi: B_0^f
\to \DD$ so that
\[
f = \varphi^{-1} \circ B \circ \varphi\quad\text{on $B_0^f$,}
\]
where
$B(z)$ is as in (\ref{eq:blaschke}). In particular, any two analytic maps satisfying the
conditions of the theorem are conformally conjugate on their respective $B_0$'s.
\end{thm}

\begin{proof} 


 For each sufficiently large
$j\in\NN$, let $\sigma_j$ denote that component of the intersection with
$B^f_0$ of the circle of radius $1/j$ about $0$ which contains
$-1/j$. Thus for large $j$, the arc $\sigma_j$ is almost the whole
circle: a small closed arc near the positive real axis has been cut
out. It is immediate that each $\sigma_j$ is a crosscut of $B^f_0$ and
that $\sigma_j$'s for different $j$'s are disjoint (see Figure \ref{fig-crosscut} (a)).
Let $N_j$
be the crosscut neighborhood  of $\sigma_j$
which does not intersect $\sigma_{j-1}$.
 It is evident that 
$\sigma_{j+1}\subset N_j$.
The diameter of $\sigma_j$ is at most $2/j \to 0$, hence  the crosscut neighborhoods $N_j$ form
a fundamental chain. It is clear that the impression of this
fundamental chain contains $0$.
 Recalling that a prime end is an equivalence
class of fundamental chains, we denote by $\widehat{0}$ the prime end
containing the above fundamental chain.

By  Riemann Mapping Theorem, there is a conformal isomorphism
$\varphi$ from $B^f_0$ to $\DD$ mapping the unique critical point of
$f$ in $B^f_0$ to $0$. By Carath\'eodory theory, $\varphi$ extends to
map the set of prime ends of $B^f_0$ homeomorphically to the unit
circle. We can further require that the extension of $\varphi$ sends
$\widehat{0}$ to $1$; with this additional condition, $\varphi$ is
unique. 

We use $z_n \underset{\text{p.e.}}{\to}
  \widehat{0}$ 
as an abbreviation for the assertion that $z_n$ is
  eventually in $N_j$ for each $j$. 
From the construction of the topology on the space of prime ends and
of the Carath\'eodory extension, we extract the following:

\medskip\noindent{\em{\bf Continuity at $\bm{0}$:} Let $(z_n)$ be a sequence in
  $B^f_0$. Then $\varphi(z_n) \to 1$ if and only if $z_n\underset{\text{p.e.}}{\to}\widehat{0}$.
}

\medskip\noindent 
In view of the construction of the
  crosscuts $\sigma_j$, it might appear at first glance that the
  requirement that $z_n \underset{\text{p.e.}}{\to}
  \widehat{0}$ 
 is
  more or less the same as the requirement that $z_n \to
  0$. However, proving this requires more control over the
  structure of $\partial B^f_0$ than we have. We next develop a
  convenient condition which suffices to guarantee convergence of
  $\varphi(z_n)$ to $1$.

Let $P_R$ be a repelling petal which maps under $\phi_R$ to a left
half-plane $\{ u + i v : u < \beta \}$. To give some room for
maneuver, 
 we assume that there is a larger petal $P_R'$ mapping to the half-plane
$\{ u < \beta' \}$ for some $\beta' > \beta$. For the next few
paragraphs, $\phi_R^{-1}$ will denote the inverse of the restriction
of $\phi_R$ to $P_R'$.

We next need to argue that $\Im(\phi_R)$ is bounded on $P_R \setminus
B^f_0$. More precisely, let $M$ be a positive number such that, if $z
\in P_R$ with $$\vert \Im(\phi_R(z)) \vert \geq M,$$ then $z \in B^f_0$ (see Figure \ref{fig-crosscuts} (b)).

\begin{figure}
\centerline{\includegraphics[width=\textwidth]{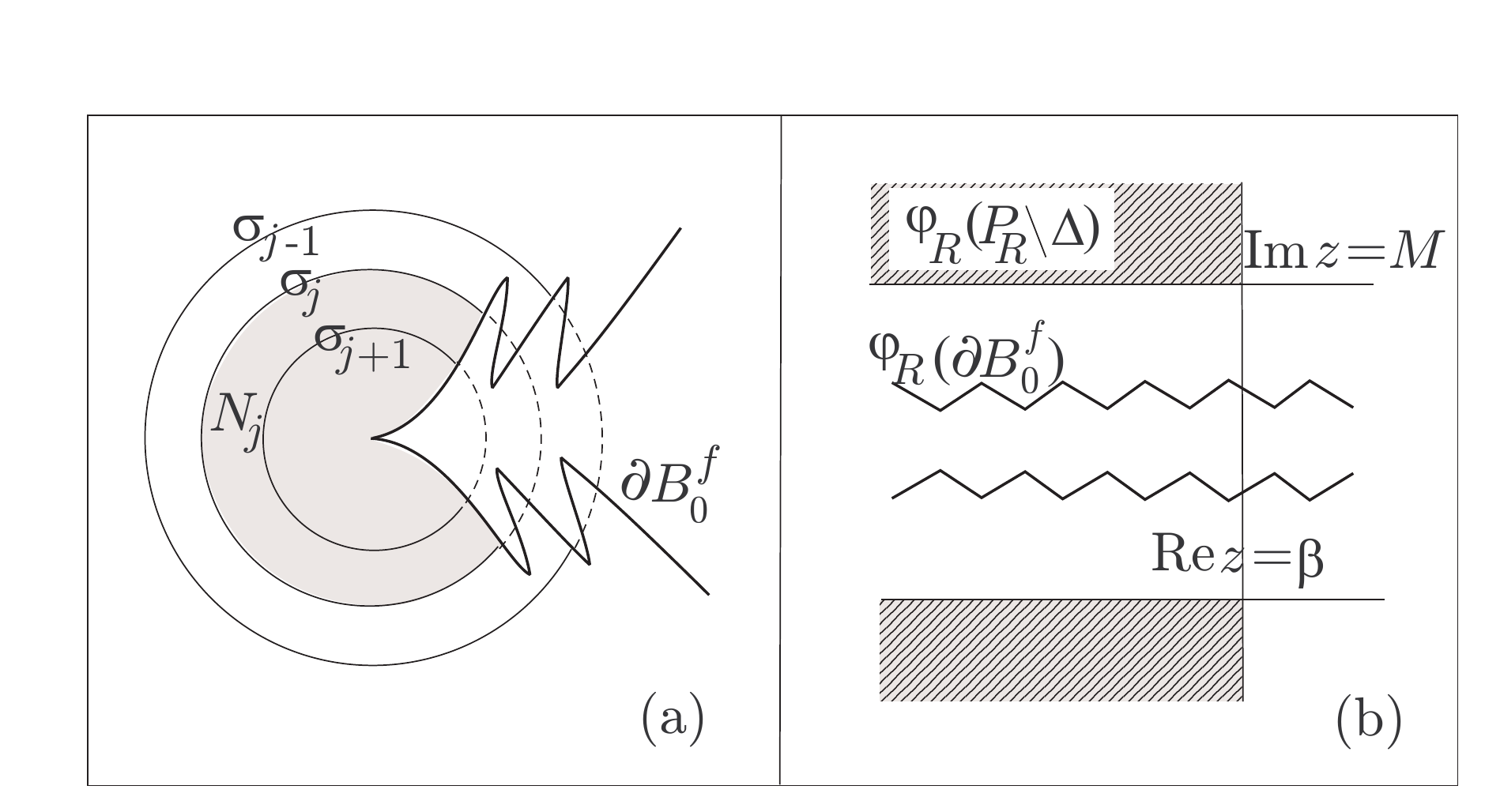}}
\caption{\label{fig-crosscuts}An illustration of the proof of \thmref{th:blaschke-model}}

\end{figure}

\ignore{We look first inside the crescent $P_R \setminus
f^{-1}P_R$, which maps under $\phi_R$ to the strip $\{ \beta-1 \leq
\Im(z) < \beta \}$. By the crude asymptotic estimates on Fatou
coordinates, as $z \to \infty$ inside this strip, $\phi_R^{-1}(z) \to
0$ along the imaginary axis and so -- by local theory -- is eventually
in $B^f_0$. Formulated precisely, what this argument shows is that
there is a real number $M$ so that, if $\beta - 1 \leq \Re(z) <
\beta$  and $\Im(z) \to $ large enough,
then $\phi_R^{-1}(y)$ as
$First: $\Im(\phi_R)$ is continuous on the compact set
\[
\partial B^f_0 \cap \bigl( \overline{P_R} \setminus (f^{-1}P_R \cup \{
0 \}) \bigr),
\]
and hence bounded there. Let $M$ be and upper bound for $\vert
\Im(\phi_R) \vert$ on this set. We claim that
\[
\vert \Im(\phi_R) \vert \leq M \quad\text{on}\quad P_R \cap 
B^f_0.
\]
Suppose not. Then for some $z \in P_R \cap B^f_0$, $\vert
\Im(\phi_R(z)) \vert > M$. But then, for appropriate $n$, $f^n(z) \in
P_R \setminus f^{-1}P_R$, and -- by forward invariance of $B^f_0$ --
$f^n(z)$ is also in $B^f_0$. Since $\Im(\phi_R(f^n(z)) =
\Im(\phi_R(z)$, this contradicts the choice of $M$. 
}

Now let $\Delta := \{ z \in P_R: \vert \Im(\phi_R(z) \vert \leq M
\}$. 


\begin{lem}\label{thm:convergence-in-prime-ends}
If $(z_n)$ is a sequence in $B^f_0 \setminus \Delta$ such that $z_n\to 0$, then $z_n
\underset{\text{p.e}}{\longrightarrow} \widehat{0}$
\end{lem}

\begin{proof}[Proof of lemma \ref{thm:convergence-in-prime-ends}]
Recall the definition of the crosscut neighborhoods $N_j$ above, and
let $D_j$ denote the open disk of radius $1/j$ about $0$. Let
$\gamma_1$ and $\gamma_2$ denote the segments of the upper and lower
edges of $\Delta$ which start at $0$ and extend to the first point
where the edges in question meet the circle bounding of $D_j$. 

Then $\gamma_1$ and $\gamma_2$ are crosscuts of $N_j$; deleting them
splits $N_j$ into 3 subdomains, one of which is bounded by $\gamma_1$,
$\gamma_2$, and a left-facing arc of the circle bounding $D_j$.
All we want to extract from
the above is that this latter domain is contained in $N_j$. On the
other hand, elementary topological considerations, resting on the fact
that the two edges of $\Delta$ are smooth arcs tangent to the positive
imaginary axis at $0$, show that the domain described above is $D_j
\setminus \Delta$ for large enough $j$. In short:
\[
D_j \setminus \Delta \subset N_j\quad\text{for sufficiently large
  $j$.}
\]
Now fix a $j$ large enough so that the above holds. If $z_n$ converges
to $0$ staying outside of $\Delta$, it is eventually in $D_j \setminus
\Delta$; hence, eventually in $N_j$. This holds
for all sufficiently large $j$, so $z_n
\underset{\text{p.e}}{\longrightarrow} \widehat{0}$, as asserted.
\end{proof}

\medskip\noindent We return to the proof of
\thmref{th:blaschke-model}. By assumption, $f: B^f_0 \to B^f_0$ is a
degree-2 branched cover, and $B$ is the conjugate of $f$ on $B^f_0$
under a conformal isomorphism $\varphi: B^f_0 \to \DD$. Thus, $B$ is a
degree-2 branched cover of $\DD$ to itself. It is a standard fact --
following from an easy application of the Schwarz lemma -- that such a
$B$ has the form
\[
B(z) = e^{i \theta} \frac{z^2 - a}{1 - \overline{a}z^2},
\]
where $a \in \DD$ and $\theta \in \Reals$.  By this formula, $B$ extends
analytically to a neighborhood of the closed disk. To complete the
proof, we just need to determine $\theta$ and $a$.

For any $z_0 \in B^f_0$, $f^n(z_0)$ converges to $0$ from the negative
real direction, so $z_n$ is eventually outside $\Delta$, so, by
\lemref{thm:convergence-in-prime-ends} $z_n \underset{\text{p.e}}{\to}
\widehat{0}$, so $\varphi(f^n(z_0)) = B^n(\varphi(z_0)) \to 1$, from
which it follows that $1$ is a fixed point of $B$:
\[
B(1) = 1.
\]
Since $B$ maps the unit circle to itself, preserving orientation,
$B'(1)$ is real and positive. Since the
orbit $B^n(\varphi(z_0)) \to 1$, it is not possible that $B'(1) > 1$.

Now specialize to $z_0 \in B^f_0 \cap P_R$ with $\Im(\phi_R(z_0)) >
M$. Let $f^{-1}$ denote the inverse of the restriction of $f$ to
$P_R$. Then, for any $n \geq 0$, $f^{-n}(z_0) \in P_R$ and
$\Im(\phi_R(f^{-n}(z_0)) = \Im(\phi_R(z_0)) > M$, so $f^{-n}(z_0) \in
B^f_0 \setminus \Delta$. Since $z_0$ is in the repelling petal $P_R$,
$f^{-n}(z_0) \to 0$, so, by \lemref{thm:convergence-in-prime-ends},
$y_{-n} := \varphi(f^{-n}(z_0)) \to 1$. The sequence $(y_{-n})$ is a
backward orbit for $B$: $B(y_{-n-1}) = y_{-n}$. The existence of just
one backward $B$-orbit converging to the fixed point $1$ of $B$
implies
\begin{itemize}
\item $B'(1)$ cannot be $< 1$, so $B'(1)$ must be equal to $1$
\item $B''(1)$ cannot be non-zero; if it were, any backward orbit
  converging to $1$ would have to do so from the positive real
  direction -- in particular, from outside $\DD$ -- whereas our
  sequence $y_{-n}$ is inside $\DD$.
\end{itemize}
We thus have the general algebraic formula for $B(z)$ and the conditions:
\[
B(1) = 1, \quad B'(1) = 1, \quad B''(1) = 0;
\]
routine algebra then shows that $B(z)$ must be as in the statement of
the theorem. 

\end{proof}

\section{Global theory}
\subsection{Basic facts about branched coverings.}
We give a brief summary of relevant facts about analytic branched coverings. A more detailed exposition is
found e.g. in the Appendix E of \cite{Mil}.

For a holomorphic map $f:X\to Y$ between two Riemann surfaces a {\it regular value} is a point $y\in Y$ for which
one can find an open neighborhood $U=U(y)$ such that
$$f^{-1}(U)\overset{f}{\longrightarrow}U$$
is a covering map.
The complement of this set consists of the {\it singular values} of $f$, and will be denoted $\sing(f)$.

By definition, $y\in Y$ is an {\it asymptotic value} of $f$ if there exists a parametrized path 
$$\gamma:(0,1)\to X\text{, such that }\lim_{t\to 1-}f(\gamma(t))=y,$$
and such that $\lim_{t\to 1-}\gamma(t)$ does not exist in $X$. 
We will be concerned with the situation when $X$ is a proper subdomain of the Riemann sphere. In this case, the
non-existence of the limit can be replaced with:
$$\gamma(t)\underset{t\to 1-}{\longrightarrow}{\partial X}.$$
We will denote $\asym(f)\subset Y$ the set of all asymptotic values of $f$. 

Recall that $y_0\in Y$ is called a {\it critical value} (or a {\it ramified point}) of $f$ if 
there exists $x_0\in X$ such that $y_0=f(x_0)$, and the local degree of $f$ at $x_0$ is $n\geq 2$. Thus, in local coordinates,
one has
$$f(x)-y_0=c(x-x_0)^n+\cdots\text{ where }n\geq 2.$$
The point $x_0$ is called a {\it critical point} of $f$ (or a {\it ramification point}).
We denote 
 $\crit(f)\subset X$  the set of the critical points of $f$.
One then has:
\begin{prop}
For a holomorphic map $f:X\to Y$ between Riemann surfaces,
$$\sing(f)=\overline{\asym(f)\cup f(\crit(f))};$$
that is, the set of the singular values of $f$ is the closure of the union of its asymptotic and critical values.
\end{prop}
Recall that a non-constant map $g:S_1\to S_2$ between two Hausdorff topological spaces is called {\it proper} if the preimage of every compact set of $S_2$ is compact in $S_1$.
It is easy to see that:
\begin{prop}
\label{proper1}
Let $f:X\to Y$ be a proper holomorphic map of Riemann surfaces. Let $W$ be a non-empty
connected open subset of $Y$, and let $V$ be
a connected component of $f^{-1}(W)$. Then $f:V\to W$ is also proper.
\end{prop}

For a holomorphic map $f:X\to Y$ define 
the degree of $f$ at $y\in Y$ (denoted $\deg_y(f)$) as the possibly infinite
sum of the number of preimages of $y$ in $X$, counted with
multiplicity.
It is not difficult to see that a proper analytic map has a well-defined local degree:
\begin{prop}
\label{proper2}
If $f:X\to Y$ is a proper analytic map between Riemann surfaces, and $Y$ is connected, then
$\deg_y(f)$ is finite and independent of $y$ (and can be denoted as $\deg(f)$).
\end{prop}
As a consequence, note:
\begin{prop}
If $f:X\to Y$ is a proper analytic map between Riemann surfaces, and $Y$ is connected, then $f(X)=Y$.
\end{prop}

\begin{prop}
\label{proper3}
Let $g:S_1\to S_2$ be a proper continuous map between Hausdorff topological spaces which is everywhere a local
homeomorphism. Then $f$ is a covering map.
\end{prop}

\noindent
In particular, putting together \propref{proper2} and \propref{proper3}, we have
\begin{prop}
\label{proper4}
Let $f:X\to Y$ be a proper analytic mapping between Riemann surfaces. Let $Y$ be connected, and set $d=\deg(f)$. Then
$$f:X\setminus \crit(f)\to Y\setminus f(\crit(f))$$
is a degree $d$ covering.
\end{prop}
In view of the above, a proper analytic map is sometimes called a {\it branched covering of a finite degree}.
Generalizing to the case when local degree is infinite gives the following definition:
\begin{defn}
A holomorphic map $f:X\to Y$ between Riemann surfaces is a {\it branched covering} if every point $y\in Y$ has a connected
 neighborhood $U\equiv U(y)$ such that the restriction of $f$ to each connected component of $f^{-1}(U)$ is a proper map.
\end{defn}

\noindent
As a canonical example, consider the case $X=Y=\hat\CC$. Non-constant rational maps $f$ are clearly branched coverings;
and every branched covering is, in fact, a rational map.

Let us formulate another general lemma which we will find useful:
\begin{lem}
\label{proper5}
Let $f:U\to V$ be a proper analytic map between connected subdomains of $\hat\CC$. Assume that $f$ 
has a single critical value $v\in V$, and that $V$ is a topological disk. Then $U$ is also a topological disk,
and $f$ has only one critical point $u$ in $U$, such that $f^{-1}(v)=\{u\}$.

\end{lem}
\begin{proof}
Set $\hat V\equiv V\setminus \{v\}$ and $\hat U\equiv f^{-1}(\hat V)$.
By \propref{proper4}, the map $f:\hat U\to\hat V$ is a covering. The domain $\hat V$ is homeomorphic to $\DD\setminus \{0\}$.
By the standard facts about coverings of $\DD\setminus\{0\}$, the domain $\hat U$ is also homeomorphic to the punctured disk.
Moreover, denoting $\eta_{\hat V}$ and $\eta_{\hat U}$ conformal homeomorphisms mapping  $\DD\setminus \{0\}$ 
to $\hat V$ and $\hat U$ respectively, we see that 
$$\eta_{\hat V}(f(z))=\text{const}\cdot(\eta_{\hat U}(z))^d\text{ for }d\geq 2.$$
\end{proof}
\subsection{Parabolic renormalization of the quadratic map.}
A key example for our investigation is the quadratic polynomial 
$$f_0(z)=z\mapsto z+z^2.$$
It is well known that
\begin{itemize}
\item The basin of attraction $B^{f_0}=: B_0$ of the parabolic point is
  connected.
\item $f: B_0 \to B_0$ is a branched cover of degree 2.
\end{itemize}
We will prove more detailed or more general versions of these
assertions later -- see, in particular, the proof of
\thmref{cauliflower} -- so we won't repeat the proofs here.

By a trivial computation, $-1/2$ is a critical point of $f_0$; it is
the only one in the finite plane, and the corresponding critical value
is $-1/4$. Let $\phi_A$ be an attracting Fatou coordinate for $f_0$,
defined on all of $B_0$. Using \propref{th:fatou-crit} and taking into
account the surjectivity of $f_0$ on $B_0$, we get

\begin{prop}
The critical points of $\tlphi_A$ are the preimages -- under iterates
of $f_0$ --  of the critical value $-1/4$:
$$\crit(\tlphi_A)=\cup_{n\geq 1}f_0^{-n}(-1/4).$$
All critical points have local degree $2$, and their images are  integer translates of each other:
$$\tlphi_A(\crit(\tlphi_A))=\{\tlphi_A(-1/4)-n, n \geq 1\}.$$
\end{prop}

By standard results of complex dynamics:
\begin{equation}
\label{closure-crit}
\overline{\text{Crit}(\tlphi_A)}=\text{Crit}(\tlphi_A)\cup J(f_0).
\end{equation}
In fact (cf. \cite{Do}), $\phi_A$, mapping $B_0$ to $\CC$, is a branched cover
(but we won't need this result.)

Our next  objective is  the following:
\begin{thm}
\label{quad-extension}
Let $(h^+, h^-)$ be an \'Ecalle-Voronin pair for $f_0$ (i.e., one
constructed with some choice of normalization for $\phi_A$ and
$\phi_R$.) Then the germs $h^+$ (at 0) and $h^-$ (at $\infty$) have
maximal analytic continuations to Jordan domains $\hat
W^+$ and $\hat W^-$ (in the sphere $\hat\CC$).
\end{thm}

\noindent See \figref{fig-first-renorm}.
\thmref{quad-extension} follows from \thmref{th:W-Jordan-domain} and  a well-known folklore result:
\begin{thm}
\label{cauliflower}
The Julia set of the quadratic polynomial $f_0(z)=z+z^2$ is a Jordan
curve. Orbits starting inside this curve converge to $0$; those
starting outside converge to $\infty$.  The dynamics of $f_0$
restricted to the Julia set is topologically conjugate to the
angle-doubling map of the circle;  specifically, there exists a unique
continuous map $\rho:J(f_0)\to\TT$ such that
$$\rho(f_0(z))=2\rho(z)\mod 1.$$
\end{thm}

In the next subsection, we will provide a proof of
\thmref{cauliflower}; while it is fairly standard, the strategy used
will be useful to us later in a more complicated situation.

\subsection{The Julia set of $f_0(z)=z+z^2$ is a Jordan curve.}\label{sec:quadratic-example}
We begin by defining several useful local inverse branches of $f_0$, by choosing the appropriate branches of 
the square root in the formal expression 
\begin{equation}\label{sqrt1}
f_0^{-1}(w)=\sqrt{w+1/4}-1/2.
\end{equation}
We first cut the plane along the positive real ray $R\equiv [-1/4,+\infty)$ and 
denote $$\finv{0}:{\CC\setminus R}\to \HH$$
the inverse branch with non-negative imaginary part.
Similarly, we define 
$$\finv{1}:{\CC\setminus R}\to -\HH$$
to be the inverse branch with negative imaginary part, so that
$\finv{1}\equiv 1 -\finv{0}$.

Finally, we define an inverse branch  $\finvloc$ which fixes the parabolic point $z=0$.
We slit the plane along the ray $(-\infty, -1/4]$, and select the branch of the square root in (\ref{sqrt1})
with non-negative real part. In this way, we get
$$\finvloc:\CC\setminus (-\infty,-1/4]\to \{ \Re z > -1/2 \}.$$
We note that Taylor expansion of $\finvloc$ at $w=0$ begins with $w-w^2+\cdots$, and thus it also has a parabolic fixed point
at the origin.

When needed, we will continuosly extend the three inverse branches
defined above to the edges of the slits. 

\begin{prop}\label{th:2}
The branch $\finvloc$ maps $\CC \setminus (-\infty, 0]$ into itself.
\end{prop}

\begin{proof} As already noted, the image of $\finvloc$ is contained
  in the half-plane $\{ \Re z > -1/2 \}$. Furthermore, since
  $\finvloc$ is a branch of the inverse of $f_0$, if $\finvloc(w) \in
  [-1/2, 0)$, then $w \in f_0([-1/2, 0)) = ([-1/4, 0) \subset
  (-\infty, 0)$, so the image under $\finvloc$ of $\CC \setminus
  (-\infty, 0]$ does not intersect $[-1/2,0)$ and so is contained in
  $\CC \setminus (-\infty, 0]$
\end{proof}
Applying the Denjoy-Wolff Theorem, we obtain
\begin{cor}\label{th:3}
The successive iterates ${\finvloc}^n$ converge to the constant $0$ --
the parabolic point -- uniformly
on compact subsets of $\CC \setminus
(-\infty, 0]$.
\end{cor}
Next, we observe:
\begin{prop}
The intersection
\[
J(f_0) \cap \Reals = \{ -1, 0\}.
\]
\end{prop}
\begin{proof}
Indeed, if $x>0$ then $$f(x)=x+x^2>x\text{, and, furthermore, }
f^n(x)>x+nx^2\to\infty,$$ and hence $(0,\infty)\subset F(f_0)$.
Since $$f_0((-\infty, -1)) = (0, \infty),$$ the same holds for $(-\infty,-1)$.
Finally, $f_0$ maps the
interval $(-1, 0)$ to itself and $f_0(x) > x$ for $x$ in this
interval, so $f_0^n(x) \to 0$. 
Hence, $(-1,0)\subset F(f)$ as well, as claimed.
\end{proof}



Define 
\[
\bA := \{ z : 1/2 < \vert z + 1/2 \vert < 2 \}.\,
\]
i.e., $\bA$ is an annulus, centered at the critical point $-1/2$, with
inner radius $1/2$ and outer radius $2$. 
We collect some elementary facts about this annulus into the following
proposition:
\begin{prop}\label{th:1}
$\-$

\begin{enumerate}
\item If $\vert z + 1/2 \vert \geq   2$ -- i.e., if $z$ is outside the
  annulus -- then $\vert f_0(z) + 1/2 \vert > \vert z + 1/2 \vert$,
  so $f_0(z)$ is again outside the annulus, and $\vert f^n_0(z) +
  1/2 \vert$ goes monotonically to $\infty$.
\item If $\vert z + 1/2 \vert < 1/2$ -- i.e., if $z$ is strictly inside the
  annulus -- then $\vert f_0(z) + 1/4 \vert < 1/4$.
\item If $\vert z + 1/2 \vert < 1/2$, then $f_0^n(z) \to 0$.
\item if $f(z) \in \bA$, then $z \in \bA$, i.e.,
$f_0^{-1}\bA \subset \bA$.
\item $J(f_0)$ is contained in the closure $\overline{\bA}$ of
  $\bA$.
\end{enumerate}
\end{prop}

\begin{proof} Items (1) and (2) follow from elementary estimates. For
  (3): Since the open disk of radius $1/4$ about $-1/4$ is contained
  in the open disk of radius $1/2$ about $-1/2$, this latter disk is
  mapped into itself by $f_0$ and hence by all its iterates
  $f_0^n$. In particular, this sequence of iterates is uniformly
  bounded on the disk in question and so -- by Montel's Theorem -- is
  a normal family. Furthermore, elementary considerations show that,
  for $-1 < x < 0$, $f_0^n(x) \to 0$. It follows then from Vitali's
  Theorem that $f_0^n$ converges to the constant $0$ uniformly on
  compact sets of the open disk $\{ \vert z + 1/2 \vert < 1/2 \}$.

From (1) and (2):  If $z \notin \bA$, then $f_0(z) \notin \bA$, which
is tautologically equivalent to (4). Finally, from (1), if $z$ is
outside $\overline{\bA}$, then $f_0^n(z) \to \infty$, so $z \notin
J(f_0)$, and, from (3). If $z$ is inside $\overline{\bA}$, $f_0^n(z)
\to 0$ so -- again -- $z \notin J(f_0)$.
\end{proof}

To prove \thmref{cauliflower}, we are going to show that
$f_0^{-n} \bA$ are a decreasing sequence of topological annuli which
shrink down to a Jordan curve $J$. Points inside $J$ are attracted to
the parabolic point; those outside are attracted to $\infty$. Thus,
the Jordan curve $J$ contains the Julia set. Our argument also shows
that the preimages (of arbitrary order) of the parabolic point $0$
under $f_0$ are dense in $J$, so the Julia set is the whole of $J$.

We define first
\[
\bA_0 := \{ z \in \bA \,:\, \Im z > 0 \}\quad\text{and}\quad
\bA_1 := \{ z \in \bA \,:\, \Im z < 0 \},
\]
and then, for any $n = 2, \ldots$ and any sequence $i_0 i_1 \ldots
i_{n-1}$ of $n$ 0's and 1's,
\[
\bA_{i_0 i_1 \ldots i_{n-1}} := \{ z \in \bA_{i_0} : f_0^j(z) \in
\bA_{i_j}\quad\text{for}\quad 1 \leq j < n \}
\]
We will refer to the sets $\bA_{i_0 \ldots i_{n-1}}$ as {\em puzzle
    pieces (of level $n$.)} It is easy to see that
\begin{itemize}
\item 
$\bA_{i_0 i_1 \ldots i_{n-1}}$ is decreasing in $n$:
$
\bA_{i_0 \ldots i_{n-1}} \supset \bA_{i_0 \ldots i_n};
$
\item
$
f \bA_{i_0 i_1 \ldots i_n} = \bA_{i_1 \ldots i_n};
$

\item
$
\bA_{i_0 i_1 \ldots i_n} = \finv{i_0} \bA_{i_1 \ldots i_n},
$
and hence\\
$
\bA_{i_0 i_1 \ldots i_{n-1}} = \finv{i_0} \circ \finv{i_1} \circ
\cdots \circ \finv{i_{j-1}}\bA_{i_j i_{j+1} \ldots
  i_{n-1}}\quad\text{for $1 \leq j < n-1$.}
$
\end{itemize}

\begin{figure}
\centerline{\includegraphics[scale=1.0]{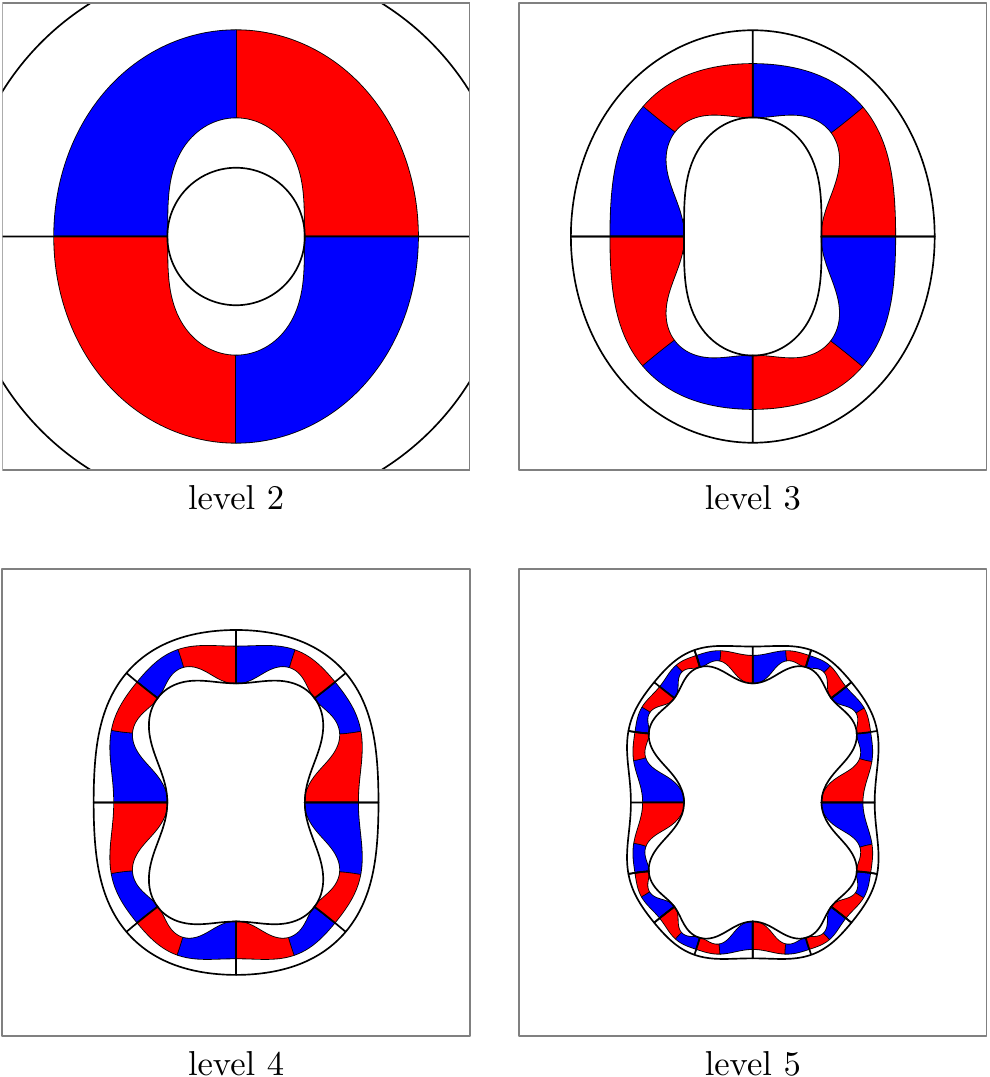}}
\caption{\label{fig-puzzle}There are $2^n$ puzzle pieces of level
$n$. In the picture for level $n$, we also show the outlines of the
puzzle pieces of level $n-1$, each of which contains two pieces of
level $n$. The pieces of level $n$ can be labeled
(successively, counterclockwise, starting from the positive real axis)
with $0, 1, \ldots , 2^{n}-1$; the labeling used in the text is the
binary representation of this one.}
\end{figure}

The main step in the argument is to show that the diameters of the
puzzle pieces $\bA_{i_0 i_1 \ldots i_{n-1}}$ go to zero as $n \to
\infty$, uniformly in $i_0 i_1 \ldots i_{n-1}$. The strategy that we
use is standard; it also
appears in, e.g., \cite{lyubich}, Proposition 1.10.

Concretely, we define 
\[
\rho_n := \sup \left\{ \text{diam}(\bA_{i_0 \ldots i_{n-1}}) : i_0
  \ldots i_{n-1}
\in \{0,1\}^{n} \right\}
\]
Since $\bA_{i_0 \ldots i_{n-1}} \supset \bA_{i_0
  \ldots i_{n}}$, the sequence $\rho_n$ is non-increasing in $n$, so 
\begin{equation}
\label{limrho}
\rho_* :=\lim_{n \to \infty} \rho_n
\end{equation}
exists. 

We will prove:
\begin{prop}\label{th:5}
The limit $\rho_*=0.$  
\end{prop}

The first step is to argue that it suffices to consider binary
sequences one at a time:
\begin{lem}
\label{th:5:1} 

There is an infinite sequence $i_0 i_1 \ldots i_n
  \ldots$ so that
\[
\diam(A_{i_0 i_1 \ldots i_n}) \geq \rho_* \quad\text{for all
  $n$,}
\]
i.e., a {\em single} descending chain of puzzle pieces with diameter
converging to $\rho_*$.
\end{lem}

\begin{proof}[{\it Proof of \lemref{th:5:1}}]
A simple diagonal argument gives the
existence of a sequence $i_0 i_1 \ldots i_k
\ldots$ such that, for all $k$,
\[
\lim_{n \to \infty} \sup \left\{ \text{diam}(\bA_{i_0, \ldots i_{k-1}
  i'_{k} \ldots i'_{k+n-1}})
 : i'_{k} \ldots i'_{k+n-1} 
\in \{0,1\}^{n} \right\} \geq \rho_*
\]
Since $\text{diam}(\bA_{i_0 i_1 \ldots i_{n-1}})$
is non-increasing in $n$, it follows that
\[
\text{diam}(\bA_{i_0 i_1 \ldots i_k}) \geq \rho_*\quad\text{for all
  $k$,}
\]
proving the lemma.
\end{proof}

\begin{proof}[{\it Proof of \propref{th:5}}]
We assume the contrary, that  $\rho_*>0$ and we fix a sequence $i_0 i_1
\ldots$ as in the \lemref{th:5:1}, i.e., so that
$\text{diam}(\bA_{i_0 i_1 \ldots i_{k-1}})$ does not go to zero as $k
  \to \infty$. We consider separately the three cases:
\begin{enumerate}
\item $i_j$ is eventually 0;
\item $i_j$ is eventually 1;
\item $i_j$ takes both values $0$ and $1$ infinitely often.
\end{enumerate}

We start with case (3).  There are then infinitely many $j$'s so that
\[
i_{j} = 0 \quad\text{and}\quad i_{j+1} = 1
\]
Let $j_k$ be a strictly increasing sequence of such $j$'s. Then, for
each $k$,
\[
f_{-j_k} := \finv{i_0} \circ \finv{i_1} \circ \cdots \circ \finv{i_{j_k-1}}
\]
is an analytic branch of the inverse of $f_0^{j_k}$ mapping $\bA_{01}$
bijectively to $\bA_{i_0 \ldots i_{j_k+1}}$. The closure of
$\bA_{01}$ does not intersect the closure of the postcritical set of $f_0$; 
we let $U$ be an bounded simply connected open neighborhood of
$\overline{\bA_{01}}$ disjoint from the postcritical set. Then each
$f_{-j_k}$ has an extension to a branch of the inverse of
$f_0^{j_k}$ defined and analytic on $U$; we denote this extension also by
$f_{-j_k}$.

We next want to argue that the $f_{-j_k}$ are uniformly bounded on $U$
and hence form a normal family there. We can see this as follows: Let
$\hat{U}$ be a neighborhood of $\infty$ which is forward-invariant under
$f_0$ and disjoint from $U$. Then, if $f_{j_k}(z)$
were in $\hat{U}$ (for some $k$ and some $z \in U$), we would have --
since $f_{j_k}$ is an inverse 
branch of $f_0^{j_k}$ and $\hat{U}$ is forward invariant --
\[
z = f_0^{j_k}(f_{-j_k}(z)) \in \hat{U},
\]
which is impossible since by hypothesis $z \in U$ and $U$ is disjoint
from $\hat{U}$. Thus, the sets $f_{-j_k}U$ are all disjoint from
$\hat{U}$, so $(f_{-j_k})$ is uniformly bounded on $U$, as desired.


By Montel's Theorem, there is a subsequence of $(j_k)$ along which $f_{-j_k}$
converges uniformly on compact subsets of $U$. By adjusting the
notation, we can assume that the sequence $(f_{-j_k})$ itself
converges; we denote its limit by $h$.
Since
\[
\text{diam}(f_{-j_k} \bA_{01}) = \text{diam}(\bA_{i_0 \ldots i_{j_k+1}})
\]
does not go to zero as $k \to \infty$, and since $f_{-j_k}$ converges
uniformly on the closure of $\bA_{01}$, the limit $h$ is non-constant.

Let $z_0:=\finv{0}(\finv{1}(-1))\in\bA_{01}\cap J(f_0)$. Since the
Julia set if backward-invariant under $f_0$ and closed,
\[
w_0 := h(z_0) = \lim_{k \to \infty} f_{-j_k}(z_0)
\]
is again in the Julia set. Since $h$ is
non-constant, $h\bA_{01}$ is an open neighborhood of $w_0$; let $W$ be
another open neighborhood whose closure is compact and contained in
$h\bA_{01}$. A straightforward application of Rouch\'e's theorem shows
that 
\[
f_{-j_k}\bA_{01} \supset W\quad\text{for sufficiently large $k$.}
\]
Since $f_{-j_k}$ is a branch of the inverse of $f_0^{j_k}$, it follows
that
\begin{equation}\label{eq:1}
f_0^{j_k}W \subset \bA_{01}\quad\text{for all sufficiently
large $k$.}
\end{equation}

We are going to deduce from (\ref{eq:1}) that
\[
\vert f_0^n(w) + 1/2\vert \leq 2\quad\text{for $w \in W$ and $n \geq
0$,}
\]
which by Montel's theorem
contradicts the fact that $w_0 \in W$ is in the Julia set.
Suppose, therefore, that $\vert f_0^{n}(w) + 1/2 \vert > 2$ for some
$w \in W$ and some $n$.
Then, by \propref{th:1}(1),
\[
\vert f_0^{j_k}(w) + 1/2 \vert > 2 \quad\text{for all $k$ with
$j_k \geq n$.}
\]
Since
\[
\vert z + 1/2 \vert \leq 2\quad\text{for all $z \in \bA$,}
\]
this contradicts (\ref{eq:1}), and so completes the proof in case 3.

We turn next to case (1) above, and deal first with the situation $i_j =
0$ for all $j$. We write
temporarily
\[
\bA^{(n)} := \bA_{\underbrace{0\cdots 0}_{\text{$n$ terms}}}
= (\finv{0})^{n-2}\bA_{00}
\]
Now $\finv{0}$ maps $\bA_{00}$ into itself, and $\finv{0} = \finvloc$
on $\bA_{00}$, so we can write
\[
\bA^{(n)} = {\finvloc}^{n-2} \bA_{00}.
\]
By Corollary \ref{th:3} and local properties of parabolic dynamics near $0$, 
\[
{\finvloc}^{n-2} \to 0\quad\text{uniformly on $\bA_{00}$},
\]
so
\[
\text{diam}(\bA^{(n)}) \to 0\quad\text{as $n \to \infty$.}
\]

Next consider sequences of the form $i_0 i_1 \cdots i_k 0 0 \cdots$,
and use the formula
\[
\bA_{i_0 \ldots i_k \underbrace{0 \cdots 0}_{\text{$n$ terms}}} =
  \finv{i_0} \circ \finv{i_1} \circ \cdots \circ \finv{i_k} A^{(n)}.
\]
The mapping $\finv{i_0} \circ \finv{i_1} \circ \cdots \circ
\finv{i_k}$ extends to be continuous on $\overline{\bA_{00}}$,
and $\text{diam}(A^{(n)}) \to 0$ by what we just proved, so
\[
\text{diam}(\bA_{i_0 \ldots i_k \underbrace{0 \cdots 0}_{\text{$n$
      terms}}}) \to 0\quad\text{as $n \to \infty$.}
\]

A similar argument, using $\finv{1} = \finvloc$ on $\bA_{11}$ shows
that
\[
\text{diam}(\bA_{i_0 \ldots i_k \underbrace{1 \cdots 1}_{\text{$n$
      terms}}}) \to 0\quad\text{as $n \to \infty$.}
\]
%

\end{proof}

Let $\bm{i}=(i_j)_{j=0}^N$ be a finite or infinite sequence ($N\leq\infty$) of $0$'s and $1$'s.
We interpret it as a binary representation of an element of the circle $\RR/\ZZ$:
$$\bin{i}\equiv \sum_{j=0}^N i_j 2^{j+1}\bmod{\ZZ}.$$
For any such dyadic sequence,
\[
\overline{\bA_{i_0}} \supset \overline{\bA_{i_0 i_1}} \supset \cdots \supset
\overline{\bA_{i_0 \ldots i_n}} \supset \ldots
\]
is a nested sequence of compact sets in $\CC$ with diameter going
to 0, so its intersection contains exactly one point, which we denote by
$\hat{z}(\bm{i})$. Since any $\bA_{i_0 \ldots i_{n-1}}$ contains
$g_{i_0} g_{i_1} \cdots g_{i_{n-1}}(-1) \in J(f_0)$, each
$\hat{z}(\bm{i}) \in J(f_0)$. It is immediate from the construction that $\bm{i}
\mapsto \hat{z}(\bm{i})$ is continuous from $\{0,1\}^\NN$ to
$\CC$. It is not injective, but its non-injectivity is of a simple and
familier nature:
\begin{lem}
\label{ambiguity}
Let $\bm{i}=i_0i_1\ldots i_{n-1}$ and  $\bm{i}'=i'_0i'_1\ldots i'_{n-1}$ be two finite dyadic
sequences of equal length such that
$$\overline{\bA_{\bm{i}}}\cap\overline{\bA_{\bm{i}'}}\neq\emptyset.$$
Then either $\bin{i}=\bin{i'}$ or $\bin{i}=\bin{i'}\pm 2^{-n}.$
\end{lem}
The proof is a straightforward induction in $n$ which we leave to the reader.
As a corollary, we get:
\begin{cor}\label{th:7}
Let $\bm{i}$ and $\bm{i'}$ be two infinite dyadic sequences. Then
\[
\hat{z}(\bm{i}) = \hat{z}(\bm{i'})\quad\text{if and only if}\quad
\bin{i} = \bin{i'}
\]
\end{cor}
\begin{proof}
For any
$\bm{i} = i_0 \cdots i_n \cdots$, and any $n$
\[
\vert \bin{i} - \underline{i_0 \cdots i_{n-1}}_2 \vert \leq 2^{-n}.
\]
Hence, if $\bin{i} \neq \bin{i'}$ and $n$ is large enough so that
\[
\vert\bin{i} - \bin{i'} \vert > 3 \times 2^{-n},
\]
then
\[
\vert \underline{i_0 \cdots i_{n-1}}_2 - \underline{i'_0 \cdots i'_{n-1}}_2 \vert
> 2^{-n},
\]
which by \lemref{ambiguity} implies that
$\overline{\bA_{i_0 \cdots i_{n-1}}}$ and $\overline{\bA_{i'_0 \cdots
    i'_{n-1}}}$ are disjoint, and hence that
    $\hat{z}(\bm{i}) \neq \hat{z}(\bm{i'})$.

Conversely, if $\bin{i} = \bin{i'}$, then either $\bm{i} = \bm{i'}$ or
-- possibly after interchanging the two sequences -- $\bm{i}$ has the
form $i_0 i_1 \ldots i_{m-1}0111\ldots$ and $\bm{i'}$ the form $i_0
i_1 \ldots i_{m-1} 1000 \ldots$ and from this it follows easily that
$\overline{\bA_{i_0 \cdots i_{n}}}$ and $\overline{\bA_{i'_0 \cdots
i'_{n}}}$ intersect for all $n$ and hence that $\hat{z}(\bm{i})
= \hat{z}(\bm{i'})$
\end{proof}

Hence, there is an injective mapping $\theta \mapsto z(\theta)$, from
the circle $\Reals/\ZZ$ to $\CC$, such that
\[
\hat{z}(\bm{i}) = z(\bin{i})\quad\text{for all $\bm{i}$.}
\]

\begin{prop}\label{th:6}
The mapping $z(\,.\,)$ is a homeomorphism from the circle onto the
image $J$ of $\bm{i} \mapsto \hat{z}(\bm{i})$. In particular, $J$ is a
Jordan curve.
\end{prop}
\begin{proof}
$z(\,.\,)$ is injective by construction, and its continuity follows
easily from the fact that $\diam(\bA_{i_0 \cdots i_{n-1}}) \to 0$. It
is therefore a homeomorphism by the standard topological fact that an
injective continuous mapping from a compact space (to a Hausdorff
space) is a homeomorphism.
\end{proof}

It remains to show that

\begin{prop}
$J = J(f_0) = \bigcap_{n \geq 0} f_0^{-n} \overline{A}$.
\end{prop}
\begin{proof}
We show
\begin{enumerate}
\item $J \subset J(f_0)$
\item $J(f_0) \subset \bigcap_{n \geq 0} f_0^{-n} \overline{A}$
\item $\bigcap_{n \geq 0} f_0^{-n} \overline{A} \subset J$
\end{enumerate}

\noindent (1) If $z \in J$, then $\{ z \}
= \bigcap \overline{\bA_{i_0 \cdots i_{n-1}}}$ for some sequence $i_0
i_1 \ldots$. Since $A_{i_0 \cdots i_{n-1}}$ contains
$g_{i_0} \circ \cdots \circ g_{i_{n-1}}(-1) \in J(f_0)$, and since
$\diam(A_{i_0 \cdots i_{n-1}}) \to 0$, $z \in \overline{J(f_0)} =
J(f_0).$.

\medskip\noindent (2) This proof has essentially already been given: If
$z \notin \bigcap_{n \geq 0} f_0^{-n} \overline{A}$, then some
$f_0^n(z) \notin \overline{A}$, so -- by the elementary properties of
the basic annulus $\bA$
(\propref{th:1})  -- $f_0^n(z)$ converges either to $0$ or to
$\infty$, so $z \notin J(f_0)$.

\medskip\noindent (3) Assume $z \in \bigcap_{n \geq 0}
f_0^{-n} \overline{A}$, i.e., $f_0^n(z) \in \overline{A}$ for
all $n$. We consider two cases:
\begin{itemize}
\item $f_0^n(z) \notin \Reals$ for all $n$. Put $i_n = 0$ of $1$
according as $f_0^n(z)$ is in the upper or lower half-plane. Then
$z \in \bA_{i_0 \cdots i_{n-1}}$ for all $n$, so $z
= \hat{z}(\bm{i}) \in J$.
\item $f_0^n(z) \in \Reals$ for some $n$. The intersection of
$\overline{\bA}$ with $\Reals$ is the union of the two intervals $[-2.5,
-1]$ and $[0, 1.5]$, and the first of these maps into the
positive real axis, which is mapped into itself. Hence
$f_0^n(z)$ is eventually in $[0, 1.5]$. But the only orbit staying
forever in $[0, 1.5]$ is the fixed point $0$, so $z$ must be a
preimage of finite order of $0$. It is then straighforward to show
that there is a sequence of the form $i_0 i_1 \cdots i_{n-1} 0 0
0 \cdots$ with $z = \hat{z}(\bm{i})$
\end{itemize}
\end{proof}
This completes the proof of \thmref{cauliflower} We note further that
\[
f_0(\hat{z}(i_0 i_1 i_2 \cdots)) = \hat{z}(i_1
i_2 \cdots) \quad\text{for all binary sequences $i_0 i_1 \cdots$.}
\]
from which it follows that:

\begin{prop}\label{th:9}
The map $\theta \mapsto z(\theta)$ conjugates $\theta \mapsto
2 \cdot \theta$ on $\Reals/\ZZ$ to $f_0$ on its Julia set $J(f_0)$.
\end{prop}

\ignore{ 
\begin{proof}[Proof of \thmref{quad-extension}]

Standard considerations of Montel's Theorem imply that the topological disk $B_0\equiv\overset{\circ}{K}(f_0)$ is the basin
of the parabolic fixed point $0$. Thus, an orbit of a point $z_0\in B_0$ eventually enters an attracting petal $P_A$. 
On the other hand, the orbit of a point $z_0\in\hat\CC\setminus \overline{B_0}$ converges to $\infty$, and hence
$$\text{Dom}(H_{f_0})= \ixp\circ \tlphi_R(K(f_0)\cap P_R)\cup\{0,\infty\}=\hat W.$$
Consider the unique real-symmetric analytic conformal map $\psi$ which sends $B_0$ to the slit domain
 $$W\equiv \hat\CC\setminus (\RR_{\geq 0}\cup\{\infty\}),$$
so that $$\psi(-1/2)=-1\text{ and }\underset{x\to 0-\text{ and }x\in\RR}{\lim}\psi(x)=0.$$
 By Carath{\'e}odory Theorem, $\phi$ has a continuous extension to the boundary:
$$J(f_0)\to \RR_{\geq 0}\cup\{\infty\}\subset\hat\CC.$$ By real symmetry, the map
$$h\equiv \psi\circ f_0\circ \psi^{-1}:W\to W$$
extends continuously, and hence, holomorphically, to $\hat\CC$. The quadratic rational
map $h$ has a parabolic fixed point at the origin:
this follows from the fact that on the real axis,
$0$ is attracting for $h$ in the negative direction, and repelling in the positive one. 
It is not difficult to verify that 
\begin{equation}
\label{eqn-h}
h(z)=\frac{z}{(z-1)^2}.
\end{equation}
The conformal conjugacy $\psi$ induces conformal isomorphisms between spaces of orbits:
$$\cC_A^{f_0}\to \cC_A^h,\text{ and }$$
$$\Dom(H_{f_0})=\hat W\to
 \Dom(H_{h})= \ixp\circ \tlphi^h_R(P_R^h\cap W)\cup\{0,\infty\}.$$

In view of real symmetry of $h$, we may assume that the Fatou coordinates $\tlphi^h_R$ and $\tlphi^h_A$ are real-symmetric.
Real symmetry implies that 
$$\ixp\circ \tlphi^h_R(\RR_{>0})=S^1.$$
The Julia set $$J(h)=\RR_{\geq 0}.$$
By Montel's Theorem, every open set $U\cap \RR_{\geq 0}\neq \emptyset$ contains preimages of all points in $\hat\CC$, except at most 
two. Hence, $H_h$ cannot be analytically extended across the ``equator'' $S^1=\ixp\circ \tlphi^h_R(\RR_{>0}),$
and thus
$\Dom(H_h)$ is a union of disks around $0$ and $\infty$. 

This implies that 
the domain $\hat W$ is a union of two topological disks $\hat W^+$ and $\hat W^-$, whose boundaries lie in the projection
$\ixp\circ\tlphi^{f_0}_R(J(f_0))$. Suppose that one of the boundaries is not a Jordan curve; to fix the ideas, assume that
it is $L\equiv \partial \hat W^+$.
 In view of \thmref{cauliflower}, $L$ is locally connected; in view of the above, it is also connected. In view of Carath{\'e}odory theory, there is a point $w\in L$, which is bi-accessible from $\hat W^+$. 
Let $z\in J(f_0)$ be a point with $\ixp\circ \tlphi^{f_0}_R(z)=w$. The map $\ixp\circ \tlphi^{f_0}_R$ is a local homeomorphism at $z$, hence the point $z$ is bi-accessible from the parabolic basin of $f_0$. This contradicts \thmref{cauliflower}, and the proof is thus completed.
\end{proof}

}

\subsection{Covering properties of {\'E}calle-Voronin invariant of $f_0$}
For a parabolic germ $f(z)$ of the form
(\ref{parabolic-normal-form-2}) we will denote $\bm{h}_f$ its {\'E}calle-Voronin invariant. We remind the reader, 
that the pair of germs $\bm{h}_f=(h^+,h^-)$ is defined up to pre- and post-composition with a multiplication by
a non-zero constant.

We approach the discussion of covering properties of  $\bm{h}_{f_0}$ indirectly, by introducing a 
different, more convenient, dynamical model.

By  Riemann Mapping Theorem, there is a conformal isomorphism
$\psi$ from $B_0$ to the cut plane $\CC \setminus (-\infty, 0]$
  sending $-1/2$ (the critical point of $f_0$) to $-1$. These
  conditions do not fix $\psi$ uniquely, but it is not difficult to
  see that $\psi$ can be chosen so that $\psi'(-1/2) > 0$ and, under
  this condition, becomes unique. 
Because of uniqueness, and because $B_0$ and the
    cut plane are invariant under complex conjugation, $\psi$ commutes
    with complex conjugation. It must
    therefore map $(-1,0)$ to $(-\infty, 0)$, these being respectively
    the real points of $B_0$ and those of the cut plane. Since
    $\psi'(-1/2) > 0$, $\psi$ is strictly increasing on $(-1, 0)$, and
    $\lim_{x \to 0^-} \psi(x) = 0$.

Define an analytic mapping  $K$ of the cut plane to itself by
\begin{equation}\label{eq:defn-of-h}
K = \psi \circ f_0 \circ \psi^{-1}.
\end{equation}
So defined, $K$ is a real-symmetric analytic degree-2 branched cover: 
$$K:\CC
\setminus [0, \infty) \to \CC \setminus [0, \infty).$$
 By construction,
    $-1$ is a critical point of $h$, and the corresponding critical value is
    $\psi(-1/4)$.
There is a simple explicit formula for $h$:

\begin{prop}\label{th:determine-h} The mapping $K(z)$ is the Koebe function
\begin{equation}\label{eq:koebe-function}
K(z) = \frac{z}{(z-1)^2}
\end{equation}
\end{prop} 

\begin{proof}
For $x \in [0, \infty)$ and small positive $\epsilon$,
  $K(x+i\epsilon)$ and $K(x-i\epsilon)$ are complex conjugates. As
  $\epsilon \to 0^{+}$, $x+i\epsilon$ converges to the boundary of the
  cut plane, so $\psi^{-1}(x+i\epsilon)$ converges to the boundary of
  $B_0$, so $f_0(\psi^{-1}(x+i \epsilon))$ also converges to the
  boundary of $B_0$, so $K(x+i \epsilon)$ converges to the boundary
  $[0, \infty]$ of the cut plane. 

We apply Schwarz
  reflection to see that $K$ can be extended analytically through
  $[0, \infty)$.
The above argument does not rule out that $K(x+i\epsilon)$ goes to $\infty$. 
To get around this, we recall that
    $K$ is a degree-2 branched cover, and that,  on a small disk about
    the critical point $-1$, it takes on every value near the critical
    value $K(-1)$ twice. Hence, $K(x+i\epsilon)$ stays away from $K(-1)$,
    so Schwarz reflection can be applied without difficulty to the
    function $1/(K(-1) - K(z))$, which is bounded on a neighborhood of
    the positive axis.

Thus, $K$ extends to a function meromorphic on the finite plane. Since
$1/(K(-1) - K(z))$ is bounded near $\infty$, $K$ is also meromorphic at
$\infty$, and hence, it is a rational function. Since it takes each
finite non-real value exactly twice, it has degree 2. It never takes
on a value in $[0,\infty]$ anywhere off $[0, \infty]$ so it must map
$[0,\infty]$ to itself, with each point having two preimages, counted
with multiplicity. As $x$ runs up the positive axis, $K(x)$ starts at
$0$ and runs up, reaching $\infty$ at some finite $a$; it then
runs back down, approaching $0$ as $x \to \infty$. The pole at $a$
must have order 2, so $K(z)$ must have the form
$$K(z)=(bz^2+cz+d)/(z-a)^2.$$ From $K(0) = K(\infty) = 0$, we have $b = d = 0$. Thus,
$$K(z) = c z/(z-a)^2$$ for some constant $c$, which clearly must be
positive. By a simple calculation, a function of this form has only one
(finite) critical
point, at $-a$, so $a = 1$. It remains only to show $c=1$.

We saw
above, from elementary considerations, that $\psi(x) \to 0$ as $x \to
0^{-}$. Hence, for small negative $x$,
\[
K^n(\psi(x)) = \psi(f_0^n(x)) \to 0,
\]
so there are at least some $K$-orbits converging to $0$; this rules out $c >
1$, so we need only show that $c < 1$ is also impossible.

By \thmref{cauliflower}, 
and Carath\'eodory Theorem, $\psi$ has a continuous extension
to the boundary $J(f_0)$ of $B_0$ -- which we continue to denote by
$\psi$. The extension is not a homeomorphism; it identifies complex
conjugate pairs on $J(f_0)$. A useful invertibility property can be
formulated as follows: $J(f_0)\setminus{-1, 0}$ is a disjoint union of
two Jordan arcs, and the extended $\psi$ maps each of these arcs
homeomorphically to $(0, \infty)$. Note that the conjugation equation,
written in the form
\[
K \circ \psi = \psi \circ f,
\]
extends by continuity to $\overline{B_0}$. Because $f_0$ on $J(f_0)$
is conjugate to $s \mapsto 2s$ on $\TT$, there
exist infinite backward orbits for $f_0$ in $J(f_0)$ which converge to
$0$, i.e., sequences $z_0$, $z_{-1}$, $z_{-2}$, \dots with $z_{-n+1} =
f(z_{-n})$ and $z_{-n} \to 0$. Then $\psi(z_{-n}) \to 0$ and
$K(\psi(z_{-n}) = \psi(z_{-n+1})$, that is, $(\psi(z_{-n})$ is a backward
orbit for $K$ converging to $0$. This rules out $c < 1$, so the only
remaining possibility is $c = 1$, so the formula
(\ref{eq:koebe-function}) is proved.
 
\end{proof}

\noindent
We note several other contexts in which 
the Koebe function (\ref{eq:koebe-function}) comes up. It is, of course,  the
essentially unique function univalent on the unit disk which saturates
the Koebe distortion estimates. (See, for example, [Co], Theorem 7.9
and p. 31.)

Furthermore, and perhaps more instructively for our study, one can deduce from the above
discussion that $K(z)$ is the {\it conformal mating} (see \cite{Do1})
of the map $f_0$ with the Chebysheff quadratic polynomial $f_{-2}(z)=z^2-2.$

We next develop an important piece of technique:

\begin{prop}\label{th:conjugate-on-B_zero}
Let $f_1$, $f_2$ each have a normalized simple parabolic point at $0$,
and assume that their restrictions to their respective principal
basins are conformally conjugate, i.e., that there is a conformal isomorphism
$\varphi: B^{f_1}_0 \to B^{f_2}_0$ so that
\begin{equation}\label{eq:conjugate-on-B_zero}
f_1 = \varphi^{-1} \circ f_2 \circ \varphi\quad\text{on
  $B^{f_1}_0$.}
\end{equation}
Assume further that $\varphi$ and $\varphi^{-1}$ extend continuously
to the boundary point $0$ of $B^{f_1}_0$ respectively
$B^{f_2}_0$. Then there is a conformal isomorphism $\psi: W^+_{f_1}
\to W^+_{f_2}$, sending $0$ to $0$, and a non-zero constant $v$ so
that
\[
h^+_{f_1} = v \cdot h^+_{f_2} \circ \psi.
\]
\end{prop}

\medskip\noindent We begin with a few simple
remarks. From the continuity of $\varphi$ at $0$, it follows
that $\varphi(0) = 0$:
\[
\varphi(0) = \lim_{n \to \infty} \varphi(f_1^n(z)) = f_2^n(\varphi(z)) =
0\quad\text{for any $z \in B^{f_1}_0$,}
\]
It is easy to verify, again making use of continuity of
$\varphi$ at $0$, that $\varphi$ maps attracting petals for $f_1$ to
attracting petals for $f_2$, and that $\varphi^{-1}$ maps attracting
petals for $f_2$ to attracting petals for $f_1$.  For the proposition
to make sense, we have to have chosen normalizations for attracting and
repelling Fatou coordinates for the two mappings, but the choices are
obviously immaterial. We fix -- however we like -- the normalization
of $\phi^{(2)}_A$. Then $\phi_A^{(2)}
\circ \varphi$ is an attracting Fatou coordinate for $f_1$, so, by the
uniqueness of Fatou coordinates (\propref{uniqueness-fatou}),
$\phi_A^{(2)} \circ \varphi - \phi_A^{(1)}$ is constant. For
notational simplicity, we may assume
\[
\phi_A^{(1)} = \phi_A^{(2)} \circ \varphi;
\]
this has the effect of making the constant $v$ in the proposition
equal to one.

The situation with repelling Fatou coordinates is more
complicated. It is not true that $\varphi$ maps repelling petals to
repelling petals, since $\varphi$ is
only available on $B^{f_1}_0$ and repelling petals are not
contained in $B_0$. What we have to work with instead is the weaker
observation that $\varphi$ maps $f_1$-orbits coming out from $0$ in
$B^{f_1}_0$ to the same kind of orbits for $f_2$.
To exploit this observation, we need
to develop some technique for a single function with a simple
parabolic fixed point.

Let $f$ have the form (\ref{asym-expansion-3}); let $P_R$ be a
repelling petal for $f$; and write $f^{-1}$ for the inverse of the
restriction of $f$ to $P_R$. We define 
\[
\widetilde{P_R} := \{ z \in P_R: f^{-n}(z) \in B^f_0\text{ for $n = 0, 1, \ldots$} \}
\]


\medskip\noindent In the usual way, it is straightforward to show that
the image under $\ixp \circ \phi_R$ of $\widetilde{P_R}$ is
independent of the petal $P_R$. We denote this image by
$\widetilde{W}$; since $\widetilde{P_R} \subset P_R \cap B^f$,
$\widetilde{W}$ is contained in $\cD(\bm{h}_f)$. 

\begin{lem}\label{thm:tilde-P_R}
1. Let $t \mapsto \tau(t), 0 \leq t \leq 1$ be a continuous arc in $P_R
\cap B^f$ with $\tau(0) \in \widetilde{P_R}$. Then $\tau(1) \in
\widetilde{P_R}$.

\smallskip\noindent 2. $\widetilde{P_R}$ is open (so its image
$\widetilde{W}$ under $\ixp \circ \phi_R$ is also open.)

\smallskip\noindent 3. A connected component of $\widetilde{W}$ is also
a connected component of $\cD(\bm{h})$. In other words, a component
of $\cD(\bm{h})$ is either disjoint from $\widetilde{W}$ or contained
in it.

\smallskip\noindent 4. If there
is a continuous path $\tau$ in $P_R \cap B^f_0$ from $z$ to
$f^{-1}(z)$, then $z \in \widetilde{P_R}$.

\smallskip\noindent 5. The components $W^+ \setminus \{ 0 \}$ and $W^-
\setminus \{\infty\}$ of $\cD(\bm{h})$ are contained in $\widetilde{W}$. 
\end{lem}

\begin{proof} 1. If $\tau(1) \notin \widetilde{P_R}$, there must be an
  $n$ so that $f^{-n}(\tau(1)) \notin B_0^f$. Fix such an $n$, and let
\[
t_0 := \inf \{ t : f^{-n}(\tau(t)) \notin B_0^f \}.
\]
Since $\tau(\,.\,)$ and $f^{-1}$ are continuous and $B^f_0$ is open,
\[
f^{-n}(\tau(t_0)) \in \partial B^f_0 \subset \partial B^f,
\]
so in particular $f^{-n}(\tau(t_0)) \notin B^f$. But $B^f$ is
backward-invariant under $f$ and the path $\tau$ is by assumption in 
$B^f$, so this is a contradiction.

\medskip\noindent 2. If $z_0 \in \widetilde{P_R}$ and $z$ is near
enough to $z_0$, then the straight-line segment $[z_0, z]$ is in $P_R
\cap B^f$ and so -- by 1. -- in $\widetilde{P_R}$.

\medskip\noindent 3. Since $\widetilde{W}$ is contained in
$\cD(\bm{h})$, each component of $\widetilde{W}$ is contained in a
component of $\cD(\bm{h})$. Let $U$ be a component of $\widetilde{W}$
and $V$ the component of $\cD(\bm{h})$ which contains it. If $U$ is
not all of $V$, the relative boundary must be non-empty, i.e., there
must exist a $w_0 \in (\partial U) \cap V$. This $w_0$ is in
$\cD(\bm{h})$, so it can be written as $w_0 = \ixp \circ \phi_R(z_0)$
with $z_0 \in P_R \cap B^f$. Let $D$ be an open disk about $z_0$,
small enough to be contained in $P_R \cap B^f$ and also so that $\ixp
\circ \phi_R$ maps it homeomorphically into $V$. Since $\ixp \circ
\phi_R(D)$ contains points of $\widetilde{W}$, $D$ contains points of
$\widetilde{P_R}$, and it then follows from 1. that $D \subset
\widetilde{P_R}$, so $z_0 \in \widetilde{P_R}$, so $w_0 \in
\widetilde{W}$, contradiction.

\medskip\noindent 4. We argue by induction on $n$ that $f^{-n} \tau
\subset B^f_0$ for all $n$. This is true by hypothesis for
$n=0$. Suppose it is true for $n$ but not for $n+1$. Then
$f^{-(n+1)}(\tau(0)) = f^{-n}(\tau(1)) \in B^f_0$, but $f^{-(n+1)}
(\tau(t_0)) \in \partial B^f_0$ for some $t_0$. Then
$f^{-(n+1)}(\tau(t_0)) \notin B^f$, which contradicts $\tau(t_0) \in
B^f$ by backward invariance of $B^f$. The contradiction proves $f^{-n}
\tau \subset B^f_0$ for all $n$, so, in particular, $f^{-n}(\tau(0)) =
f^{-n}(z) \in B^f_0$ for all $n$, i.e., $z \in \widetilde{P_R}$.

\medskip\noindent 5. Since the resulting domain $\widetilde{W}$
does not depend on the repelling petal used to construct it, we may
assume that $P_R$ is ample. Then a sector $\{ z : 0 < \vert z \vert <
\rho, \pi/2 - \delta < \Arg(z) < \pi/2+\delta \}$, with $\rho$ and
$\delta$ small enough, is contained in $P_R \cap B^f_0$. By 4. there
is a sector of this form but with smaller $\delta$ and $\rho$ which is
contained in $\widetilde{P_R}$, and, by the proof of
\lemref{th:W-nbhd-of-0}, the image under $\ixp \circ \phi_R$ of this
smaller sector is a punctured neighborhood of $0$. Thus,
$\widetilde{W}$ intersects $W^+ \setminus \{ 0 \}$, so, by 3,
$W^+\setminus \{ 0 \}$ is one of the components of $\widetilde{W}$. The
proof for $W^- \setminus \{ \infty \}$ is similar.

\end{proof}

We return to the proof of \propref{th:conjugate-on-B_zero}; we are now
now ready to construct the mapping $\psi$. We fix repelling petals
$P^{(1)}_R$ for $f_1$ and $P_R^{(2)}$ for $f_2$. The situation is the
familiar one: We need these petals to make a construction, but the
objects constructed turn out not to depend on them. We let
$\phi^{(1)}_R$ be a repelling Fatou coordinate for $f_1$ defined at
  least on $P^{(1)}_R$, and similarly $\phi^{(2)}_R$ for $f_2$. Let $w_0
    \in \widetilde{W_1}$, and write $w_0 = \ixp \circ \phi_R^{(1)}(z_0)$
    with $z_0 \in \widetilde{P^{(1)}_R}$. Then $z_{-n} := f^{-n}(z_0)$
    is an outcoming orbit in $B^f_0$. Because $\varphi$ conjugates
    $f_1$ to $f_2$ (on their respective principal basins) and is
    continuous at $0$, $y_{-n} = \varphi(z_{-n})$ is an outcoming
    orbit for $f_2$. For sufficiently large $n$, $y_{-n}$ is in
    $P^{(2)}_R$. Take any $n_0$ so that $y_{-n} \in P^{(2)}_R$ for $n
    \geq n_0$, and set
\[
\psi(w_0) = \ixp \circ \phi^{(2)}_R(y_{-n_0}) \in \widetilde{W_2}
\]
It is easy to check that the result does not depend on the choices
of repelling petals, preimage $z_0$, or number $n_0$ of steps back
before applying $\ixp \circ \phi^{(2)}_R$, so we have constructed a
mapping from $\widetilde{W_1}$ to $\widetilde{W_2}$. Each of the steps
\[
w_0 \mapsto z_0 \mapsto z_{-n_0} = f^{-n_0}(z_0) \mapsto y_{-n_0}
\mapsto \psi(w_0) = \ixp \circ \phi^{(2)}_R(y_{-n_0})
\]
can be done in a neighborhood of each of the points involved by
applying a local conformal isomorphism, so $\psi(\,.\,)$ is a local
conformal isomorphism. Carrying out the same construction with $f_1$
and $f_2$ interchanged and $\varphi^{-1}$ in the role of $\varphi$
produces a mapping from $\widetilde{W_2}$ to $\widetilde{W_1}$ which
inverts $\psi$, so $\psi$ is a global conformal
isomorphism. Furthermore,
\[\begin{split}
\bm{h}_2(\psi(w_0)) &= \ixp \circ \phi^{(2)}_A(y_{-n_0}) = \ixp \circ
\phi^{(2)}_A(\varphi(y_{-n_0}))\\
&= \ixp \circ \phi^{(1)}_A(z_{-n_0}) = \ixp \circ \phi^{(1)}_A(z_0) =
\bm{h}_1(w_0)
\end{split}\]

The mapping $\psi$ just constructed is a more global version of the
one in the proposition. To complete the proof of the proposition, we
need to show
\begin{enumerate}
\item $\psi$ maps $W^+_1 \setminus \{ 0 \}$ to $W^+_2 \setminus \{ 0 \}$.
\item $\psi$ extends analytically to map $0$ to $0$
\end{enumerate}

\medskip\noindent Proof of (1): Since $\psi: \widetilde{W_1} \to
\widetilde{W_2}$ is a homeomorphism, it maps components to
components. We know from 3 of \lemref{thm:tilde-P_R} that $W^+_1
\setminus \{ 0 \}$ is a component of $\widetilde{W_1}$; we have only
to show that the component it maps to is $W_2^+ \setminus \{ 0 \}$. To
do this, choose ample petals $P^{(1)}_A$ and $P^{(1)}_R$ for $f_1$,
and let $t \mapsto \lambda_1(t)$ be an arc in $P^{(1)}_A \cap
P^{(1)}_R$, ending at $0$, which is mapped by $\phi^{(1)}_A$ to an
upward vertical ray.  By the asymptotic estimate for Fatou
coordinates, $\lambda$ projects under $\ixp \circ \phi^{(1)}_R$ to an
arc $\widetilde{\lambda_1}$ in $\widetilde{W_1}$ ending at $0$. Let
$\lambda_2 := \varphi \circ \lambda_1$. It is then immediate that
$\lambda_2$ maps under $\phi^{(2)}_A$  to an
upward vertical ray.

\medskip\noindent{\em {\bf Claim:} $\lambda_2$ is tangent to the positive
imaginary axis at $0$.}

\medskip\noindent Accepting the claim for the moment: We can then
project $\lambda_2$ under $\ixp \circ \phi^{(2)}_R$ to an arc
$\widetilde{\lambda_2}$ in $\widetilde{W_2}$. By the asymptotic
estimates on Fatou coordinates, $\widetilde{\lambda_2}$ ends at
$0$. By the construction of $\psi$,
\[
\widetilde{\lambda_2} = \psi \circ \widetilde{\lambda_1}.
\]
Thus: $\psi(W^+_1 \setminus \{ 0 \})$ intersects $W^+_2$ and hence, by
connectedness, $\psi(W^+_1 \setminus 0) = W^+_2 \setminus \{0\}$.

\medskip\noindent
Proof of claim: It would be immediate if we knew that a terminal
subarc of $\lambda_2$ is contained in an attracting petal. Rather than
prove this directly, we argue as follows. Let $P^{(2)}_A$ be an
attracting petal for $f_2$ which maps under $\phi^{(2)}_A$ to a right
half-plane. Since $\lambda_2$ is contained in $B^f$, there is an $n$
so that $f_1^n(\lambda_s(0)) \in P^{(2)}_A$. Since $\Re(\phi^{(2)}_A)$
is constant on $\lambda_2$, $\lambda_2(t)$ stays in $P^{(2)}_A$, and,
as already noted, approaches $0$ as $t \to 1^{-}$. By the crude
estimates on Fatou coordinates, $f_2^n \circ \lambda_2$ is tangent to
the positive imaginary axis at $0$. Applying an analytic local
inverse of $f_2$ $n$ times proves the claim.

\medskip\noindent Proof of (2): This can be proved in a
straightforward way using the crude local estimates on Fatou
coordinates, but it is quicker to argue that
\[
\psi = (h^+_2)^{-1} \circ h^+_1\quad\text{on a punctured neighborhood
  of $0$},
\]
(where $(h^+_2)^{-1}$ means an analytic local inverse of $h^+_2$,
which exists since $h^+_2$ has a simple zero at $0$)
and to note that the right-hand side is analytic at $0$. To prove the
preceding equation, we invoke the general formula $h^+_1 = h^+_2
\circ \psi$ and note that, by (1), $\psi$ and $(h_2^{+})^{-1} \circ
h_1^{+}$ agree at points of the form $\widetilde{\lambda_1}(t)$ for
$t$ near enough to 1. The assertion then follows by analytic continuation. The proof of 
\propref{th:conjugate-on-B_zero} is thus completed. $\Box$

\begin{defn}\label{defn-big-H}
For $K$ as in (\ref{eq:koebe-function}), let us select a real symmetric Fatou coordinate $\phi^K_R$, 
and a Fatou coordinate $\phi^K_A$ so that
$$\bm{h}_K'(0)=1\text{ and }\cD(h^+_K)=\DD,\;\cD(h^-_K)=\CC\setminus\DD\text{, where }\bm{h}_K=(h^+_K,h^-_K).$$
We will denote thus normalized map $\bm{h}_K$  by $\bm{k}=(k^+,k^-)$ for the remainder of this paper.
\end{defn}
In view of the above discussion, we note:
\begin{prop}
\label{straighten}
There exist $v\neq 0$ and $\theta\in\RR$ such that 
$$\bm{h}_{f_0}=v\cdot \bm{k}\circ \psi(e^{2\pi i\theta}z),\text{ where }\psi:\DD\to\hat W^+$$
is the conformal Riemann mapping, which satisfies $\psi(0)=0,$ $\psi'(0)>0$.
\end{prop}


We show:
\begin{thm}
\label{covering-properties}
$\bm{h}_{f_0}$ has
\begin{itemize}
\item exactly one critical value $v:=\ixp\circ\tlphi_A(-1/4)$
\item exactly two asymptotic values $0$ and $\infty$.
\end{itemize}\end{thm}
\begin{pf}
To show that the above-defined $v$ is the only possible critical value of
$\bm{h}_{f_0}$,
we write
\[
\bm{h}_{f_0} = \bigl( \ixp \circ \phi_A \bigr)\circ \bigl( \ixp \circ
\phi_R\bigr)^{-1},
\]
for an appropriate branch of the inverse. Whatever inverse branch is
used, its derivative does not vanish. Hence, a critical value of
$\bm{h}_{f_0}$ is a critical value $\bigl( \ixp \circ \phi_A \bigr)$, i.e.,
the image under $\ixp(\,.\,)$ of a critical value of $\phi_A$.  Since
$-1/4$ is the unique critical value of $f_0$, it follows from
\propref{th:fatou-crit} that the critical values of $\phi_A$ have the
form $\phi_A(-1/4) - n$, $n \geq 1$. Since $\ixp(\,.\,)$ is
1-periodic, $v$ is the unique critical value of $\ixp \circ \phi_A$.

We turn to determining the asymptotic values of $\bm{h}_{f_0}$. We show
that the asymptotic values of the restriction $\bm{h}_{f_0}|_{\hat W^+}$
are $\{0,\infty\}$; the argument for $\bm{h}_{f_0}|_{\hat W^-}$ is
identical.  To simplify geometric considerations, we invoke
\propref{straighten}, and present the proof for $\bm{k}$.

We need a number of elementary facts about $K$ which we summarize below for ease of reference:

\begin{prop}\label{th:h-elem}
The Koebe function $K(z) = z/(1-z)^2$ is a degree-2 ramified cover $\hat{\CC} \to
  \hat{\CC}$; its critical values are $-1/4$ and $\infty$; the
  corresponding critical points are $-1$ and $+1$. The interval $[0,
    \infty]$ is fully invariant under $h$; the interval $(-\infty, 0]$
    is forward invariant. 
\end{prop}

\ignore{
\begin{proof}
That $h$ is a degree-2 ramified cover is standard, and the
determination of its critical points and values is elementary. It
follows immediately from the definition that $[0,\infty]$ and
$(-\infty, 0]$ are mapped into themselves by $f$. From elementary
  considerations, $h$ maps $(0,1)$ to $(0, \infty)$ (increasing) and
  $(1, \infty)$ to $(0, \infty)$ (decreasing). Thus, every $z \in (0,
  \infty)$ has two primages in $(0, 1) \cup (1, \infty)$. $0$ also has
  two preimages -- itself and $\infty$, both in $[0,
    \infty]$. $\infty$ has only the single preimage $1$, but it has
  multiplicity 2. Combining all these remarks: Every point of
  $[0,\infty]$ has two preimages -- counted with multiplicity -- in
  $[0,\infty]$, so -- since $h$ has degree 2 -- it can't have any
  preimage outside of $[0,\infty]$, i.e., $h^{-1}[0, \infty] \subset [0,
    \infty]$, i.e., $[0, \infty]$ is backward invariant.
\end{proof}

\medskip\noindent{\bf Remark:} We showed above -- as a by-product of
a complicated analysis resting on the fact the the Julia set of $f_0$ is a
Jordan curve -- that every $h$-orbit with initial point off
the positive real axis is attracted to the parabolic point $0$. It may
be worth remarking that this result can be given a simple and direct
proof as follows: Since $[0,\infty]$ is backward invariant, $\CC
\setminus [0,\infty)$ is mapped into itself by $h$, and hence by all
  iterates of $h$. By Montel's theorem, the sequence $h^n$ of iterates
  is a normal family on $\CC \setminus [0,\infty)$. Furthermore, by
    elementary considerations, $h^n(x) \to 0$ for all $x \in (-1, 0)$,
    so it follows from Vitali's theorem that $h^n \to 0$ uniformly on
    compact subsets of $\CC \setminus [0,\infty)$

}

\newcommand{\hinv}{K^{-1}}
\medskip\noindent The interval $[-\infty, -1/4]$ contains both of the
critical values of $h$, so $h$ is a cover in the strict sense above
$\CC \setminus (-\infty, -1/4]$. In other words: $h$ admits two
    analytic right-inverses defined on the cut plane $\CC \setminus
    (-\infty, -1/4]$. One of these inverse branches sends $0$ to $0$;
  the other sends $0$ to its other preimage $\infty$. At the cost of
  introducing a little ambiguity, we will write simply $\hinv$ for the
  inverse branch which is regular at $0$. A straightforward elementary
  calculation shows that the preimage of $(-\infty, 0]$ under $h$ is
the unit circle, so the image of $\hinv$ is the unit disk $\DD$.

\begin{prop}\label{th:hinverse}
The inverse branch
$\hinv$ maps the open upper (respectively lower) half-plane into
  itself. It also maps $\CC \setminus (-\infty, 0]$
into itself. 
The sequence of iterates $(\hinv)^n$ converges uniformly to $0$ on compact
subsets of $\CC \setminus (-\infty, 0]$. Any repelling Fatou
  coordinate extends to a univalent analytic function on $\CC
  \setminus (-\infty, 0]$.
\end{prop}

\begin{proof} 
The proof may be given as an application of Denjoy-Wolff Theorem, however, 
a more direct argument is possible in our case.


Elementary considerations show that the upper  half-plane
is mapped into itself by $\hinv$, and similarly for the lower half-plane.
Furthermore,
$$\hinv:\CC \setminus (-\infty, 0]\longrightarrow\CC \setminus (-\infty, 0].$$
 By Montel's theorem, the sequence of iterates $(\hinv)^n$
  is a normal family on $\CC \setminus (-\infty,
  0]$. Elementary considerations show that $(\hinv)^n(x) \to 0$ for
    $x$ real, positive, and small. By Vitali's theorem, $(\hinv)^n$
    converges uniformly to $0$ on compact subsets of $\CC \setminus
    (-\infty, 0]$

The preceding assertion says that
\[
B^{\hinv}_0 = \CC \setminus (-\infty, 0]
\]
Note that any $\hinv$-orbit starting in $(-1/4, 0)$ gets in finitely many
  steps into $(-1, -1/4]$, which is outside $\Dom(\hinv)$. Thus,
    $(-1/4,0)$, although in $\Dom(\hinv)$, is outside of
    $B^{\hinv}$.
Using repeatedly the functional equation
\[
\phi_R(\hinv(z)) = \phi_R(z) - 1,
\]
we see that any repelling Fatou coordinate 
extends (uniquely) to a function defined on
all of $\CC \setminus (-\infty, 0]$ and satisfying the above
  functional equation everywhere. The extended function is analytic;
  from the univalence of $\hinv$ and the fact that repelling Fatou
  coordinates are univalent on small petals, the extended function is
  also univalent. 
\end{proof}

Next we introduce -- in the context of the particular mapping $K$ -- a
useful dynamically-defined family of arcs which come up, with minor
variations, at a number of places in this work. Let $\alpha$ be a real
number so that
\begin{equation}\label{eq:alpha}
\Re(\phi_A(-1)) < \alpha < \Re(\phi_A(-1/4)) \, (= \Re(\phi_A(-1))+1.)
\end{equation}

\begin{prop}\label{th:alpha}
There is an attracting petal $P$ which maps under $\phi_A$ to the
right half-plane $\{ u + i v : u > \alpha \}$.
\end{prop}

For $\alpha$ a sufficiently large real number, this
proposition is part of the local theory of parabolic fixed
points. Showing that, in the present situation, any $\alpha >
\Re(\phi_A(-1))$ is large enough depends on some covering properties
of $K$. We omit the proof here.

The boundary of $P$ can be written as $\{0\} \cup \gamma$, where
$\gamma$ is an arc in the cut plane $\CC \setminus [0,\infty)$ which
  approaches the parabolic point $0$ at both ends. The arc $\gamma$ is
  a component of the preimage under $\phi_A$ of the vertical line $\{
  \alpha + i v : -\infty < v < \infty \}$; we give it the
  counterclockwise orientation, corresponding to the decreasing
  orientation for the vertical line: $\Im(\phi_A(z))$ runs from
  $+\infty$ to $-\infty$ as $z$ traverses $\gamma$ in the
  counterclockwise sense. This can be seen using the crude asymptotic
  formula $\phi_A(z) \approx -1/z$, valid for small $z$ not too near to the
  positive real axis.

We will speak of an arc $s \mapsto \sigma(s)$, defined and continuous
for $0 < s < 1$ -- i.e., without beginning or end point -- as an {\em
  open} arc. We say that such an arc {\em starts at} $z_0$ if
$\lim_{s \to 0^+} \sigma(s) = z_0$ and that it {\em ends at} $z_1$ if $\lim_{s
  \to 1^{-}}\sigma(s) = z_1$.

\begin{prop}\label{th:preimages-of-gamma}
With $\gamma$ as above:
\begin{enumerate}
\item the preimage of $\gamma$ under $K^n$ is a disjoint union of
  $2^n$ smooth arcs in the cut plane. Each of the component arcs
  starts from and ends at a preimage of the parabolic point $0$. The
  complex conjugate of a component arc of $K^{-n}\gamma$ is a
  component arc.
\item the component arcs of $K^{-n}\gamma$ do not intersect the real
  axis; each of them lies either entirely in the upper half plane or
  entirely in the lower half plane.
\item one of the component arcs of the preimage of $\gamma$ under $K$
  -- equipped with the pullback of the orientation of $\gamma$ under
  $h$ -- lies in the upper half plane, starts at $0$, and ends at
  $\infty$. The other component is the complex conjugate of this one,
  with its orientation reversed. It lies in the lower half plane and
  runs from $\infty$ to $0$.
\end{enumerate}
\end{prop}

\begin{figure}
\centerline{\includegraphics[width=0.7\textwidth]{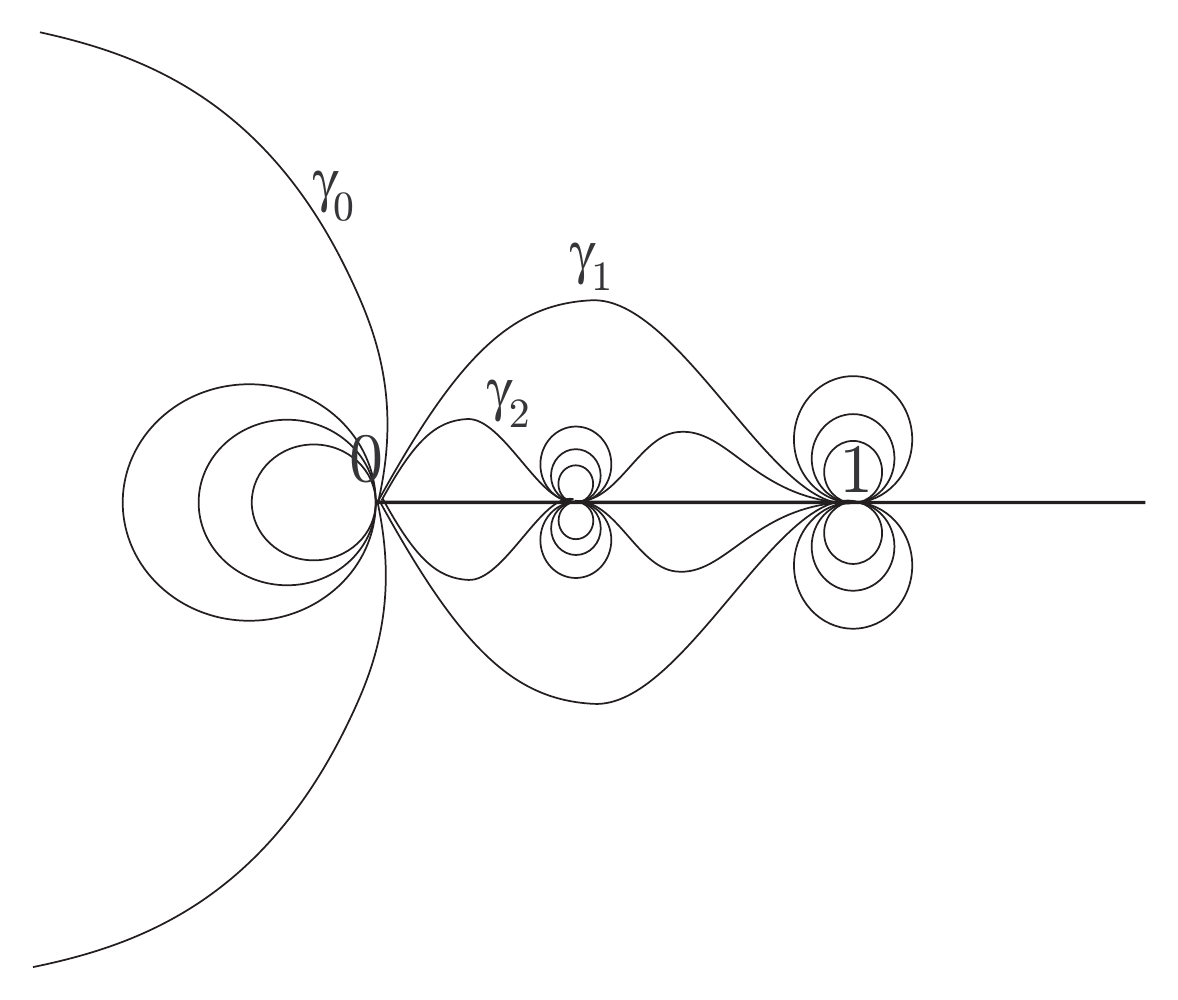}}
\caption{\label{fig-maph}A few  preimages of $\gamma$ under the iteration of  $K$ (\propref{th:preimages-of-gamma}).}
\end{figure}

\begin{proof}
For a postcritical point $z = K^n(-1)$, $n \geq 0$, 
\[
\Re(\phi_A(z)) = \Re(\phi_A(K^n(-1)) = \Re(\phi_A(-1)) + n \neq
\alpha,
\]
by the condition (\ref{eq:alpha}) on $\alpha$. Hence, the postcritical
  set does not intersect $\gamma$, so $K^n$ is a strict cover of
  degree $2^n$ over $\gamma$, i.e., $(K^n)^{-1}\gamma$ has $2^n$ components,
  each of which is a smooth open arc. Since $\gamma$ starts from and
  ends at $0$, each component starts from and ends at a preimage -- of
  order $\leq n$ -- of $0$.

Since $K$ commutes with complex conjugation,
$\Re(\phi_A(\overline{z})) =
\Re(\phi_A(z))$. Hence, the image of $P$ under complex conjugation --
which is again a petal -- is mapped by $\phi_A$ to the same right
half-plane as $P$, so $P$ is invariant under complex conjugation, so
its boundary $\gamma$ is also invariant. Again using the fact that $K$
commutes with complex conjugation, the preimage under $K^n$ of
$\gamma$ is invariant, so the complex conjugate of each component arc
is a component arc.

To show that preimages of $\gamma$ do not intersect $\Reals$: It is
easy to show, using the estimate $\phi_A(z) \approx -1/z$, that
$x \mapsto \Re(\phi_A(x))$ is strictly increasing for $x$ small and
negative. Using $\phi_A(x) = \phi_A(K^n(x)) - n$, and the fact that
all iterates $K^n$ are strictly increasing on $(-1, 0)$, we see that
$\Re(\phi_A(x))$ is strictly increasing on all of $(-1,0)$. On the
other hand, $K$ is strictly decreasing on $(-\infty, -1)$ and maps
this interval onto $(-1/4, 0)$, so -- again applying $\phi_A(x) =
\phi_A(K(x)) - 1$ -- $\Re(\phi_A(x))$ is strictly decreasing on
$(-\infty, -1)$. In particular:
\[
\Re(\phi_A(x)) \geq \Re(\phi_A(-1))\quad\text{for $x \in (-\infty,
  0)$.}
\]
If $K^n(z) \in \gamma$ (with $n > 0$),
\[\begin{split}
\Re(\phi_A(z)) &= \Re(\phi_A(K^n(z))) -n = \alpha -n\\
&\leq \alpha - 1 < \Re(\phi_A(-1)) = \min_{x<0} \Re(\phi_A(x)).
\end{split}\]
Thus, the preimage under $K^n$ of $\gamma$ does not intersect the
negative real axis. Since $\gamma$ does not intersect $[0,\infty]$,
which is fully invariant, the preimage does not intersect it either,
and hence does not intersect $\Reals$.

By condition (\ref{eq:alpha}) on $\alpha$, the critical value $-1/4$
of $K$ is in $P$, so the Jordan curve made by appending $0$ to
$\gamma$ -- the boundary of $P$ -- runs once around the critical
value. Hence, lifts of $\gamma$ under $K$ do not close. One of the two
lifts starts at $0$; it cannot end at $0$ and so must end at
$\infty$, which is the other preimage of $0$. Since this lift starts
in the upper half-plane and, by assertion (2), does not
intersect the real axis, it must lie in the upper half-plane
\end{proof}

We denote by
\begin{itemize}
\item $\gamma_0$ the component of $K^{-1} \gamma$ starting at $0$,
i.e., the component in the upper half-plane
\item $\gamma_1$ the component of $K^{-1}\gamma_0$ which begins at
  $0$. Then $\gamma_1$ also lies in the upper half-plane and ends at
  $1$, the preimage of $\infty$
\item $\gamma_2$ the component of $K^{-1}\gamma_1$ which starts at
$0$. Again, $\gamma_2$ lies in the upper half-plane; it ends at the
unique preimage $x_2$ of $1$ in $(0,1)$, whose value can easily be
computed to be $(3 - \sqrt{5})/2$.
\end{itemize}

Then $\{0\} \cup \gamma_1 \cup \{1\} \cup \overline{\gamma_1}$ is a
Jordan curve; we denote its interior by $P_R$. 

\begin{lem}
$P_R$ is an attracting petal for
  $\hinv$.
\end{lem}

\begin{proof}
Since
\begin{itemize}
\item
$
\partial (\hinv P_R) = \hinv( \partial P_R ) = \{0\} \cup \gamma_2 \cup
\{x_2 \} \cup  
\overline{\gamma_2};
$
\item $x_2 \in (0,1) \subset P_R$
\item $\gamma_2 \cup \overline \gamma_2$ does not intersect $\partial
  P$,
\end{itemize}
it follows that the closure of $\hinv P_R$ is contained in $P_R \cup
\{ 0 \}$. By local theory, $\gamma_1$ is tangent to the imaginary axis
at $0$; from this it follows that every $\hinv$-orbit with initial
point not in $(-\infty, 0]$ is eventually in
$P_R$.

It remains only to verify condition (4) of Definition
\ref{defn-petal}.  We showed above -- using Vitali's theorem -- that
$(\hinv)^n$ converges uniformly to $0$ on compact subsets of $\CC
\setminus (-\infty, 0]$. Since $\partial P_R$ is tangent to the
imaginary axis at $0$, it follows from local theory that $(\hinv)^n$
converges uniformly to $0$ on the intersection of $P_R$ with a
sufficiently small disk centered at $0$; uniform convergence on all of
$P_R$ follows.
\end{proof}

Note the mixed character of $P_R$: although it is a repelling petal,
its boundary is defined as a curve of constant $\Re(\phi_A)$, except
at the two points $0$ and $x_2$, where the $\phi_A$ is not defined.

By \propref{th:fundamental-crescent}, $C_R := P_R \setminus K(P_R)$
projects bijectively to the repelling cylinder $\cC_R$ and hence is
mapped bijectively by $\ixp \circ \phi_R$ to the punctured plane $\CC
\setminus \{ 0 \}$. We write $C^+_R$ and $C^-_R$ for the intersections
of $C_R$ with the open upper and lower half-planes respectively; then
\[
C_R = C^+_R \cup (x_2, 1) \cup C^-_R.
\]
$\ixp \circ \phi_A$ is analytic on $C^+_R$ and on $C^-_R$ but cannot be
continued analytically through any point of $(x_2, 1)$. Thus, the
image under $\ixp \circ \phi_R$ of $C^+_R$ (respectively $C^-_R$) is
$W^+$ (respectively $W^{-}$.) By an argument given above for $\phi_A$,
\[
\phi_R(\overline{z}) = \overline{\phi_R(z)} + c\quad\text{with $c$ a
  pure imaginary constant.}
\]
so the imaginary part of $\phi_R$ is constant on $(0,\infty)$, so
$\ixp \circ \phi_R$ maps $[r_2, 1]$ to a circle centered at the
origin. We chose the real part of the additive constant in $\phi_R$ so
that $\phi_R$ is real on $(0, \infty)$; then
\[
W^+ = \DD \setminus \{0\}\quad\text{and}\quad W^- = \CC \setminus
\overline{\DD}.
\]

We now have all the pieces in place to prove that the only asymptotic
values of $\bm{k}$ are $0$ and $\infty$. We assume that $s \mapsto \tau(s),
s \in [0,1)$ is a continuous arc in $W^+$ converging to the unit
  circle as $s \to 1$ (but not necessarily to an particular point of
  the circle) such that $\bm{k}(\tau(s))$ converges as $s \to 1$ to a finite
  nonzero value $z_\infty$, and we show this leads to a
  contradiction. We select an $\alpha$ as in (\ref{eq:alpha}) and also
  so that
\[
\frac{1}{2\pi} \Arg(z_\infty) - \alpha \notin \ZZ.
\]
Since $\bm{k}(\tau(s)) \to z_\infty$, we can -- by deleting an initial
segment of $\tau$ and reparametrizing -- assume that
\begin{equation}\label{eq:avoid-the-boundary}
\frac{1}{2\pi} \Arg(H(\tau(s)) - \alpha \notin \ZZ \quad\text{for $0
  \leq s < 1$.}
\end{equation}
Recall that $\ixp \circ \phi_R$ maps $C^+_R$ bijectively to $W^+$. We
denote by $\hat\tau(s)$ the image of $\tau(s)$ under the inverse of
this bijection. We claim that the lift $\hat{\tau}(\,.\,)$ is
continuous. To see this, we note that
\[
\bm{k}(\tau(s)) = \exp(2 \pi i \phi_A(\hat{\tau}(s))),
\]
so, by (\ref{eq:avoid-the-boundary}).
\begin{equation}\label{eq:avoid-the-boundary-1}
\Re(\phi_A(\hat{\tau}(s)))- \alpha \notin \ZZ\quad\text{for
  $0 \leq s < 1$,}
\end{equation}
that is, the values of $\hat{\tau}(s)$ lie in the interior of $P^+_R$,
not on the ``edges''. Since $\ixp \circ \phi_R$ is a local
homeomorphism everywhere in the interior of $C^+_R$, continuity of
$\hat{\tau}(s)$ follows.

We recall that $P$ denotes here the attracting petal mapped by
$\phi_A$ to the half-plane $\{ \Re(w) > \alpha \}$. Since the orbit of
the initial point $\hat{\tau}(0)$ of the lift is eventually in $P$,
and since $\Re(\phi_A(\hat{\tau}(0)) - \alpha \notin \ZZ$, there exist
positive integers $n$ and $m$ since that
\[
K^n(\hat{\tau}(0)) \in C := \{ z \in P : \alpha +m < \Re(\phi_A(z)) <
\alpha+m+1\}
\]
Because of (\ref{eq:avoid-the-boundary-1}), $\hat{\tau}(s)$ never
intersects either boundary arc of the crescent $C$, so, by continuity,
\[
\tilde{\tau}(s) := K^n(\hat{\tau}(s)) \in C \quad\text{for $0 \leq s < 1$.}
\]
Because the base arc $\tau(s)$ (in $W^+$) approaches the unit circle
as $s \to 1$,
$\hat{\tau}(s)$ approaches $[x_2, 1]$, so, because
  $[0, \infty]$ is mapped to itself by $K$, $\tilde{\tau}(s)$
approaches $[0,\infty]$ as $s \to 1$. We thus have the following
situation:
\begin{enumerate}
\item $s \mapsto \tilde{\tau}(s)$ is a continuous arc in the crescent
  $C$.
\item $s \mapsto \tilde{\tau}(s)$ approaches $[0,\infty]$ as $s \to
  1$.
\item $\ixp(\phi_A(\tilde{\tau}(s)))$ approaches a finite non-zero
  limit $z_0$ as $s
  \to 1$.
\end{enumerate}
It is now easy to show that these three assertions lead to a
contradiction. We first argue that (1) and (2) imply
\begin{equation}\label{eq:im-to-infinity}
\vert \Im(\phi_A(\tilde{\tau}(s)))\vert \to \infty.
\end{equation}
Then $\bm{k}(\tau(s)) = \exp(2 \pi i \phi_A(\tilde{\tau})(s))$ converges to
$0$ (if $\Im(\phi_A(\tilde{\tau})(s)) \to +\infty$) or to $\infty$,
contradicting (3).

To prove (\ref{eq:im-to-infinity}): for sufficiently small positive
$\epsilon$, there is an attracting petal $P_\epsilon$ mapped
homeomorphically by $\phi_A$ to the half-plane $\{ \Im(w) > \alpha -
\epsilon \}$. Let $r$ be a positive real number and let $R_r$ denote
the preimage under the restriction of $\phi_A$ to such a $P_\epsilon$
of the compact rectangle
\[
\{ u + i v : \alpha+m \leq u \leq \alpha + m + 1, -r \leq v \leq r \}
\]
Since $R_r$ is compact and disjoint from $[0,\infty]$, its distance
from $[0,\infty]$ is strictly positive. By (1), if $\vert
\Im(\phi_A(\tilde{\tau}(s)) \vert \leq r$, then $\tilde{\tau}(s) \in
R_r$. Hence, by (2), $\vert \Im(\phi_A(\tilde{\tau}(s)) \vert$ is
eventually $>r$. Since this is true for all $r$,
(\ref{eq:im-to-infinity}) holds.

It is evident that $\{0,\infty\}\subset\asym(H)$. Indeed, for small enough $\eps>0$ the 
projection of $\gamma_1\cap D_\eps(1)$ by $\ixp\circ\phi_R$ is a simple curve in $\DD$ whose
image under $H$ lands at $0$; similarly $\gamma_2$ yields a curve whose image under $H$ lands at $\infty$.

Now let $$\tau:[0,1)\to \DD=\Dom(H)$$ be a curve such that $\tau(t)$ converges to the unit circle $\TT$ as $t\to 1-$.
Let $$\tl\tau:[0,1)\subset\hat\CC\setminus \RR_{\geq 0}$$ be a component of its lift
under $\ixp\circ \phi_R$, parametrized so that 
\begin{equation}
\label{eq-tau}
\tl\tau(t)\to \RR_{\geq 0}\text{ as }t\to 1-.
\end{equation}

Let $C'$ be a connected component of  $K^{n}(C)$ for $n\in\ZZ$. We will say that
the curve $\tl\tau$ {\it crosses} $C'$ if there exist $t_1\neq t_2\in(0,1)$ and $n\in\ZZ$ such that
$$\tl{\tau}:(t_1,t_2)\to C',\;K^n(\tl\tau(t_1))\in\gamma,\text{ and }K^n(\tl\tau(t_2))\in h(\gamma).$$
Note that every time $\tl\tau$ crosses some $C'$, the image $\bm{k}(\tau)$ winds once around the cylinder
$\CC/\ZZ\simeq \hat\CC\setminus\{0,\infty\}$. Hence, we have the following two possibilities:
\begin{enumerate}
\item $\tl\tau$ crosses finitely many connected components of $\cup_{n\in\ZZ}K^n(C)$;
\item $\lim_{t\to 1-}\bm{k}(\tau(t))\in\{0,\infty\}.$ 
\end{enumerate}
In case (2), we are done. 

Standard considerations of dynamics on $J(K)$ imply that for every $x\in\RR_{\geq 0}$ which is not a preimage of $0$,
there exists an infinite sequence of cross-cuts $l_k\equiv K^{-{n_k}}(\gamma_i)$ for some choice of the  inverse branch
and $i=1,2$ such that denoting $N_k$ the cross-cut neighborhood of $l_k$ in $\hat\CC\setminus \RR_{\geq 0}$,
we have $$\overline{N_k}\cap \RR\ni \{x\}.$$
In view of (1) and (2) this implies that $x$ cannot be accumulated by $\tl\tau$. Since points $x$ as above 
form a dense set on $\RR_{\geq 0}$, we see that there exists a limit
$$s=\lim_{t\to 1-}\tl\tau(t)\text{ and }s\in\cup K^{-n}(0).$$
In view of (1), this  implies that
$$\lim_{t\to 1-}\bm{k}(\tau(t))\in\{0,\infty\}.$$
\end{pf}

\subsection{Parabolic renormalization of $f_0$}
We now have:
\begin{prop}
\label{renormalizable1}
The maps $f_0$ and $K$ are both renormalizable.
\end{prop} 
\begin{proof}
The arguments are identical, so we only argue for $f_0$. 
Let us replace  $\bm{h}_{f_0}$ by $c\bm{h}_{f_0}$ if needed so that $\bm{h}'_{f_0}(0)=1$. We have to show that the
coefficient $a$ in the Taylor
expansion 
$$\bm{h}_{f_0}(z)=z+az^2+\cdots$$
is not equal to $0$, or, in other words, that $z=0$ is a simple parabolic point of $\bm{h}_{f_0}$. Let $P_0$ be an attracting petal
of $z=0$. For $n\geq 1$ denote $P_{-n}$ the component of the preimage $(\bm{h}_{f_0})^{-1}(P_{-(n-1)})$ which contains $P_{-(n-1)}$.
We claim that there exists $n$ such that $P_{-n}$ contains the unique critical value $v$ of $\bm{h}_{f_0}$.
Assume the contrary. Then for every $n$ the domain $P_{-n}$ is a topological disk. Denote $$B_0=\cup P_{-n}\subset \hat W^+.$$
The Fatou coordinate $\tlphi_A$ of $\bm{h}_{f_0}$ extends to all of $B_0$ via the functional equation
$$\tlphi_A\circ \bm{h}_{f_0}(z)=\tlphi_A(z)+1$$
as an unbranched analytic map.
It is trivial to see that its image is the whole complex plane $\CC$, and hence $U$ must be a parabolic domain. This 
is impossible, however, as $U$ is a subset of a Jordan domain $\hat W^+\subset\CC$.

Now assume that $a=0$. Then the parabolic point $z=0$ has at least two distinct attracting directions $v_1, v_2\in S^1$.
Consider corresponding attracting petals $P_A^1\cap P_A^2=\emptyset$. As we have shown above, the critical value $v$ must be
contained in the intersection of their preimages. This leads to a contradiction, as in this case
$$\lim \bm{h}_{f_0}^n(v)/|\bm{h}_{f_0}^n(v)|=v_1=v_2.$$
\end{proof}


We now turn to the dynamics of the parabolic renormalization $$F\equiv \cP(f_0):\hat W^+\to\hat\CC.$$ 
\begin{defn}
The points $z\in\hat W^+$ whose orbits do not escape $\hat W^+$ form the {\it filled Julia set}
$K(F)$. Its boundary $J(F)$ is the {\it Julia set}.
\end{defn}

Lavaurs \cite{Lav} has shown that
\begin{thm}[\cite{Lav}]
The interior of the filled Julia set $K(F)$ coincides with the basin of the parabolic point $0$. Repelling periodic orbits
of $F$ are contained in $J(F)$ and are dense in $J(F)$.
\end{thm}

Let $P^0\equiv P_A$ be an attracting petal of $F$. For $n\in\NN$ inductively define $P^{-n}$ to be the connected component
of the inverse image $F^{-1}(P^{-(n-1)})$ which contains $P^{-(n-1)}$. The same considerations as in the proof of \propref{renormalizable1} imply:
\begin{prop}
\label{v in petal}
There exists $n\in\NN$ such that $P^{-n}$ is a topological disk which contains $ v$ in its closure.
\end{prop}


\subsection{A note on the general theory for analytic maps of finite type}
As the previous example clearly demonstrates, when $f$ has global covering properties, 
we can expect the {\'E}cale-Voronin map $\cE_f$ to also possess a well-understood global 
covering structure. The appropriate setting for a global structure theory for {\'E}calle-Voronin maps is that
of analytic maps {\it of finite type}, developed in A.~Epstein's thesis \cite{Ep}.
While we will not need these results in our investigation, we will briefly mention
some of them below for the sake of completeness of the
exposition.

\begin{defn}[\cite{Ep}]
Let $f:W\to X$ be an analytic map between two Riemann surfaces. Assume further that $X$ is compact, and $W$ lies in some compact
surface $Y$. Suppose that $f$ is nowhere constant, and that every isolated singularity of $f$ is essential.
We say that $f$ is a map of finite type if $\sing(f)$ is a finite set.
\end{defn}

Rational endomorphisms $f:\hat\CC\to\hat\CC$ are obviously maps of  finite type. Epstein demonstrated:
\begin{thm}\cite{Ep}
If an analytic map
$f$ with a parabolic cycle is of finite type, then the corresponding {\'E}calle-Voronin maps $\cE_f:W_f\to \hat\CC$ inherit
the same property.
\end{thm}

Epstein further showed that the familiar properties, such as density of repelling periodic points and topological minimality, 
hold for Julia sets of analytic dynamical systems of finite type. Most importantly, he proved that the Fatou-Sullivan 
structure theory holds for the Fatou sets of such dynamical systems:

\begin{thm}\cite{Ep}
Every connected component of the Fatou set of a map of finite type is pre-periodic. All periodic Fatou components are either
 basins (attracting or parabolic), or rotation domains (Siegel disks or Herman rings).
\end{thm}

\section{A class of analytic mappings invariant under $\cP$}
\subsection{Definition of $\cla$}
We will now specialize to a much narrower class of analytic maps with parabolic orbits. As before, we write 
$$f_0(z)=z+z^2,$$ and we set
 $$F\equiv\pr(f_0)\text{ and }W_F=\Dom(\pr(f_0)).$$

\begin{defn}
\label{defn-class-0}
We will denote $\clab$ the class of  analytic germs $f(z)$
at the origin which have maximal analytic extensions
 $$f:\cD(f)\to\hat\CC$$
such that 
\begin{itemize}
\item[(I)] $\cD(f)$ is a simply-connected domain;
\item[(II)] Denoting $\varphi_f$ the Riemann mapping
$$\varphi_f:\DD\to \cD(f),\text{ with }\varphi_f(0)=0\text{ and }\varphi'_f(0)>0,$$
we have
$$f\circ \varphi_f(z)=v\cdot F\circ \varphi_{F}(e^{2\pi i\theta}z),\text{ for }v\neq 0,\;\theta\in\RR.$$
\end{itemize}

Note that denoting $c$ and $c^F$ the unique critical values of $f$ and $F=\pr(f_0)$ respectively, we get the rescaling factor as $v=c/c^F$.

\end{defn}
If we again let $\bm{k}$ be the normalized $ \bm{h}_K$ for the Koebe function $K(z)$,
 then, by \propref{straighten}, the property (II) in the
above definition is equivalent to
\begin{itemize}
\item[(II')] Denoting $\varphi_f$ the Riemann mapping
$$\varphi_f:\DD\to \cD(f),\text{ with }\varphi_f(0)=0\text{ and }\varphi'_f(0)>0,$$
we have
$$f\circ \varphi_f(z)=v\cdot \bm{k}(e^{2\pi i\theta}z),\text{ for }v\neq 0,\;\theta\in\RR.$$
\end{itemize}

\noindent
We note:
\begin{prop}
For all $f\in\clab$ we have $f''(0)\neq 0$.
\end{prop}
\begin{proof}
Suppose $f''(0)=0$. By multiplying by a non-zero constant, if necessary, we may reduce the proof to the
case when $f'(0)=1$, so that $f(z)=z+az^n+\cdots$ for $n>2$. A Leau-Fatou flower of $f$ has $n-1\geq 2$ attracting
petals. Since $f$ has a single critical value $c^f$, there exists an ample petal $P_A$ such that
$$P_A\cap \cup_{k\in\NN} f^k(c^f)=\emptyset.$$
Let $\tlphi_A$ be an attracting Fatou coordinate defined on $P_A$. Set $P_0\equiv P_A$ and inductively define
$P_{-n}$ as the component of the preimage $f^{-1}(P_{-(n-1)})$ which contains $P_{-(n-1)}$ for $n\in\NN$.
The function $\tlphi_A$ analytically extends via the functional equation $$\tlphi_A\circ f(z)=\tlphi_A(z)+1$$
to the union $$U\equiv \cup P_{-n}\subset \Dom(f).$$ A simple induction shows that 
$\tlphi_A$ is univalent on $P_{-n}$, and hence on all of $U$. On the other hand, $\phi_A(U)=\CC$, which is impossible
since $U$ is a hyperbolic domain. 
\end{proof}

\begin{defn}
\label{defn-class}
We define $\clac$ as the set of maps $f\in\clab$ of the form
$$f(z)=z+z^2+\cdots$$
at the origin.

We further define $\claa\subset \clac$ as the collection of maps $f\in\clac$ such that the domain $\cD(f)$ has locally connected
boundary. Finally,  $\cla\subset \claa$ consists of maps whose domain of definition is Jordan.
\end{defn}

\noindent
Since $F=\pr(f_0)\in\cla$, we have
\begin{prop}
The class $\cla$ is non-empty.
\end{prop}

\noindent
We now set out to prove the following theorem:
\begin{thm}
\label{class-2}
Every $f\in\cla$ is renormalizable, and  the parabolic renormalization $\pr(f)\in\cla$.
\end{thm}

Let $f\in\clab$ and let $P_0$ be an attracting petal of the parabolic point $z=0$. For $n\geq 1$ let
$P_{-n}$ be the component of $f^{-1}(P_{-(n-1)})$ which contains $P_{-(n-1)}$. We let 
$$B_0^f\equiv \cup P_{-n},$$
and call it the {\it immediate basin} of $0$. We note:
\begin{prop}
\label{critval1}
For $f\in\clab$ the unique critical value $c^f\in B_0^f$. Moreover, there exists an attracting petal $P_A\subset B_0^f$
which is a topological disk containing $c^f$ in its interior.
\end{prop} 
\begin{proof}
Assume that $v\notin B_0^f$. The attracting Fatou coordinate $\tlphi_A^f$ extends from $P_0$ to the whole of $B_0^f$
via the functional equation 
$$\tlphi_A^f\circ f(z)=\tlphi_A^f(z)+1.$$
A simple induction shows that $P_{-n}$ is an increasing sequence of topological disks, on each of which 
$\tlphi_A$ is univalent. Hence, $\phi_A$ restricted to $B_0^f$ is a conformal homeomorphism onto the image.
Yet it is clear that the image of $\phi_A$ restricted to $B_0^f$ is the whole of $\CC$, which contradicts the
fact that $B_0^f$ is a hyperbolic domain. The second part of the statement is elementary, and is left to the reader.
\end{proof}

\noindent
Pushing the argument a little further, we have:

\begin{prop}
\label{basin-simply-connected}
For $f\in\clab$ the immediate basin $B_0^f$ is simply connected and contains exactly one critical point of $f$.
The restriction $f:B_0^f\to B_0^f$ is a degree-$2$ branched covering.
\end{prop}

\begin{proof}
By \propref{critval1}, there exists  $n$ such that $P_{-n}$ contains the critical value $c^f$. Inductively applying
\lemref{proper5}, we see that for $k\in\NN$ the domain $P_{-(n+k)}$ is a topological disk, and $f:P_{-(n+k)}\to P_{-(n+k-1)}$
is a branched covering of degree $2$ with a single, simple critical point. 
\end{proof}

The proof of \thmref{class-2} will rely on the following key result, which is a direct analogue of \thmref{cauliflower}.
\begin{thm}
\label{class-1a}
Let $f\in\cla$, and denote $B^0_f$ the immediate basin of the parabolic point $0$ of $f$.
Then  $B_f^0$ is a Jordan domain. Denote $\hat f$ the continuous extension of $f$ to $\partial B_0^f$. There exists
a homeomorphism $\rho:\partial B_0^f\to\TT$ such that
$$\rho(\hat f(z))=2\rho(z)\mod 1.$$
\end{thm}

\ignore{
\begin{proof}
To begin the proof  let us construct accesses to the parabolic point $0$ of $f$ both from the inside and from the 
outside of $B^0_f$. 
Let $P_R^f\subset \Dom(f)$ be a repelling petal of $0$.
By Epstein-Fatou-Sullivan classification of the Fatou components of maps of finite type \cite{Ep}, the projection
$\tl B_f^0$  of 
$\overline{B_f^0}\cap P_R^f$ to the repelling cylinder $\cC_R^f$ has two polar connected components which do not intersect.
Consider any simple closed equatorial loop $l$ in $\cC_R^f$ which lies in the complement of $\tl B_f^0$. Its lift to $P_R^f$ is a simple curve  
$\gamma_\text{out}$ which lies outside $\overline{B^0_f}$ and has $0$ as one of the end points.

To construct an access $\gamma_\text{in}\subset B^0_f\cap P_A^f$ we follow a similar procedure. We denote $\tl c\in \cC_A^f$ the projection of the
unique critical value orbit of $f$, and consider a simple equatorial loop $l$ in $\cC_A^f$ which passes through $\tl c$ (we can simply take the circle
$\{\Im(z)=\Im(\tl c)\}\subset \CC/\ZZ\simeq \cC_A^f$ in this case). The lift $\gamma_\text{in}$ of the curve $l$ in this case
is an access to $0$ from the inside of the immediate basin, which in addition has the virtue of containing all of the critical value orbit of $f$
except at most finitely many points. We set 
$$\Gamma=\gamma_\text{out}\cup\gamma_\text{in}\cup\{0\}.$$

Denote $g\subset B^0_f$ the component of the preimage of $\gamma_\text{in}$ which does not intersect with $\gamma_\text{in}$.
By the Carath{\'e}odory Theorem, the conformal map 
$$z\mapsto \phi_f(e^{-2\pi i\theta}\phi_F^{-1}(z))$$
extends to a homeomorphism $\overline{\Dom(F)}\mapsto \overline{\Dom(f)}$. This, and the covering properties of $F$, imply that
the curve $g$ lands at a single boundary point $t\in \partial \Dom(f)$. 
Let $t_{-n}\in f^{-n}(t)\cap\partial B^0_f$ be a point in $P_R^f$. Denote $\gamma'_\text{in}\subset B^0_f$ the preimage of $\gamma_\text{in}$ which
terminates at $t_{-n}$.  
Since $\Dom(f)$ is a Jordan domain, the point $t$ is accessible 
from outside of $\Dom(f)$. Hence, the point $t_{-n}$ is accessible from outside $\overline{B^0_f}$. Let us denote $\chi\subset P_R^f$ such an access.
We set 
$$\Gamma'=\gamma'_{\text in}\cup\chi\cup\{t_{-n}\}.$$

Again using the Epstein-Fatou-Sullivan Theorem, we connect the outer ends of the curves $\Gamma$ and $\Gamma'$ by two arcs 
$\tau_\text{in}\subset B^0_f$ and $\tau_\text{out}\subset \Dom(f)\setminus \overline{B^0_f}.$ We denote $Q_0$ the region bounded
by the union $\Gamma\cup\Gamma'\cup\tau_\text{in}\cup\tau_\text{out}$.
Observe, that the way we constructed $\gamma_\text{in}$ allows us to assume in addition that
$Q_0$ does not intersect with the closure of the critical orbit of $f$.

We note that by construction:

$\bullet$ $Q_0\cap \partial B^0_f$ is connected.

Note that preimages of $\overline{Q_0}$ form a basis of closed neighborhoods for $\partial B^0_f$. 
Consider a sequence of preimages  $Q_{-n}\cap B^0_f\neq \emptyset$ with the property $f:Q_{-n}\mapsto Q_{-n+1}$ for $n\geq 1$.
We claim:

$\bullet$ $\diam Q_{-n}\rightarrow 0$.

Indeed, the univalent inverse branches 
$$\{f_{-j}:Q_0\to Q_{-j}\}$$
form a normal family by Montel's Theorem. If we assume that 
$$\diam Q_{-n}\nrightarrow 0,$$
then
$$f_{-j_k}|_{Q_0}\rightrightarrows h,$$
which is a non-constant analytic function defined on $Q_0$. Denote $G=h(Q_0)$. By construction, $G$ intersects the Julia set
$J(f)$. By Montel's Theorem again, there exists a (possibly empty) set $E\subset \hat\CC$ with $|E|\leq 2$ such that for every compact set 
$K\subset( \hat\CC\setminus E)$ there exists an iterate $f^j(G)$ which covers all of 
$K$. This is clearly impossible.

We can thus conclude:

$\bullet$ $\partial B^0_f$ is locally connected. 

Finally, the Epstein-Fatou-Sullivan Theorem implies that $\partial B^0_f$ is a Jordan curve.

\end{proof}

}

The proof of \thmref{class-1a} is quite involved, as it will require a detailed understanding of the covering properties
of a map in $\cla$. Let us show how \thmref{class-1a} implies \thmref{class-2}:

\begin{proof}
Consider a Riemann map $$v:\DD\to B^f_0$$ which maps $0$ to
the sole critical point of $f$ inside $B^f_0$.
By \thmref{class-1a} and Carath{\'e}odory Theorem, 
the Riemann map has a continuous extension to a homeomorphism $\tl v:\overline{\DD}\to \overline{ B^f_0}$. 


Let us further normalize the Riemann map $\DD\to B^f_0$ by requiring that
$\tl v(1)=0.$ This specifies the mapping uniquely, and we denote it $\psi_f$.

By \thmref{th:blaschke-model}, 
$$(\psi_f)^{-1}\circ f\circ \psi_f=B:\DD\to\DD\text{, where }B(z)= \frac{3 z^2 + 1}{3 + z^2}.$$

\ignore{
The map $B$ is a degree two Blaschke product, with a single simple ramification point at the origin.
It  fixes the boundary point $1$. 
By the local dynamics of $f$, the point  $1$ attracts some orbits of $B$ in $\bar\DD$, and hence $|B'(1)|\leq 1$.

Assume that $|B'(1)|<1$. Then there exists an arbitrarily small crosscut $\gamma$ forming a crosscut 
neighborhood $N$  of the boundary point $p=0$ in $B^f_0$,
whose  forward orbit  is entirely contained in $N$. This clearly contradicts the local dynamics of the parabolic point $p=0$,
since the boundary points $\partial\gamma\subset\partial B^f_0$ lie in a repelling petal of $f$.

Hence $|B'(1)|=1$, which evidently means $B'(1)=1$. 
Applying Schwarz reflection about the circle we see that the parabolic point
$1$ has at least two attracting directions, and hence $B''(1)=0$.

An inspection shows that the unique Blaschke product of degree two with the properties $B'(0)=0$, $B(1)=1$, $B'(1)=1$, $B''(1)=0$ is
given by the formula
$$B(z)=\frac{3z^2+1}{3+z^2}.$$
}

 Let us set 
$$\chi=\psi_{f_0}\circ (\psi_f)^{-1}:B_0^f\mapsto B_0^{f_0}.$$
By the discussion above, this mapping is a conjugacy:
$$\chi\circ f|_{B_0}=f_0\circ \chi|_{B_0}.$$

Let us fix 
$$D=\{|w|<|c^f|\},$$ and let $g$ be the branch of $f^{-1}$ fixing $0$ which is defined in $D$.
For the moment, the map $\cP_f$ is only defined locally in a neighborhood of $0$, which is {\it a priori} dependent on the
choice of a repelling petal $P_R^f$. Let us fix a petal $P_R^f\subset D$. 

The map $F=\cP(f_0)$ has a maximal extension to a Jordan domain $W_{F}$. Let us use
 \thmref{cauliflower} to select a simple arc $\gamma$ connecting $0$ with $a\in J_{f_0}$ inside $B_0^{f_0}\cap P_R^{f_0}$ in
some repelling petal of $f_0$, such that:
\begin{itemize}
\item ${\gamma}\cap J_{f_0}=\{ a\}$;
\item $\gamma\cap f_0(\gamma)=\{0\}$;
\item denoting $V_0\subset P_R^{f_0}$ the Jordan domain bounded by $\gamma$, $f_0(\gamma)$, and the 
Jordan subarc of $J_{f_0}$ between $a$ and $f_0(a)$, we have a one-to-one correspondence between points in $W_{f_0}$ and
points in $$\hat V_0\equiv (V_0\cup\gamma)\setminus\{ a\};$$
\item $\hat V\equiv \chi(\hat V_0)\subset P_R^f$.
 \end{itemize}
Denote $\cW\subset \cC_R^f$ the quotient of $\hat V$ by the action of $f$. Let $W\subset \CC$ be the image of $\cW$ by
$\ixp\circ \phi_R^f$ with the puncture at $0$ filled. By construction, $W$ is a Jordan domain, which contains a neighborhood
of $0$. We claim that:
\begin{enumerate}
\item  the germ $\cP(f)$ analytically continues to all of $W$, and, moreover, 
\item $\cP(f)$ cannot be analytically continued
to a neighborhood of any point $z\in\partial W$.
\end{enumerate}

Since $V\subset P_R^f\cap B_0^f$, the first statement is immediate by the definition of $\cP(f)$. 
Furthermore, by (\ref{closure-crit}), the set of critical points 
$$\text{Crit}(\tlphi^f_A)=\chi(\text{Crit}(\tlphi^{f_0}_A))$$
contains $\partial B_0^f$ in its closure. Since $\partial W$ is a projection of a Jordan subarc of $\partial B_0^f$,
this implies (2).
We have thus shown the existence of a Jordan domain $$\Dom(f)=W.$$

Let $$\varphi_f:\Dom(f)\to\DD$$
be the unique Riemann mapping normalized as in Definition (\ref{defn-class-0}). The uniqueness of the Riemann mapping implies
that, up to a pre-composition with a rotation, it is given by
the composition
$$\ixp\circ \phi_R^f\circ \chi^{-1}\circ (\phi_R^{f_0})^{-1}\circ \ixp^{-1}\circ \varphi_{F},$$
with suitably chosen inverse branches.

The map
$\chi$ induces a conformal isomorphism between attracting Fatou cylinders $\cC^f_A\to\cC^{f_0}_A$. 
After uniformizing the cylinders by $\CC^*$, it takes the form
$$\tau\equiv \ixp\circ\phi_A^{f_0}\circ\chi\circ (\phi_A^f)^{-1} \circ \ixp^{-1}:\CC^*\to\CC^*$$
(with suitably chosen inverse branches). By \propref{uniqueness-fatou},
$\tau$ is a multiplication by a non-zero constant.


A diagram chase now implies that
$$\cP_f\circ \varphi_f(z)=v\cdot F\circ \varphi_{F}(e^{2\pi i\theta}z)$$
for some $v\neq 0$ and $\theta\in\RR$. Since the unique critical value of the mapping on the left coincides with that of
the mapping on the right, $v=c^{\cP(f)}/c^{F}$.


\end{proof}

Let $f\in\cla$. By definition of parabolic renormalization, for any repelling petal $P_R^f$ the projection of the intersection
$P_R^f\cap B_0^f$ by $\ixp\circ \phi_R^f$ lies in the domain $\Dom(\cP(f))$. We remark, that it covers all of it:
\begin{rem}
\label{remark-intersection}
For any choice of $P_R^f$,
the projection
$$\ixp\circ\phi_R^f(P_R^f\cap B_0^f)\cup\{0,\infty\}$$
is a union of two disjoint Jordan domains $W^+\ni 0$ and $W^-\ni\infty$ with 
$$W^+=\Dom(\cP(f)).$$
\end{rem}
We begin the proof with a lemma:

\begin{lem} \label{th:22a}
Let $D$ be a Jordan domain containing $0$ such that
\begin{itemize}
\item $c^f\notin D$;
\item there exists a Jordan arc $\tau\ni c^f$ in  $\hat\CC \setminus D$ running to $\infty$ such that
$\tau \cap B^f_0$ is connected.
\end{itemize}
Analytically extend $g$ to  $D$.
 If $w \in B^f_0\cap D$, then $g(w) \in B^f_0$.
\end{lem}

\begin{proof}
Standard considerations imply 
that $g$ extends to a branch $\hat{g}$ of $f^{-1}$ defined
and analytic on $\CC\setminus \tau$.
%
%
Note that $B^f_0 \setminus\tau$ is also connected.
For a point $z_0$ on the
negative real axis and near to $0$, the asymptotic development 
for the attracting Fatou coordinate (\propref{crude asymptotics}) implies that $z_0$ and $g(z_0)\approx z_0$
both lie in $B_0^f$.

Let $z \in B^f_0 \setminus \tau$; then there is a Jordan arc
$\gamma$ in $B^f_0 \setminus \tau$ from $z_0$ to $z$. By the covering
properties of $f$, there is a unique lift $\tilde{\gamma}$ of $\gamma$
starting at $g(z_0)$. Furthermore, this lift is contained in $B^f_0$,
so its end point is in $B^f_0$. On the other hand, for $\hat{g}$ as above,
$$s \mapsto \hat{g}(\gamma(s))$$ is another lift with
the same starting point. By uniqueness of lifts, it coincides with
$\tilde{\gamma}$. In particular, the end point of $\tilde{\gamma}$ is
$\hat{g}(z)$. Since we already know that $\tilde{\gamma}$ lies in
$B^f_0$, it follows that $\hat{g}(z) \in B^f_0$. In particular, if
$z\in D$, then $\hat{g}(z) = g(z)$, so $g(z)
\in B^f_0$, as asserted.
\end{proof}

\begin{proof}[Proof of Remark \ref{remark-intersection}]
Extend the local inverse $g$ to all of $P_R^f$. If necessary, replace $P_R^f$ with $g^n(P_R^f)$ for a sufficienly
large $n\in\NN$ to guarantee that there exists a 
domain $D$ as described in \lemref{th:22a} such that $P_R^f\subset D$.

Denote $$W=\ixp\circ\phi_R^f(P_R^f\cap B_0^f)\text{ and }V=\ixp\circ\phi_R^f(P_R^f\setminus B_0^f).$$
\lemref{th:22a} 
implies that $W$ and $V$ are disjoint. The common boundary $J=\partial W=\partial V$ is the projection
 $$J=\ixp\circ\phi_R^f(P_R^f\cap \partial B_0^f).$$
By \thmref{class-1a}, $J$ is a union of two disjoint Jordan curves. Hence, $W$ is a union of two Jordan domains $W^+\ni 0$ 
and $W^-\ni \infty$,
bounded by the components of $J$. 
On the other hand, $\Dom(\cP(f))$ is also a Jordan domain, which does not intersect $J$. Hence, it is contained
in the component of $W$ which surrounds $0$, and is, in fact, equal to it by the maximality of the analytic continuation of $\cP(f)$.
\end{proof}

\ignore{
Let us denote $E:\DD\to\hat \CC$ the mapping $(c_0)^{-1}\cdot\pr(f_0)\circ\varphi_{f_0}$ as described above. As follows from the above Proposition,
this map is universal in the following sense: for every $f\in\cla$, we have
$$\pr(f)\circ \varphi_f(z)=cE(e^{2\pi i\theta}z).$$
Thus a parabolic renormalization of a map in $\cla$ has the same topological covering structure as the map $E$. 
We note that $E$ allows the following 
description. Consider the rational mapping
$$R(z)=\frac{5z^2-1}{3z^2+1}.$$
This quadratic rational map can be described as the {\it mating} of the quadratic polynomials $z\mapsto z^2-2$ and
$z\mapsto z+z^2$ (see the original paper \cite{Do1}, as well as \cite{Mil2}, \cite{YZ} for a discussion of mating).
  The critical points of $R$ are at $0$ and $\infty$.
The orbit of the finite critical point is
$$0\mapsto -1\mapsto 1,$$
and the Julia set $J(R)=[-1,1]$. The point $1$ is a simple parabolic, $R'(1)=1$. By considerations of real symmetry,
the domain of $\cP(R)$ is a round disk (the projection of the repelling Fatou  petal minus $[-1,1]$). Hence,
$E$ coincides with $\cP(R)$ up to a rescaling of the domain and the range.

\begin{figure}
\label{round_domain_fig}
\centerline{\includegraphics[width=0.8\textwidth]{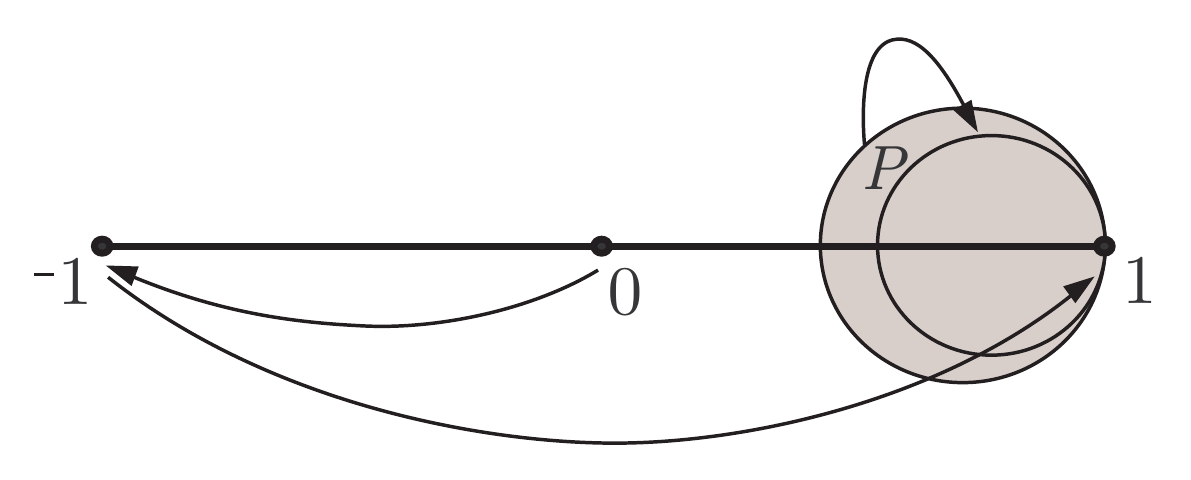}}
\caption{The Julia set of the map $R$. Also indicated is the orbit of the critical point $0$, as well as the 
image of an attracting Fatou petal $P$.}
\end{figure}

As a corollary of \propref{class-1}, we have the following invariance property for the class $\cla$:

\begin{prop}
\label{class-2}
For every $f\in\cla$ its parabolic renormalization $\pr(f)\in\cla$.
\end{prop}

\begin{proof}
For the quadratic map $f_0:z\mapsto z+z^2$ the covering properties of the map $\pr(f_0)$ imply 
that for every sufficiently large disk $D\ni 0$ there exists a simply connected subdomain 
$V\subset W_{\pr(f_0)}$ which is mapped by $\pr(f_0)$ onto $D$ as a branched covering of degree $2$,
except at the origin.

By \propref{class-1},
for an arbitrary $f\in\cla$, the parabolic renormalization $\pr(f)$ 
inherits the same property. The
rest of the proof is straightforward.
\end{proof}

Let us denote $\clas$ the space of analytic maps 
\begin{equation}
\label{decomposition}
f=E\circ \psi:W\to \hat \CC
\end{equation}
 where $E$ is the mapping from \propref{class-1},
and $\psi$ is a conformal mapping of a simply-connected domain $W\ni 0$ to $\DD$,  with the property $\psi(0)=0$, normalized at the origin so that
$f(z)=z+z^2+\cdots$.
We then have
$$\cP:\cla\to \clas\subset \cla.$$

}

\subsection{The structure of the immediate parabolic basin of a map in $\claa$}
In this section we will prove several results about lifts of parametrized paths of the form
$$s:[0,1]\to\hat\CC.$$
We will always assume $s(t)$ to be continuous on $(0,1)$. We will say that $s$ {\it lands} at a
point $a\in\hat\CC$ if $$\lim_{t\to b}s(t)=a\text{ for }b\in\{0,1\}.$$
We will generally use the same letter $s$ to denote the function $s(t)$ and the curve
$s([0,1])$ which is its range. We will call 
the image of the open interval $(0,1)$ under $s(t)$ an {\it open} path, and will denote it by $\cir{s}$.
Several times we will encounter the situation when there is a domain $W\subset\hat\CC$ and a curve 
$$s:(0,1)\mapsto{W},$$
which lands at a point $w\in\partial W$ (to fix the ideas, assume that $s(0)=w$). If $w$ has more than one prime end in $W$, then there is
a unique prime end $\hat w$ such that for every prime end neighborhood $N(\hat w)$ we have $s\cap N(\hat w)\neq\emptyset$.
In this case, we will write
$$s(t)\underset{t\to 0+}{\lra}\hat w.$$

Let $f\in\claa$. 
Let $s\mapsto\gamma(s)$ be a continuous path such that:
\begin{itemize}
\item $\gamma(0)=\gamma(1)=0$;
\item $\gamma(s)\neq c^f$ for all $s\in(0,1)$;
\item the winding number $W(\gamma,c^f)=1$;
\item there is an $\eps\in(0,\pi/2)$ such that $\gamma\cap D_\eps(0)$ lies in the sector
$\{\text{Arg}(z)\in(-\pi-\eps,-\pi+\eps)\}$.
\end{itemize}
Standard path-lifting considerations imply that there exists a unique continuous mapping 
$s\mapsto\tl\gamma(s)$, defined for $0\leq s<1$ such that
$$\tl\gamma(0)=0,\text{ and }f(\tl\gamma(s))=\gamma(s)\text{ for }0\leq s<1.$$
We claim:
\begin{prop}
\label{tip1}
We have the following. 
\begin{enumerate}
\item For any loop $\gamma$ as above,
there exists a limit  
$$t\in\partial\Dom(f)=\lim_{s\to 1}\tl\gamma(s).$$ 
\item This point (which we will denote $t^f$) is the same for all loops $\gamma$ satisfying the above properties.
\item For $f=F\equiv \pr(f_0)$, the point $t^F$ is the projection of the inverse orbit 
$$z_{-n}=(\finv{0})^n(-1)$$
by $\ixp\circ \tlphi_R$.
\item Similarly, for $f=H\equiv \pr(h)$, the point $t^H$ is the projection of the inverse orbit 
$$z_{-n}=K^{-n}(1),$$
where the inverse branch is selected to preserve the interval $(0,1)$.

\end{enumerate}
\end{prop}
\begin{figure}
\centerline{\includegraphics[width=1.2\textwidth]{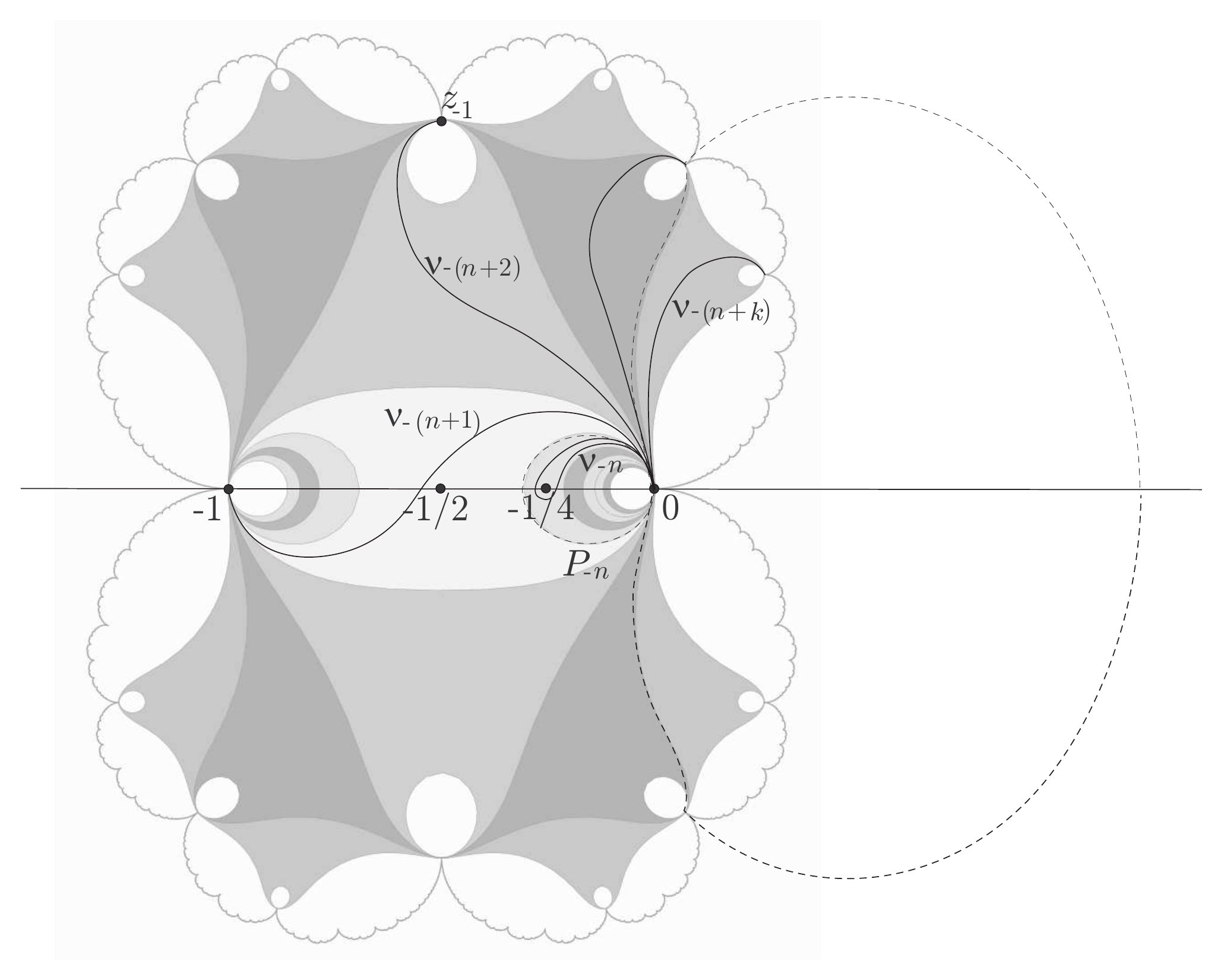}}
\caption{\label{fig-chessboard}
An illustration of the proof of \propref{tip1}}
\end{figure}

\begin{proof}[Proof of \propref{tip1}]
We will prove  (1) - (3) for $f=F\equiv \pr(f_0)$. By the definition of $\claa$ and by Carath{\'e}odory Theorem,
 this will imply (1) and (2) in the
general case. The proof of (4) follows along the same lines as the proof of (3) and will be left to the reader.

Let us begin by selecting a simple curve $\sigma\subset\hat\CC$ which connects $0$ and $\infty$, does not intersect with $\gamma$ except 
at the endpoint $0$, and such that an analytic branch of the logarithm defined on a neighborhood of $\sigma$ has bounded imaginary part.

 Recall that
the attracting Fatou coordinate $\tlphi_A^{f_0}$ holomorphically extends to the whole immediate basin $B_0^{f_0}$ via the
equation
$$\tlphi_A^{f_0}\circ f_0(z)=\tlphi_A^{f_0}(z)+1.$$
This extension is a branched covering $B_0^{f_0}\to\CC$ with simple ramification points at preimages of the critical point $-1/2$.

It is elementary to see what the lift of $\sigma$ to the dynamical plane of $f_0$ looks like.
We summarize its properties below, and
invite the reader to verify them.
Let us denote $$\cir{\sigma}\equiv\sigma\setminus\{0,\infty\}.$$
The preimage of $\cir{\sigma}$ under $\ixp\circ \tlphi_A^{f_0}$ is a countable collection of disjoint simple curves $\cup \cir{s_j}$.
We will denote the closure of $\cir{s_j}$ by $s_j$. It is obtained by adjoining two endpoints to $\cir{s}_j$: two elements of the
grand orbit $\cup_{n\geq 0}(f_0)^{-n}(0)$, which are not necessarily distinct. 

The curves  $s_j$ form a grand orbit under $f_0$ as well. The connected components 
$$S_i\text{ of }B_0^{f_0}\setminus\cup s_j$$
are mapped by the attracting Fatou coordinate $\phi_A^{f_0}$ onto curvilinear strips $\tl S_i\subset \CC$ of  infinite height and of unit
width. The strips $\tl S_i$ are bounded by unit translates of the same curve $\tl\sigma\subset\CC$. Let us enumerate 
these strips in such a way that 
$$T:\tl S_i\to \tl S_{i+1},\text{ where }T(z)=z+1.$$

For a fixed $\tl S_i$ we have
$$\sup_{x,y\in \tl S_i} (\Re(x)-\Re(y))<\infty.$$
This is where we have used the assumption on the branches of $\log$ in $\hat\CC\setminus\sigma$. Thus for every $\ell\in\ZZ$ the union
$\cup_{i\geq \ell}\tl S_i$ contains a right half-plane.

Now let us take {\it any} component $S_i\subset B_0^{f_0}$. The above observation implies that, for any $\ell\in\NN$,
$$\phi_A^{f_0}(\cup_{k\geq \ell} f^k( S_i))\supset \{\Re z>A\}$$ 
for some $A\in\RR$. Note that if $\ell$ is large enough, then $\cup_{k\geq \ell} f^k( S_i)$ does not contain any preimages of $-1/2$, and hence
the restriction of $\phi_A^{f_0}$ to it is unbranched. Denote  the boundary components of $S_{i+\ell}$ by $s$ and $f_0(s)$. Then
$s$ bounds an attracting petail $P$, and
$$P=\cup_{k\geq \ell} f^k( S_i).$$
We now complete the proof as follows. Fix $$S\equiv S_{i+\ell}=P\setminus \overline{f_0(P)}.$$
The curve $\cir{\gamma}= \gamma((0,1))$ has a univalent pull-back $\cir{\nu}\subset S$. 
We continuously extend it to a parametrized loop $$\nu:[0,1]\to S\cup\{ 0\}$$ which starts
and ends at the parabolic point $0$. 

Denote $P_0\equiv P$ and, for $j\in\NN$,  let $P_{-j}$ be the connected component of $(f_0)^{-1}(P_{-(j-1)})$ which contains $P_{-(j-1)}$. 
By \propref{critval1}, there exists $n\in\NN$ such that $P_{-n}$ is an attracting petal such that the critical value $-1/4$
is contained in $P_{-n}\setminus \overline{P_{-(n-1)}}$.
Denote $\nu_0=\nu$ and for $1\leq j\leq n$ let $\nu_{-j}\subset P_{-j}\setminus P_{-(j-1)}$ be the univalent pull-back of $\nu_{-(j-1)}$.
The curve $\nu_{-n}$ is a parametrized loop 
$$\nu_{-n}:[0,1]\to B_0^{f_0}\text{ with }\nu_{-n}(0)=\nu_{-n}(1)=0.$$ By the choice of $n$, the winding number
$$W(\nu_{-n},-1/4)=1.$$
Consider the parameterized curve $$\nu_{-(n+1)}\subset P_{-(n+1)}\setminus P_{-n}$$ such that $\nu_{-(n+1)}(0)=0$ and 
$$f_0(\nu_{-(n+1)})=\nu_{-n}.$$
Evidently,
$$\nu_{-(n+1)}(1)=-1,$$
which is the {\it other} preimage of the parabolic point $z=0$ under $f_0$. To complete the argument, let us now consider the connected component $W$ of
$$B_0^{f_0}\setminus \overline{P_{-(n+1)}}\supset W$$ which contains the ``upper'' preimage of $-1$, the point 
$$z_{-1}\equiv  \frac{-1+\sqrt{3}i}{2}\in\HH,$$
in the boundary. Let $\nu_{-(n+2)}\subset W$ be the next preimage,
$$f_0(\nu_{-(n+2)})=\nu_{-(n+1)},\text{ with }\nu_{-(n+2)}(0)=0\text{ and }\nu_{-(n+2)}(1)=z_{-1}.$$
The inverse branch $\finv{0}$ of the quadratic map $f_0$ univalently extends to a  map 
$$\finv{0}:W\to W.$$ 
For $k\geq 3$ denote $\nu_{-(n+k)}$ the preimage of $\nu_{-(n+k-1)}$ by this branch. 
By Denjoy-Wolff Theorem,
$$\nu_{-(n+k)}\to 0.$$
Furthermore, $\nu_{-(n+k)}$ is disjoint from $P_{-(n+1)}$. For any given ample repelling petal $P_R$
 there exists $k$ such that $\nu_{-(n+k)}\subset P_R$.
Consider the projection 
$$\tl\gamma\equiv\ixp\circ\phi_R^{f_0}(\nu_{-(n+k)}).$$
By construction, it satisfies the properties (1)-(3).
\end{proof}
To help understand the shape of the immediate basin of $F\in\cla$, let us
look at the drawing in \figref{fig-tail}. 
It illustrates the case $F=\cP(f)$ (the reader may think of $f=f_0$, to fix the ideas).
The left figure shows a fragment of the boundary of the immediate basin $B_0^f$
 near the parabolic fixed point $0$. If we look on the right, we see a schematic picture of the 
Jordan domain $\Dom(\cP(f))$. The point $t^{\cP(f)}\in\partial\Dom(\cP(f))$ is the ``tip of the tail'' of the immediate basin $B_0^{\cP(f)}$.
 On the left the reader can see how the tail is formed. The lift of the basin $B_0^f$ fits inside a suitable repelling crescent $C_R^f$,
reaching to the upper tip of the crescent (which under $\ixp\circ \tlphi_R^f$ becomes the parabolic point $z=0$ on the right).
The point $t^{\cP(f)}$ lifts under $\ixp\circ\tlphi_R^f$ to an $f^{-k}$-preimage $w\in C_R^f$ of the parabolic point $0$.
The lift of the immediate basin $B_0^{\cP(f)}$ to $C_R^f$  reaches to $w$; its shape near $w$ is a conformal image of the shape near $0$ (a tail).
The map $\ixp\circ \tlphi_R^f$ is conformal in a neighborhood of $w$, and hence the basin $B_0^{\cP(f)}$ also has a tail ending at $t^{\cP(f)}$.

\begin{figure}
\label{fig-tail}
\centerline{\includegraphics[width=0.7\textwidth]{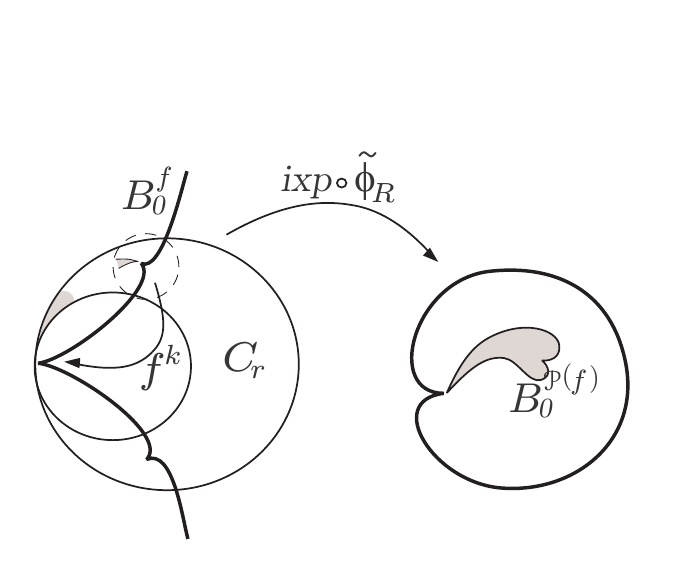}}
\caption{The formation of a ``tail'' in the immediate parabolic basin of $\cP(f)$.}
\end{figure}

\subsection{Definition of $\clad$ and puzzle partitions}
\begin{defn}
We let $\clad$ to be the collection of maps $f\in\claa$ for which the point $t^f$ is accessible from the complement of the domain $\cD(f)$. We thus have
$$\cla\subset \clad\subset \claa.$$
\end{defn}
\thmref{class-1a} will follow from a stronger statement:
\begin{thm}
\label{class-1c}
Let $f\in\clad$, and denote $B^0_f$ the immediate basin of the parabolic point $0$ of $f$.
Then  $B_f^0$ is a Jordan domain. Denote $\hat f$ the continuous extension of $f$ to $\partial B_0^f$. There exists
a homeomorphism $\rho:\partial B_0^f\to\TT$ such that
$$\rho(\hat f(z))=2\rho(z)\mod 1.$$
\end{thm}

Let us make a new definition. Let $f\in\clad$. Consider an attracting petal $P_A^f$ which contains the unique critical value
$c^f\in B_0^f$. Let $\tau\subset B_0^f$ be a simple curve which connects $c^f$ with the parabolic point $0$ and has
the property $f(\tau)\subset\tau$. To fix the ideas,
we will take this curve to be the horizontal ray 
$$\{\Re(z)\geq \Re(\tlphi_A^f(c^f)),\;\Im(z)=\Im(\tlphi_A^f(c^f))\}$$
in the Fatou coordinate. Using \propref{tip1}, and passing from a curve around the critical value to a slit 
connecting the critical value with a parabolic point we see:
\begin{prop}
\label{tip2}
There exists a unique simple curve $\Gamma^{\text{in}}$ connecting $t^f$ with $0$ such that $$f(\Gamma^{\text{in}})=\tau.$$
\end{prop}
\noindent
By definition of $\clad$, the tip of the tail $t^f\in\partial\cD(f)$ is
accessible from outside ${\cD(f)}$. Let us continue $\gin$ by attaching a simple curve $\gout$ connecting $t^f$ 
to $\infty$ without intersecting $\overline{\cD(f)}$. Let us further require that $\gout$ coincides with a negative real ray
in a neighborhood of $\infty$.
\begin{defn}
We call $\Gamma\equiv\overline{\gin\cup\gout}$ the {\it primary cut} of $f$.
\end{defn}

\noindent
We prove the following:
\begin{prop}
\label{contain1}
Let $f\in\clad$. 
There exists a domain $U\subset \CC$ such that:
\begin{itemize}
\item[(1)] $U$ is a Jordan domain;
\item[(2)] $U\subset\Dom(f)$ and $\partial U\cup\partial\Dom(f)=\{ t^f\}$;
\item[(3)] $\partial U\subset f^{-1}(\overline{\gout}\cup \partial\Dom(f))$ and $U\supset \gin$;
\item[(4)] $U\ni c^f$;
\item[(5)] there is a single critical point $p^f\in \gin$ of $f$ inside $U$.
\end{itemize}
Finally, there is a simple arc $\gin_{-1}\subset f^{-1}(\gin)\cap U$ with endpoints $u^+,u^-\in\partial U$ 
such that $\gin_{-1}\cap\gin=\{p^f\}$ for which the following holds:
\begin{itemize}
\item[(6)] consider the Jordan domain $U'$ which is the connected component of $U\setminus \gin_{-1}$ whose
boundary does not contain the point $t^f$. Then 
$$f:U'\setminus \gin\longrightarrow \Dom(f)\setminus \gin$$
is a univalent map.
\end{itemize}
\end{prop}

\begin{figure}
\centerline{\includegraphics[width=1.1\textwidth]{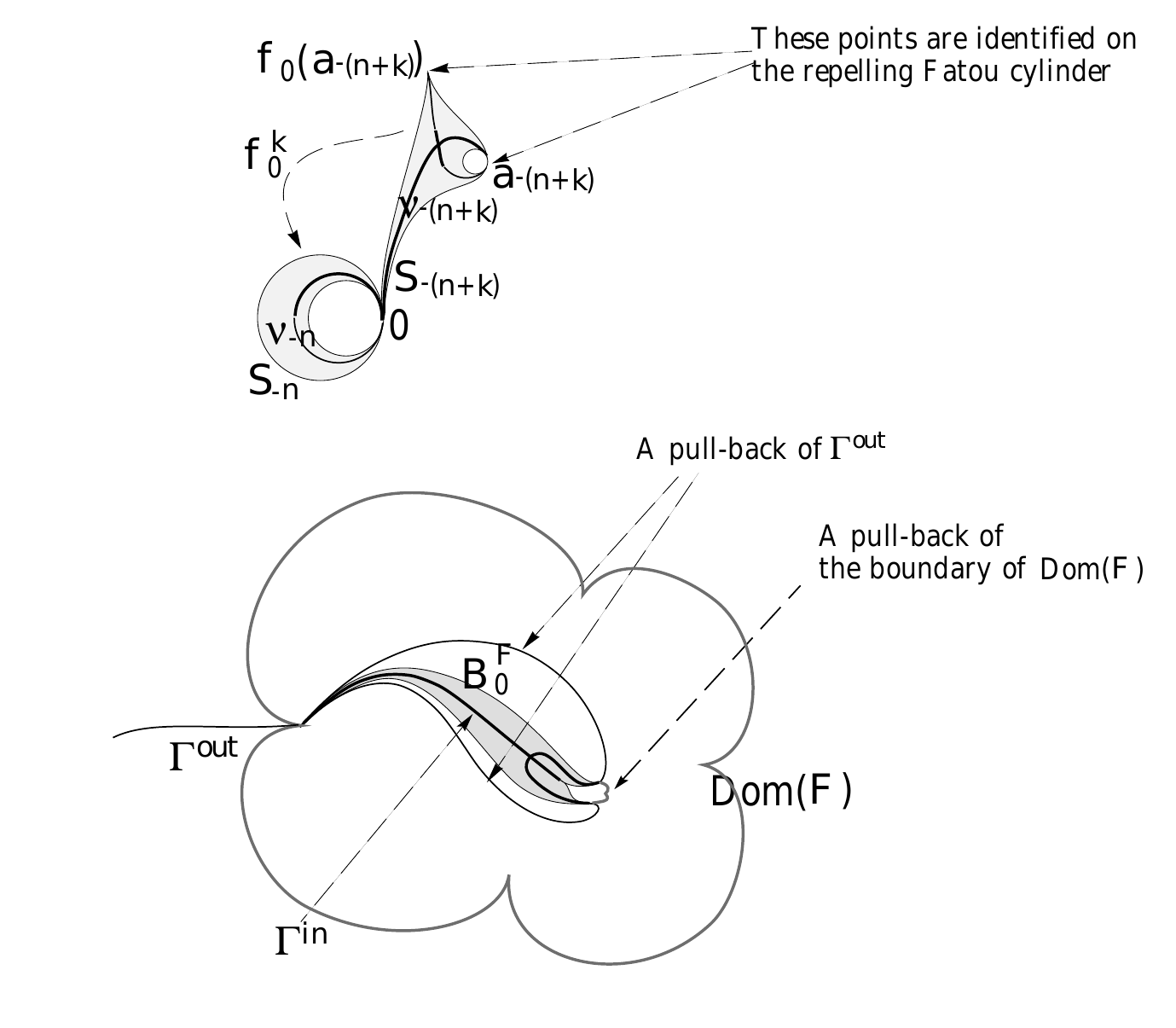}}
\caption{\label{fig-basina}
An illustration to the proof of \propref{contain2}. The pull-back of $\hat\CC\setminus\Gamma$ for $F=\pr(f_0)$. }
\end{figure}

In view of the definition of $\clad$, it suffices to prove a slightly more general statement for $f=F\equiv \pr(f_0)$:
\begin{prop}
\label{contain2}
Let $F=\pr(f_0)$. Consider any simple arc $\gamma=\gamma_1\cup\gamma_2$ and any Jordan curve 
$\tau$ with the following properties:
\begin{itemize}
\item $\gamma_1$ is a simple arc which connects $0$ and $t^F$;
\item $\gamma_2$ is a simple arc which connects $t^F$;
\item the curve $\gamma_1$ approaches $0$ within a sector $\{\Arg(z)\in(\pi/2,3\pi/2)\}$;
\item the curve $\gamma_2$ approaches $\infty$ within a sector $\{\Arg(1/z)\in(\pi/2,3\pi/2)\}$;
\item the curve $\tau$ separates $c^F$ from $\infty$ and $\tau\cap \gamma= \{ t^F\}.$
\end{itemize}
There exists a domain  $U$ of $F^{-1}(V)$ such that:
\begin{enumerate}
\item $U$ is a Jordan domain;
\item $U\subset\Dom(F)$ and $\partial U\cup\partial\Dom(F)=t^F$;
\item $\partial U\subset F^{-1}(\overline{\gamma_2}\cup \tau);$
\item $U$ contains a single critical point $p^F$;
\item there is a simple arc $\delta\subset U$ which is a preimage of $\gamma_1$, has endpoints on the boundary of $U$,
and intersects $\gamma_1$ at the critical point $p^F$ only. If we let  $U'$ be the component of $U\setminus \delta$
not containing $t^F$ on the boundary, then the map
$$F:U'\longrightarrow \Dom(F)\setminus F(\delta)$$
is univalent.
\end{enumerate}
\end{prop}
\begin{proof}
Let us again begin by selecting a simple curve $\sigma\subset\hat\CC$ which connects $0$ and $\infty$, 
and does not intersect with $\gamma\cup \tau$ except 
at the point $0$, and such that an analytic branch of the logarithm defined on a neighborhood of $\sigma$ has bounded imaginary part.

Let us parametrize $\gamma$:
$$\gamma(t):[0,1]\to \hat\CC$$
so that $\gamma(0)=0$ and $\gamma(1)=\infty$.

With a slight abuse of notation we use the same notation for the preimages of $\sigma$ and $\gamma$, etc. as in the proof of 
\propref{tip1}. Let $S_{-n}$ be the fundamental crescent $P_{-n}\setminus \overline{f(P_{-n})}$ which contains the critical value $-1/4$.
The corresponding component $\nu_{-n}\subset \overline{S_{-n}}$ of the lift of $\gamma$ is a simple arc which passes through
$-1/4$, and whose two ends land at $0$. 

There are two prime ends of the point $0$ in the crescent $S_{-n}$. We denote the ``upper'' one by $0_{-n}^+$, and the ``lower'' one by
$0_{-n}^-$. 
They correspond to the points $0$ and $\infty$
respectively in the quotient $\cC_A\simeq \hat\CC$. 
We will write 
$$\nu_{-n}(t)\underset{t\to 0+}{\lra}0^+_{-n}\text{ and }\nu_{-n}(t)\underset{t\to 1-}{\lra}0^-_{-n}.$$
Writing $S_{-(n+k)}=P_{-(n+k)}\setminus P_{-(n+k-1)}$, we see that 
$$f_0:S_{-(n+1)}\to S_{-n}$$
is a double covering, branched at $$-1/2\mapsto -1/4\in \nu_{-n}.$$
We thus obtain two parametrized curves $\nu^1_{-(n+1)}$ and $\nu^2_{-(n+1)}$ as the lifts of $\nu_{-n}$ by $f_0$. The intersection
$$\cir{\nu}^1_{-(n+1)}\cap\cir{\nu}^2_{-(n+1)}=-1/2.$$
For $j\leq n+1$ we denote $0_{-j}^+$, $0_{-j}^-$ the ``upper'' and the ``lower'' prime ends of $0$ in $S_{-j}$,
so that
$$\tl f_0: 0^\pm_{-j}\mapsto 0^\pm_{-(j-1)}.$$
 The other preimage of $0$
in $\partial S_{-(n+1)}$ is the point $-1$. We use $-1_{-(n+1)}^+$ and $-1_{-(n+1)}^-$ to denote its two prime ends in $S_{-(n+1)}$
labeled so that the Carath{\'e}odory extension $\tl f_0$ maps
$$-1_{-(n+1)}^+\mapsto 0^+_{-n}\text{ and }-1_{-(n+1)}^-\mapsto 0^-_{-n}.$$
We have
$$\nu^1_{-(n+1)}(t)\underset{t\to 0+}{\lra}0^+_{-(n+1)}\text{ and }\nu_{-(n+1)}(t)\underset{t\to 1-}{\lra}-1^-_{-(n+1)};$$
$$\nu^2_{-(n+1)}(t)\underset{t\to 0+}{\lra}-1^+_{-(n+1)}\text{ and }\nu_{-(n+1)}(t)\underset{t\to 1-}{\lra}0^-_{-(n+1)}.$$
Denote $S_{-(n+k)}\subset W$ the pull-back of $S_{-(n+k-1)}$ by the inverse branch $\finv{0}$; and 
$$\nu^{1,2}_{-(n+k)}\subset S_{-(n+k)}$$
the corresponding preimage of $\nu_{-n}$. 
A trivial induction shows that there are exactly three points $$\{0,a_{-(n+k)},f(a_{-(n+k)})\}\subset\partial S_{-(n+k)},$$
 which map to $0$ under $f_0^k$. Furthermore,  $a_{-(n+k)}$ has two prime ends in $S_{-(n+k)}$;
the points $0$ and  $f_0(a_{-(n+k)})$ each have a single prime end.
We denote $a_{-(n+k)}^\pm$ the prime end mapped to $0_{-n}^\pm$ by the Carath{\'e}odory extension $\tl f_0^k$.

We have
$$\nu^1_{-(n+k)}(t)\underset{t\to 0+}{\lra}0_{-(n+k)}\text{ and }\nu_{-(n+1)}(t)\underset{t\to 1-}{\lra}f_0(a_{-(n-k)});$$
$$\nu^2_{-(n+k)}(t)\underset{t\to 0+}{\lra}a^+_{-(n+1)}\text{ and }\nu_{-(n+k)}(t)\underset{t\to 1-}{\lra}a^-_{-(n+1)}.$$

Each of the sets $\nu_{-j}$ is composed of two parts, $\nu_{-j,1}$ and $\nu_{-j,2}$ which are the preimages of 
$\gamma_1$ and $\gamma_2$ respectively. For $j\leq n$, they intersect at a single point $b_{-j}$; for $j>n$ there are 
two such points $b^1_{-j}$, $b^2_{-j}$. 

Denote $\kappa_{-(n+k)}\subset S_{-(n+k)}$ the preimage of $\tau$. 
Elementary considerations of monodromy, which we spare the reader, imply that
$\kappa_{-(n+k)}$ is a union of two simple arcs: one connecting $0$ with $b^1_{-(n+k)}$, the other with $b^2_{-(n+k)}$; and otherwise
disjoint from $\nu_{-(n+k)}$.

Considerations of Denjoy-Wolff Theorem imply that for every ample repelling petal $P_R$ 
there exists $k$ such that $S_{-(n+k)}\subset P_R$. 
Let us  project the picture in $S_{-(n+k)}$ back to $\hat\CC$ using $\ixp\circ\phi_R^{f_0}$. The points $a_{-(n+k)}$ and
$f_0(a_{-(n+k)})$ are identified in the projection. We obtain a Jordan domain $U$, bounded by the projections
of $\kappa^1_{-(n+k)}$, $\kappa^2_{-(n+k)}$, $\nu_{-(n+k),2}$. It satisfies the properties (1)-(5) by the construction. 

\end{proof}

\begin{cor}
\label{basinbd}
Let $f\in\clad$, and let $U$ be the domain constructed in \propref{contain1}. We have
\begin{itemize}
\item $B_0^f\subset U$;
\item $\partial B_0^f\cap\partial\Dom(f)=\{t^f\}$.
\end{itemize}
\end{cor}
\begin{proof}
Since $B_0^f$ does not intersect with $\overline{\gout}\cup \partial\Dom(f)$ and 
$\partial U\subset f^{-1}(\overline{\gout}\cup \partial\Dom(f))$, we have 
$B_0^f\subset U$. Hence,
$$\partial B_0^f\cap\partial\Dom(f)\subset\{t^f\}.$$
On the other hand, the cut $\gin\subset B_0^f$, hence
$$t^f\in\overline{\gin}\in\partial B_0^f.$$

\end{proof}

\renewcommand{\finv}[1]{{G_{#1}}}
\renewcommand{\finvloc}{{G}}

\subsection{The immediate basin of a map in $\clad$ is a Jordan domain}
We are now in a position to begin the proof of \thmref{class-1c}, which in turn implies \thmref{class-1a}. 
Let us fix $F\in\clad$ and let $U$ be the domain constructed in
\propref{contain1}. 
The preimage $\overline{F^{-1}(\gin)}\cap B_0^F$
consists of the curve $\gin$ and a simple curve $\gin_{-1}$; the two curves cross at the critical point $p^F\in B_0^F$.
To fix the ideas, let us parametrize $$\gin:[0,1]\to \overline B_0^f$$
so that $\gin(0)=0$ and $\gin(1)=t^f$.
Let $s_1<s_2\in(0,1)$ be such that $$\gin(s_1)=c^F, \;\gin(s_2)=p^F.$$
By elementary path-lifting considerations, the curve $\gin_{-1}$ is a cross-cut in $U$.
Let $V$ be the connected component of $U\setminus \gin_{-1}$ which does not contain $t^F$ in its boundary.
By construction (\propref{contain1}), there exists a univalent branch $G$ of $F^{-1}$ which maps
$$\Dom(F)\setminus \gin([s_1,1])\mapsto V\setminus \gin([s_1,s_2]).$$
We note:
\begin{prop}
\label{in1}
The inverse branch $G$ maps the domain $\Dom(F)\setminus\gin$ inside itself.
\end{prop}

Let $l(t)$ be a simple path such that $l(0)=u$, $l(1)=F(u)\in\partial\Dom(F)$, and 
$\cir{l}\cap \overline{U}=\emptyset.$
Since $\Dom(F)$ is a Jordan domain, the inverse branch $G$ has a continuous extension to the boundary point $u$.
Set $u_0\equiv u$, and denote $u_{-n}$ the orbit of $u_0$ under $G$.
The curve $l$ has a univalent pull-back $l_{-1}\subset \overline{V}$ such that $l_{-1}(0)=u_{-1}$ and $l_{-1}(1)=u.$
We let
$$l_{-n}=G(l_{-(n-1)})\text{ for }n\geq 2.$$
Set
$$\Lambda=\cup_{n\geq 1} l_{-n}.$$
We parametrize this curve by $(0,1]$ so that 
$$\Lambda\left(\left[\frac{1}{j},\frac{1}{j-1}\right]\right)=l_{-(j-1)}\text{ for }j\geq 2.$$
Thus, $\Lambda(1/j)=u_{-j}.$ By \propref{in1} and Denjoy-Wolff Theorem, we have:
\begin{prop}
\label{in2}
The curve $\Lambda$ is disjoint from $B_0^F$, and lands at $0$:
$$\lim_{t\to 0-}\Lambda(t)=0.$$
\end{prop}

\begin{defn}
We call the curve $\hat\Gamma\equiv \Gamma\cup\Lambda$ the {\it secondary} cut.
\end{defn}

Let $u^+,u^-\in\partial U$ be the two endpoints of $\gin_{-1}$, labeled in such a way that $u^+$ is encountered
first, when going around $\partial U$ in the positive direction  from $t^f$. 
Denote the connected components of $U\setminus \hat\Gamma$ by $U^+$, $U^-$ so that $U^\pm\ni u^\pm$.

Consider the inverse branch of $f$ which maps $B_0^F\setminus \gin([0,s_1])$ into $B_0^F\cap U^+$, and let $G_0$
be its univalent extension to
$$G_0:U\setminus (\Lambda\cup\gin([0,s_1]))\hookrightarrow U^+.$$
Replacing $U^+$ with $U^-$ we similarly define an analytic branch of $F^{-1}$ 
$$G_1:U\setminus (\Lambda\cup\gin([0,s_1]))\hookrightarrow U^-.$$

Let $P\subset B_0^F$ be an attracting petal which contains $c^F$ in its interior, and let $P_{-1}$ be a 
connected component of $F^{-1}(P)$
such that
$$P\subset P_{-1}.$$

We set $$\bA\equiv U\setminus P_{-1}.$$

\begin{figure}
\centerline{\includegraphics[width=\textwidth]{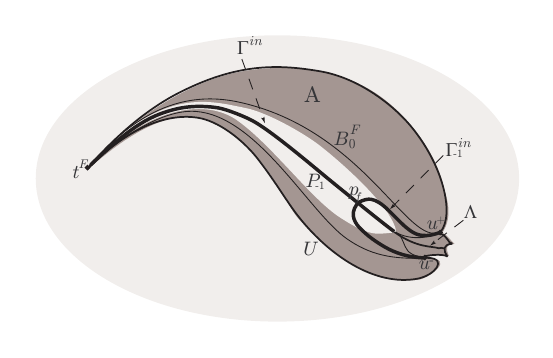}}
\caption{\label{fig-basinb}
The domain $\bA$.}
\end{figure}

Let
\[
\bA_0 \equiv \bA\cap U^+\quad\text{and}\quad
\bA_1\equiv \bA\cap U^-
\]
If $i_0 i_1 \ldots i_n$ is an arbitrary sequence of $n+1$ 0's and 1's,
we define
\[
\bA_{i_0 i_1 \ldots i_n} := \{ z \in \bA_0 : f^j(z) \in
\bA_{i_j}\quad\text{for}\quad 1 \leq j \leq n \}
\]

The proof of \thmref{class-1c} now follows the same lines as in \secref{sec:quadratic-example}.
It follows from the definitions that
\begin{itemize}
\item 
$\bA_{i_0 i_1 \ldots i_n}$ is decreasing in $n$:
$
\bA_{i_0 \ldots i_n} \subset \bA_{i_0 \ldots i_{n-1}};
$
\item
$
F \bA_{i_0 i_1 \ldots i_n} = \bA_{i_1 \ldots i_n};
$

\item
$
\bA_{i_0 i_1 \ldots i_n} = \finv{i_0} \bA_{i_1 \ldots i_n},
$
and hence
\[
\bA_{i_0 i_1 \ldots i_n} = \finv{i_0} \circ \finv{i_1} \circ
\cdots \circ \finv{i_{n-1}} \bA_{i_n},
\]
or, more generally,
\[
\bA_{i_0 i_1 \ldots i_n} = \finv{i_0} \circ \finv{i_1} \circ
\cdots \circ \finv{i_{j-1}}\bA_{i_j i_{j+1} \ldots
  i_n}\quad\text{for $1 \leq j \leq n-1$.}
\]
\end{itemize}

We show that the diameters of the
puzzle pieces $\bA_{i_0 i_1 \ldots i_n}$ go to zero as $n \to
\infty$, uniformly in $i_0 i_1 \ldots i_n$. More precisely,
let us define 
\[
\rho_n := \sup \left\{ \text{diam}(\bA_{i_0 \ldots i_n}) : i_0 \ldots i_n
\in \{0,1\}^{n+1} \right\}
\]
Since $\bA_{i_0 \ldots i_n} \supset \bA_{i_0
  \ldots i_{n+1}}$, the sequence $\rho_n$ is non-increasing in $n$, so 
\begin{equation}
\label{limrho1}
\rho_* :=\lim_{n \to \infty} \rho_n
\end{equation}
exists. 

We will prove 
\begin{prop}\label{th:5a}
The limit $\rho_*=0.$  
\end{prop}

In the same way as \lemref{th:5:1}, we have
\begin{lem}
\label{th:5:1a} 

There is an infinite sequence $i_0 i_1 \ldots i_n
  \ldots$ so that
\[
\diam(A_{i_0 i_1 \ldots i_n}) \geq \rho_* \quad\text{for all
  $n$.}
\]
\end{lem}

\begin{proof}[{\it Proof of \propref{th:5a}}]
Let us assume the contrary: $\rho_*>0$.

We fix a sequence $i_0 i_1 \ldots$ as in the \lemref{th:5:1}, and we
split the proof into three cases:
\begin{enumerate}
\item $i_j$ is eventually 0;
\item $i_j$ is eventually 1;
\item neither of the above holds.
\end{enumerate}

We start with case (3). There are then infinitely many $j$'s so that
\[
i_{j} = 0 \quad\text{and}\quad i_{j+1} = 1
\]
Let $j_k$ be a strictly increasing sequence of such $j$'s. Then, for
each $k$,
\[
f_{-j_k} := \finv{i_0} \circ \finv{i_1} \circ \cdots \circ \finv{i_{j_k-1}}
\]
is an analytic branch of the inverse of $f_0^{j_k}$ mapping $\bA_{01}$
bijectively to $\bA_{i_0 \ldots i_{j_k+1}}$. The closure of
$\bA_{01}$ does not intersect the postcritical set of $f$. 
We now use a version of the  argument we gave the proof of \thmref{th:5}.


We let 
$Q$ be an simply connected open neighborhood of $\overline{\bA_{01}}$
disjoint from the postcritical set. Then each $f_{-j_k}$ extends to
an analytic branch of the inverse of $F^{j_k}$ defined on $Q$ (and
we denote this extension also by $f_{-j_k}$). 


By Montel's Theorem, there is a subsequence of $(j_k)$ along which $f_{-j_k}$
converges uniformly on compact subsets of $U$ in the spherical metric on $\hat\CC$. By adjusting the
notation, we can assume that the sequence $(f_{-j_k})$ itself
converges to an analytic function which we denote $h$. 
Since
\[
\text{diam}(f_{-j_k} \bA_{01}) = \text{diam}(\bA_{i_0 \ldots i_{j_k+1}})
\]
does not go to zero as $k \to \infty$, the function $h$ is non-constant.

Let  $z_0:=\finv{0}(u^-)\in\bA_{01}\cap \partial B_0^f$. By invarince of the basin boundary,
we have $w_0=h(z_0)\in \partial B_0^f$. Since $h$ is non-constant, $h(\bA_{01})$ contains an open neighborhood
$W$ of $w_0$. Let $V\Subset W$ be a smaller open neighborhood of $w_0$.
Then, an arbitrary large iterate $F^{j_k}$ maps $V$ inside $Q$, which is impossible, since $f(u^-)=t^f\in\partial \Dom(f)$.

We turn next to case 1 above, and deal first with the situation $i_j =
0$ for all $j$. We write
\[
\bA^{(n)} := \bA_{\underbrace{0\cdots 0}_{\text{$n$ terms}}}
= (\finv{0})^{n-2}\bA_{00}
\]
Now $\finv{0}$ maps $\bA_{00}$ into itself, and $\finv{0} = \finvloc$
on $\bA_{00}$, so we can write
\[
\bA^{(n)} = (\finvloc)^{n-2} \bA_{00}.
\]
By Denjoy-Wolff Theorem applied to $g$ and local dynamics near the parabolic point $0$,
\[
(\finvloc)^{n-2} \to 0\quad\text{uniformly on $\bA_{00}$},
\]
so
\[
\text{diam}(\bA^{(n)}) \to 0\quad\text{as $n \to \infty$.}
\]

Next consider sequences of the form $i_0 i_1 \cdots i_k 0 0 \cdots$,
and use the formula
\[
\bA_{i_0 \ldots i_k \underbrace{0 \cdots 0}_{\text{$n$ terms}}} =
  \finv{i_0} \circ \finv{i_1} \circ \cdots \circ \finv{i_k} A^{(n)}.
\]
The mapping $\finv{i_0} \circ \finv{i_1} \circ \cdots \circ
\finv{i_k}$ extends to be continuous on $\overline{\bA_{00}}$,
and $\text{diam}(A^{(n)}) \to 0$ by what we just proved, so
\[
\text{diam}(\bA_{i_0 \ldots i_k \underbrace{0 \cdots 0}_{\text{$n$
      terms}}}) \to 0\quad\text{as $n \to \infty$.}
\]

A similar argument, using $\finv{1} = \finvloc$ on $\bA_{11}$ shows
that
\[
\text{diam}(\bA_{i_0 \ldots i_k \underbrace{1 \cdots 1}_{\text{$n$
      terms}}}) \to 0\quad\text{as $n \to \infty$.}
\]
%

\end{proof}

Let $\bm{i}=(i_j)_{j=0}^N$ be a finite or infinite sequence ($N\leq\infty$) of $0$'s and $1$'s.
We interpret it as a binary representation of a number in $[0,1]$:
$$\bin{i}\equiv \sum_{j=0}^N i_j 2^{j}\in[0,1].$$
For any such dyadic sequence,
\[
\overline{\bA_{i_0}} \supset \overline{\bA_{i_0 i_1}} \supset \cdots \supset
\overline{\bA_{i_0 \ldots i_n}} \supset \ldots
\]
is a nested sequence of compact sets in $\CC$ with diameter going
to 0, so its intersection contains exactly one point, which we denote by
$\hat{z}(\bm{i})$. It is immediate from the construction that $\bm{i}
\mapsto \hat{z}(\bm{i})$ is continuous from $\{0,1\}^\ZZ$ to
$\CC$. It is not injective, however, this ambiguity is easily tractable:
\begin{lem}
\label{ambiguitya}
Suppose $\bm{i}=i_0i_1\ldots i_{n-1}$ and  $\bm{i}'=i'_0i'_1\ldots i'_{n-1}$ are two finite dyadic
sequences of an equal length, and
$$\overline{\bA_{\bm{i}}}\cap\overline{\bA_{\bm{i}'}}\neq\emptyset.$$
Then either $\bin{i}=\bin{i'},$ or $\bin{i}=\bin{i'}\pm 2^{-n}.$
\end{lem}
The proof is a straightforward induction in $n$, and is left to the reader. 

As a corollary, we get:
\begin{cor}\label{th:7a}
Let $\bm{i}$ and $\bm{i'}$ be two infinite dyadic sequences.
If $\bin{i} \neq \bin{i'}$, then $\overline{\bA_{i_0
    \cdots i_n}}$ and $\overline{\bA_{i'_0 \cdots i'_n}}$ are disjoint
for large enough $n$
\end{cor}
\begin{proof}
For any
$\bm{i} = i_0 \cdots i_n \cdots$, and any $n$
\[
\vert \bin{i} - \underline{i_0 \cdots i_{n-1}}_2 \vert \leq 2^{-n}
\]
so, if $n$ is large enough so that
\[
\vert\bin{i} - \bin{i'} \vert > 3 \times 2^{-n},
\]
then
\[
\vert \underline{i_0 \cdots i_{n-1}}_2 - \underline{i'_0 \cdots i'_{n-1}}_2 \vert
> 2^{-n},
\]
which by \lemref{ambiguity} implies that
$\overline{\bA_{i_0 \cdots i_{n-1}}}$ and $\overline{\bA_{i'_0 \cdots
    i'_{n-1}}}$ are disjoint, as asserted.
\end{proof}

Putting together what we know about the map $\bm{i}\mapsto \hat z(\bm{i})$, we finally obtain:
\begin{prop}\label{th:6a}
The angles $\bin{i} = \bin{i'}$ if and only if  $\hat{z}(\bm{i}) =
\hat{z}(\bm{i'})$. Hence there is a function $\theta \mapsto z(\theta)$
from the circle $\RR/\ZZ$ to $\CC$ so that
\[
\hat{z}(\bm{i}) = \bin{i}.
\]
The function $z(\,.\,)$ is continuous, bijective, and maps the circle onto
$$J\equiv \overline{\cap_n((F^{-n})|_U\overline{\bA})}.$$
\end{prop}

We thus have:
\begin{cor}
The set $J$ is a Jordan curve.
\end{cor}

\begin{prop}\label{th:9a}
Let $\tl F:J\to J$ be defined as $\tl F(t^F)=0$, and coincide with $F$ elsewhere on $J$.
The map $\bin{i}\mapsto \hat z(\bm{i})$ is a conjugacy between $\theta\mapsto 2\theta$ and $\tl F(z)$:
\[
\tl F(\hat z(\bm{i})) = \hat z(2 \cdot \bin{i}).
\]
\end{prop}
Let us denote $B$ the connected component of $\hat\CC\setminus J$ which contains $c^F$. By construction,
$B\cap B_0^F\neq\emptyset$. By the Maximum Principle, 
$$F:B\to B.$$
By the classification of dynamics on a hyperbolic domain, 
$$B\subset B_0^F.$$
On the other hand, preimages of $t^f$ are dense in $J=\partial B$, and hence
$$B_0^F\cap J=\emptyset.$$
We conclude:
$$B=B_0^F,$$
and thus \thmref{class-1c} is proven. 

We illustrate the topological structure of the parabolic basin of a map $f\in\pr(\cla)$ in \figref{fig-basin}.
We have indicated the immediate basin $B_0$ as well as several components of $f^{-1}(B_0)$. To better 
understand those, we have drawn the preimage $f^{-1}(C)$ of the circle 
$$C=\{|z|=|c|\},$$
which cuts across the only critical value $c$ of $f$. These critical level
curves partition the domain $W$ into univalent preimages 
of the disks $D^+=D_{|c|}(0)$ and $D^-=\hat\CC\setminus \overline{D^+}$.

Two smooth  branches in the  preimage $f^{-1}(C)$ may intersect at a ramification point of $f$ (all of the
ramification points are simple). Each {\it critical component} of $f^{-1}(B_0)$ has one of these branch points,
and touches $\partial W$ in two points (``preimages'' of the asymptotic value $0$).
The {\it non-critical components} of $f^{-1}(B_0)$ arrange themselves in sequences along the critical level curves,
as shown.

\begin{figure}
\label{fig-basin}
\centerline{\includegraphics[width=0.8\textwidth]{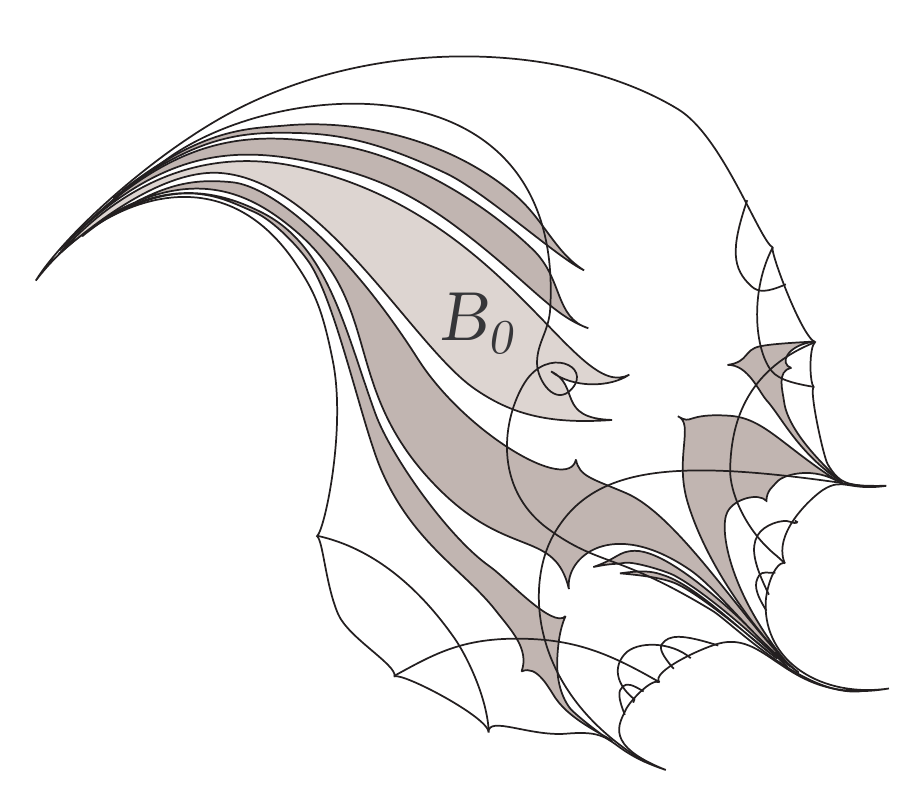}}
\caption{Several components of the parabolic basin of $f\in\cla$}
\end{figure}

\subsection{Convergence of parabolic renormalization}

By definition of the class $\cla$, every map $f\in\cla$ can be decomposed as 
$$f=v\cdot \cP(f_0)\circ \varphi_f(e^{2\pi i\theta}z):\Dom(f)\to\hat\CC,$$
where $v=c^f/c^{\cP(f_0)}.$
We will denote $\psi_f=\varphi_f^{-1}$. Thus, $\varphi_f$ conformally maps the unit disk $\DD$ onto $W_f=\text{Dom}(f)$
with $\psi_f(0)=0$ and $\psi_f'(0)>0$.
We will topologize $\cla$ by identifying it with the space of thus normalized conformal maps of the unit disk:
$$f\mapsto \psi_f,$$
equipped with the compact-open topology.

Let us state an obvious consequence of the Koebe Distortion Theorem:
\begin{lem}
\label{pre-compact}
Let $\cS$ be a family of univalent maps $h:\DD\to\hat\CC$ with  $h(0)=0$. Assume further that
there exist positive constants $0<a<b$ such that $a\leq |h'(0)|\leq b$ for every $h\in\cS$. Then the
family $\cS$ is equicontinuous.
\end{lem}

H. Inou and M. Shishikura \cite{IS} have recently demonstrated the following:

\begin{thm}[\cite{IS}]
\label{thm-IS}
There exists a class $\mathbf F$ of analytic maps with a simple parabolic fixed point at the origin, such that the 
following properties hold:

\begin{itemize}
\item $\cP({\mathbf F})\subset ({\mathbf F})$;
\item there exists a map $f_*\in\mathbf F$ which is a fixed point of the parabolic renormalization:
$$\cP(f_*)=f_*;$$
\item denoting $f_0(z)=z+z^2$, we have $\cP(f_0)\in\mathbf F$, and 
$$\cP^n(f_0)\to f_*\text{ in }\mathbf F;$$
\item in fact, there is a structure of an infinite-dimensional complex-analytic manifold on $\mathbf F$ which is
compatible with the local-uniform norm, in which $\cP$ is a contraction. 
\end{itemize}

\end{thm}

We remark:
\begin{cor}
The fixed point $f_*$ has an analytic extension to a mapping in $\cla$ which we will denote in the same way.
It is also fixed under $\cP$, considered as a transformation $\cla\to\cla$.
\end{cor}
\begin{proof}
By \lemref{pre-compact}, the sequence of parabolic renormalizations $\cP^n(f_0)$ is pre-compact in $\cla$. 
By \thmref{thm-IS}, every limit point of this sequence coincides with $f_*$ on a neighborhood of the origin.
\end{proof}

\section{Numerical results}

\subsection{Properties of the asymptotic expansion of the Fatou coordinates}
A justification of the use of the asymptotic expansion of the Fatou coordinate
(\thmref{asym-expansion-2}) for computation has a theoretical basis in the 
following. Recall (see e.g. \cite{Ram}) that a formal power series
$\sum_{m=1}^\infty a_mx^{-m}$ is of {\it Gevrey order }$k$ if
$$|a_m|<CA^n(n!)^{\frac{1}{k}}\text{ for some choice of positive constants }C, A.$$
As was shown by {\'E}calle \cite{Ec}:

\begin{thm}
\label{gevrey-order}
The asymptotic expansion of \thmref{asym-expansion-2} is of Gevrey order $1$. 
\end{thm}

Thus, the asymptotic series of $g(w)$ is rather slowly divergent. A Stirling formula
estimate suggest that the first $n$ terms of the series are useful in estimating the 
value of $g$ for $|w|>\text{const}\cdot n$.

\thmref{gevrey-order} is a part of {\'E}calle's theory of {\it resurgence} as applied specifically to Fatou coordinates
(see \cite{Sau} for an account). Recall, that the {\it Borel transform} of a formal power series
$\sum_{m=1}^\infty a_mx^{-m}$ consists in applying the termwise inverse Laplace transform:
$$a_m x^{-m}\mapsto \frac{a_m\zeta^{m-1}}{(m-1)!}.$$
In the case when the formal power series is of Gevrey order $1$, this yields a series
$$\sum_{m=1}^\infty \frac{a_m\zeta^{m-1}}{(m-1)!},$$
which converges to an analytic function $\hat h(\zeta)$ in a neighborhood of the origin.
Assume further, that the function $\hat h$ extends analytically along the positive real axis, and
the extension is of an exponential type. Then, the Laplace transform 
$$h(x)=\int_0^\infty e^{-x\zeta}\hat h(\zeta)d\zeta$$
is defined when $\Re x$ is sufficiently large. The original formal power series
$\sum _{m=1}^\infty a_m x^{-m}$ is the asymptotic series of $h(x)$, which should thus 
be considered a summation of the divergent formal series.

Resurgence theory implies much more detailed information about the asymptotic expansion 
of the Fatou coordinate given in \thmref{asym-expansion-2}, than that required to perform the Borel summation.
In particular, the function $\hat h$ analytically extends not just along the positive real axis, but along
every path in $\CC\setminus \{2\pi i n,\; n=\pm 1,\pm 2,\ldots\}.$ 
The attracting Fatou coordinate is obtained using the Laplace transform as described above. The same
function $\hat h$ also produces the repelling Fatou coordinate: by using the Laplace transform 
$$\int_{-\infty}^0 e^{-x\zeta}\hat h(\zeta)d\zeta$$
for $x$ with a large negative real part. The development of the resurgence theory for Fatou coordinates
was started by {\'E}calle in \cite{Ec},  and was largely completed in the recent work of Dudko and Sauzin \cite{DuSa}.

\subsection{Computational scheme for $\cP$}

Having mentioned the resurgent properties of the asymptotic expansion of the Fatou coordinate, we
proceed to describe the computational scheme for $\cP$. We begin with a germ of an analytic mapping
$$f(z)=z+z^2+O(z^3)$$
defined in a neighborhood of the origin. Applying the change of coordinates $w=\kappa(z)=-1/z$, we obtain
$$F(w)=w+1+\frac{A}{w}+O\left(\frac{1}{w^2}\right)$$
defined in a neighborhood of $\infty$. 
We again use the notation $\Phi_A(w)$ for the function which conjugates $F$ with the unit translation
$$\Phi_A(F(w))=\Phi_A(w)+1$$
for $\Re w>> 1$. We let $\Phi_R(w)$ be the solution of the same functional equation for $\Re w<<-1$.
These changes of coordinate are well-defined up to an additive constant, and 
$$\tlphi_A(z)=\kappa^{-1}\circ \Phi_A\circ \kappa(z),\; \tlphi_R(z)=\kappa^{-1}\circ \Phi_R\circ \kappa(z).$$

As we have seen  in \thmref{asym-expansion-2}, the function $\Phi_A(w)$ has an asymptotic development
$$\Phi_A(w)\sim w- A\log w + \text{const}_A +\sum_{k=1}^\infty b_k w^{-k}.$$
The coordinate $\Phi_R(w)$ has an {\sl identical} asymptotic development, differing only by the value of
$\text{const}_R$. While this may seem surprising at first glance, recall that these functions are Laplace transforms of {\sl different}
analytic continuations of the Borel transform of the same divergent series (plus the $w-A\log w+ \text{const}$ term).
 
We select a large integer $M$ (in practice, $M\approx 100$). We will use the asymptotic expansion to estimate
$\Phi_A(w)$ for $w\geq M$ and $\Phi_R(w)$ for $w\leq -M$. 
Consider an iterate $N\approx 2M$ such that
$$\Re F^N(w)\geq M\text{ for }\Re w\in[-M-1,-M].$$
Denote $\nu(z)$ the function
$$\nu(z) = \ixp \circ \Phi_A \circ F^N\circ (\Phi_R)^{-1}\circ \ixp^{-1}(z).$$
It differs from the parabolic renormalization $\cP(f)$ only by rescaling the function and its argument:
$$\cP(f)(z)=a_1\nu(a_0 z).$$
Now consider a contour $\Gamma$ connecting $w=-M-1+iH$ with $F(w)\approx -M+iH$ which is mapped onto the circle
$S_\rho=\{|z|=\rho\}$ for a small value of $\rho$ by $\ixp \circ \Phi_R$. 
Select $n\in\NN$ and consider the $n$ points in $S_\rho$ given by $x_k=\rho \exp(2\pi k/n),\; k=0,\ldots,n-1$.
We then evaluate the first $n$ coefficients in the Taylor expansion of $\eta$ at the origin
$$\eta(z)=\sum_{j=0}^\infty r_j z^j$$
using a discrete Fourier transform. Specifically, we calculate 
$$s_k=\nu(x_k)\approx \sum_{j=0}^{n-1}r_j (x_k)^j=\sum_{j=0}^{n-1}r_j\rho^j \exp(2\pi kj/n),$$
and apply the inverse discrete Fourier transform:

$$r_j\approx \frac{1}{n\rho ^j}\sum_{k=0}^{n-1} s_k \exp(- 2\pi kj/n).$$

Since 
$$\cP(f)(z)=\sum_{j=1}^\infty s_ja_1a_0^jz^j,$$
we have $$a_1a_0s_1=1,\text{ and further }a_0=\frac{s_1}{s_2}.$$
This step completes the computation of the Taylor expansion of $\cP(f)$.

\subsection{Computing $f_*$}
In computing the fixed point $f_*(z)$ we find it more convenient to work with the representation of 
a germ $f(z)=z+z^2+\cdots$ in the form:
$$f(z)=z\exp(f_{\log}(z)),$$
where $f_{\log}$ is a germ of an analytic function at the origin with $f_{\log}(z)=z+\cdots.$
We then rewrite the parabolic renormalization operator in terms of its action on $f_{\log}$:
$$\cP_{\log} (f_{\log})(z)=(2\pi i)^{-1}\Phi_A\circ F^{N}\circ (\Phi_R)^{-1}\circ \ixp^{-1}(z)-\ixp^{-1}(z).$$
This helps to avoid the round-off error which arises from the growth of $f_*$ near
the boundary $\partial \operatorname{Dom}(f_*).$

Modifying the scheme described above to the operator $\cP_{\log}$, we calculate the fixed point by iterating $\cP$ starting at $f_0(z)=z+z^2$:
\begin{empirical}
$$f_*(z)\approx z+z^2+0.(514 -0.0346i)z^3+\cdots$$
\end{empirical}
Our calculations appear reliable up to the size of the round-off error in double-precision arithmetic ($\sim 10^{-14}$)
in the disk of radius $r=5$ around the origin.
As we will see below, the true radius of convergence for the series for $f_*$ is approximately $41$ (see the Empirical Observation \ref{emp1}).

We also estimated the leading eigenvalue of $D\cP|_{f_*}$:
\begin{empirical}The eigenvalue of $D\cP|_{f_*}$ with the largest modulus is 
$$\lambda\approx -0.017+ 0.040i,\; |\lambda|\approx 0.044.$$
\end{empirical}
The small size of $\lambda$ explains the rapid convergence of the iterates of $\cP$ to the fixed point.
To obtain this estimate, we write 
$$f(z)=z+z^2+\sum_{k=3}^\infty \text{coeff}_k(f) z^k,$$
and consider the spectrum of the $N\times N$ matrix $A=(a_{ij})_{i,j=3..N+3}$, with
$$a_{ij}=\frac{\text{coeff}_j(\cP(f_*+\eps z^i))-\text{coeff}_j(f_*)}{\eps},$$
which serves as a finite-dimensional approximation to $D\cP|_{f_*}$.

\begin{figure}
\label{fig-first-renorm}
\centerline{\includegraphics[height=0.7\textheight,angle=-90]{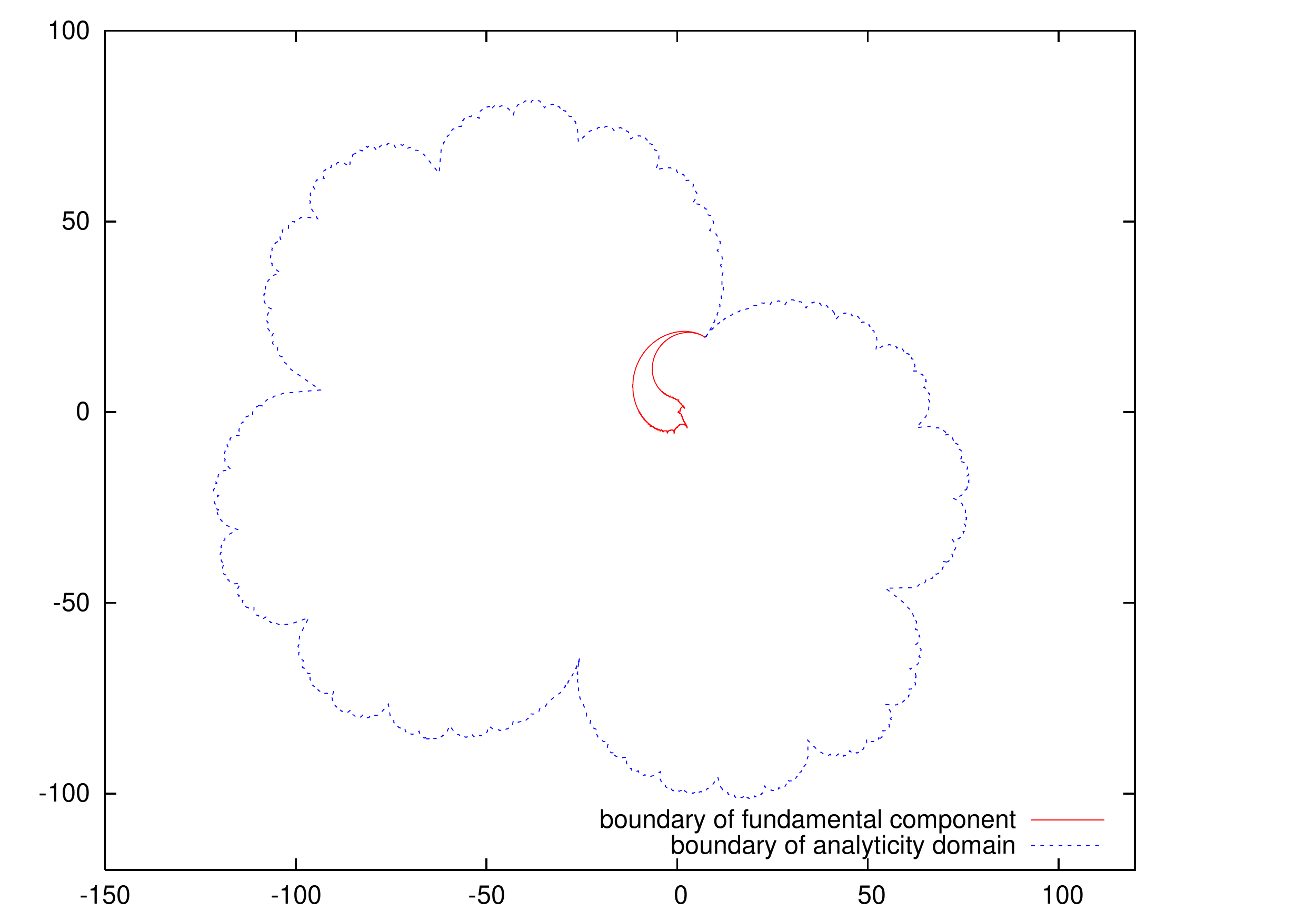}}
\centerline{\includegraphics[height=0.6\textheight,angle=-90]{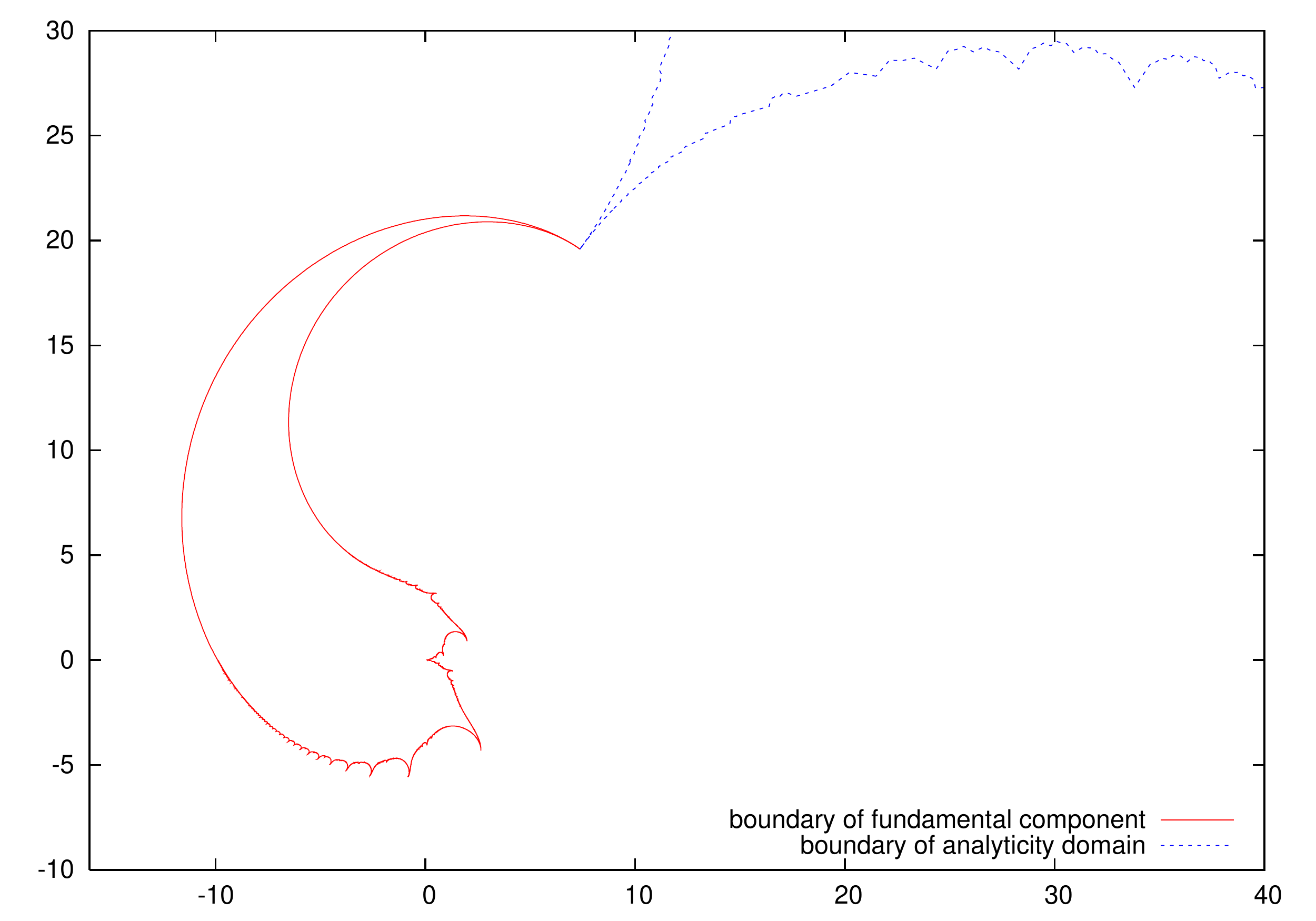}}
\caption{The domain of analyticity of $\cP f_0(z)$ for $f_0(z)=z+z^2$, with the
immediate parabolic basin indicated.
}
\end{figure}

\subsection{Computing a tail of the domain $\text{Dom}(f_*)$}$\-$

\medskip

\noindent{\bf Computing the tail using an approximate self-similarity near the tip.}
Let us denote 
$$t_*\equiv t^{f_*}=\partial\text{Dom}(f_*)\cap\overline{B_0^{f_*}}$$
the endpoint of the tail of the immediate basin of $f_*$.
Let $C_R$ be a repelling fundamental crescent of $f_*$, and let $w\in C_R$ have the property
$$t_*=\ixp\circ \tlphi_R(w).$$
Let $k\geq 2$ be such that 
$$f^k_*(w)=0,\text{ so that }f^{k-1}_*(w)=t_*.$$
Denote $\chi$ the local branch of $f_*^{-(k-1)}$ which sends   $t_*$ to $w$. 
Then the composition
$$\nu\equiv \ixp\circ \tlphi_R\circ \chi$$
is an analytic map defined in a neighborhood of the endpoint $t_*$, which fixes it:
$$\nu(t_*)=t_*.$$
This point can be found numerically:
\begin{empirical}
$$t_*\approx -779.306-643.282i,\text{ and }\nu'(t_*)\approx 0.232+0.264i.$$
\end{empirical}
Thus, we have identified the endpoint of the largest tail of $\text{Dom}(f_*)$. This construction
also gives us the means to compute the tail itself. This can be done  by successively applying $\nu$ to the 
immediate basin $B_0^{f_*}$ thus pulling it in towards $t_*$.

Now let $q\in C_R$ be any other preimage of $0$:
$$f_*^l(q)=t_*\text{ for some } q\in\NN.$$
Then $v=\ixp\circ\tlphi_R(q)$ is the endpoint of a different tail in $\partial\text{Dom}(f_*)$. It can be 
computed by first pulling back the tail of $B_0^{f_*}$ using the inverse branch
$$f_*^{-l}:t_*\mapsto q,$$
and then applying $\ixp\circ\tlphi_R.$

\begin{figure}
\label{fig-xi}
\centerline{\includegraphics[width=1.2\textwidth]{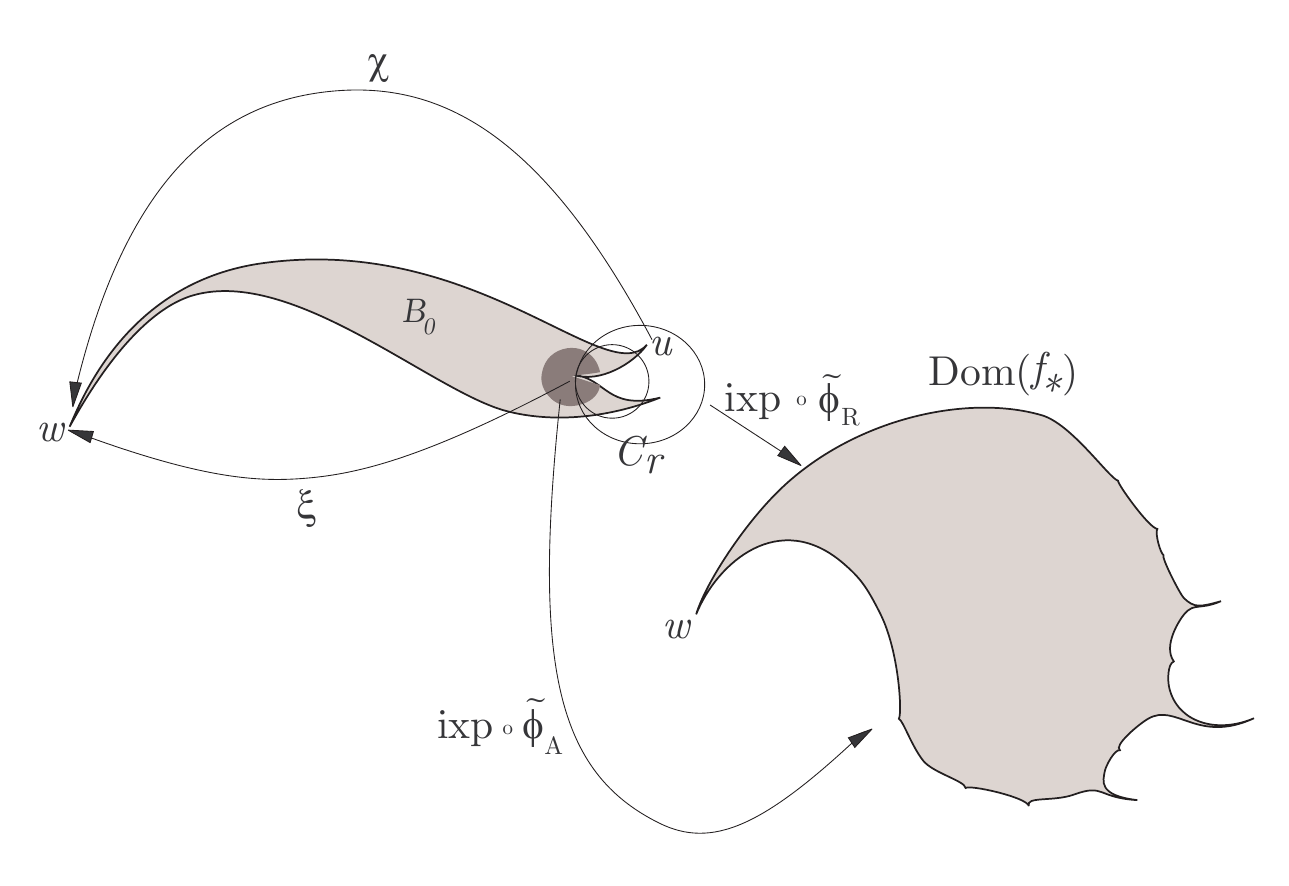}}
\caption{The inverse branches used in computing the tail of $\text{Dom}(f_*)$.}
\end{figure}

\medskip

\noindent{\bf Computing the tail using the functional equation for an inverse branch.}
A more careful analysis of the tail can be done as follows. Denote $\xi$ the local branch of $f_*^{-1}$ defined
in a slit neighborhood $D_r(0)\setminus [0,r)$ for some small value of $r$, which sends $0\mapsto t_*$. We can write
the renormalization fixed point equation for this particular branch:
\begin{equation}
\label{renorm inverse}
\xi=\psi_R\circ \xi\circ \psi_A^{-1},
\end{equation}
where $\psi_R=\chi\circ \ixp\circ\tlphi_R$, and $\psi_A^{-1}$ is the appropriately chosen branch of 
$(\ixp\circ \tlphi_A)^{-1}$ (thus the ``self-similarity'' of the tail is exponential, rather than linear). 
We are going to use the renormalization equation (\ref{renorm inverse})
{\sl inductively} to compute $\xi(z)$ for sufficiently small values of $z$, and thus plot the tail.

\medskip

\noindent{\bf Representing the numbers in the image of the tail.}
Numerical computations indicate that the value of $r=0.0002$ is sufficiently small for our needs, and for $|z|<r$
the difference between the left and the right sides of (\ref{renorm inverse}) is of the order of $10^{-11}$. 
The values of $z$ for which we would like to evaluate $\xi(z)$ become too small to be represented by the standard
double precision numbers (and even too small for their logarithms to be so represented).
We write
$$s(t)=\exp(2\pi t),$$
and choose $\hat t$ so that
$$\exp(-2\pi \hat t)=0.0002,\text{ that is }\hat t=1.3555...$$
We then represent a small positive number $x$ as
$$x=\frac{1}{s^k(t)},$$
for the unique choices of $t\in[\hat t,s(\hat t)),$ and an iterate $k\in\NN$.

We can write any complex number $z$ with $|z|<r$ uniquely as
$$z=(k,t,\theta)\equiv \frac{\exp(2\pi i\theta)}{s^k(t)},\; 0\leq\theta<1.$$
Note that this representation of small numbers makes it very easy to compute logarithms. In particular,
$$\ixp^{-1}((k,t,\theta))=\theta+is^{k-1}(t).$$
The next step in applying (\ref{renorm inverse}) is to apply $\phi_A^{-1}$ to the right-hand side of the equation.
From the first two terms in the asymptotics of
 $$\phi_A(z)=-\frac{1}{z}+O(\log|z|)\text{ for small }z,$$
it follows that
$$\phi_A^{-1}(y)=-\frac{1}{y+O(\log|y|)}\text{ for large }|y|.$$
A numerical estimate shows that for $|y|\geq 10^{18}$ the $O(\log|y|)$ term dissapears into the round-off error,
when added to $y$. Thus
$$\psi_A^{-1}((k,t,\theta))\approx -\frac{1}{\theta+i s^{k-1}(t)}\approx is^{k-1}(t)=(k-1,t,1/4)$$
provided $s^{k-1}(t)\geq 10^{18}.$
A direct estimate shows that for either $k\geq 3$, or $k=2$ and $t>18\log 10/2\pi\approx 6.596$
the last inequality will hold. 

\begin{figure}
\label{fig-fixed-pt}
\centerline{\includegraphics[height=0.5\textheight,angle=-90]{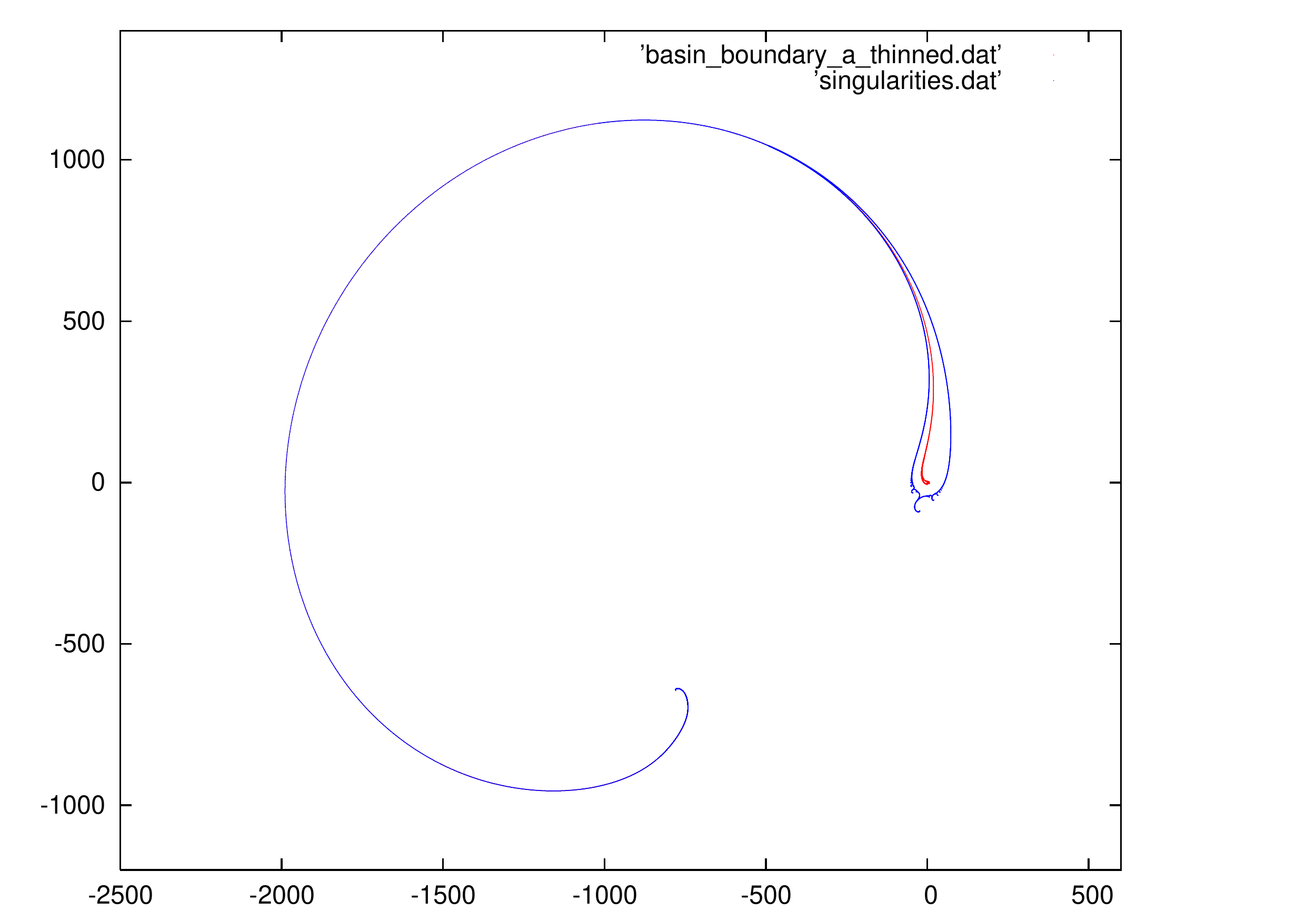}}
\centerline{\includegraphics[height=0.5\textheight,angle=-90]{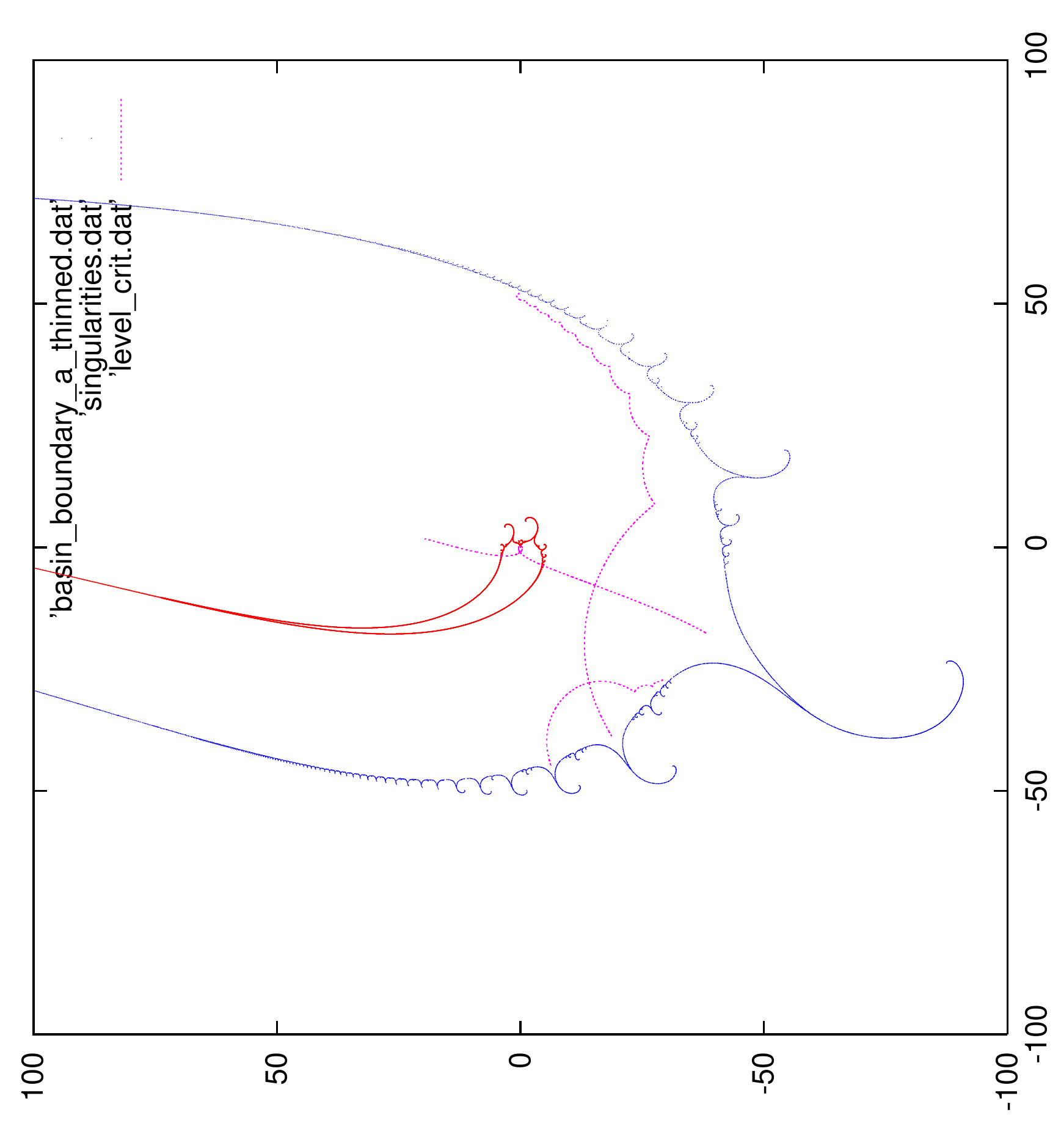}}
\caption{The domain of analyticity of $f_*$ and the boundary of the immediate parabolic basin $B_0^{f_*}$. In the second figure, a part of the
critical level curve of $f_*$ is also indicated.
}
\end{figure}

\begin{figure}
\label{fig-fixed-pt2}
\centerline{\includegraphics[height=0.7\textheight,angle=-90]{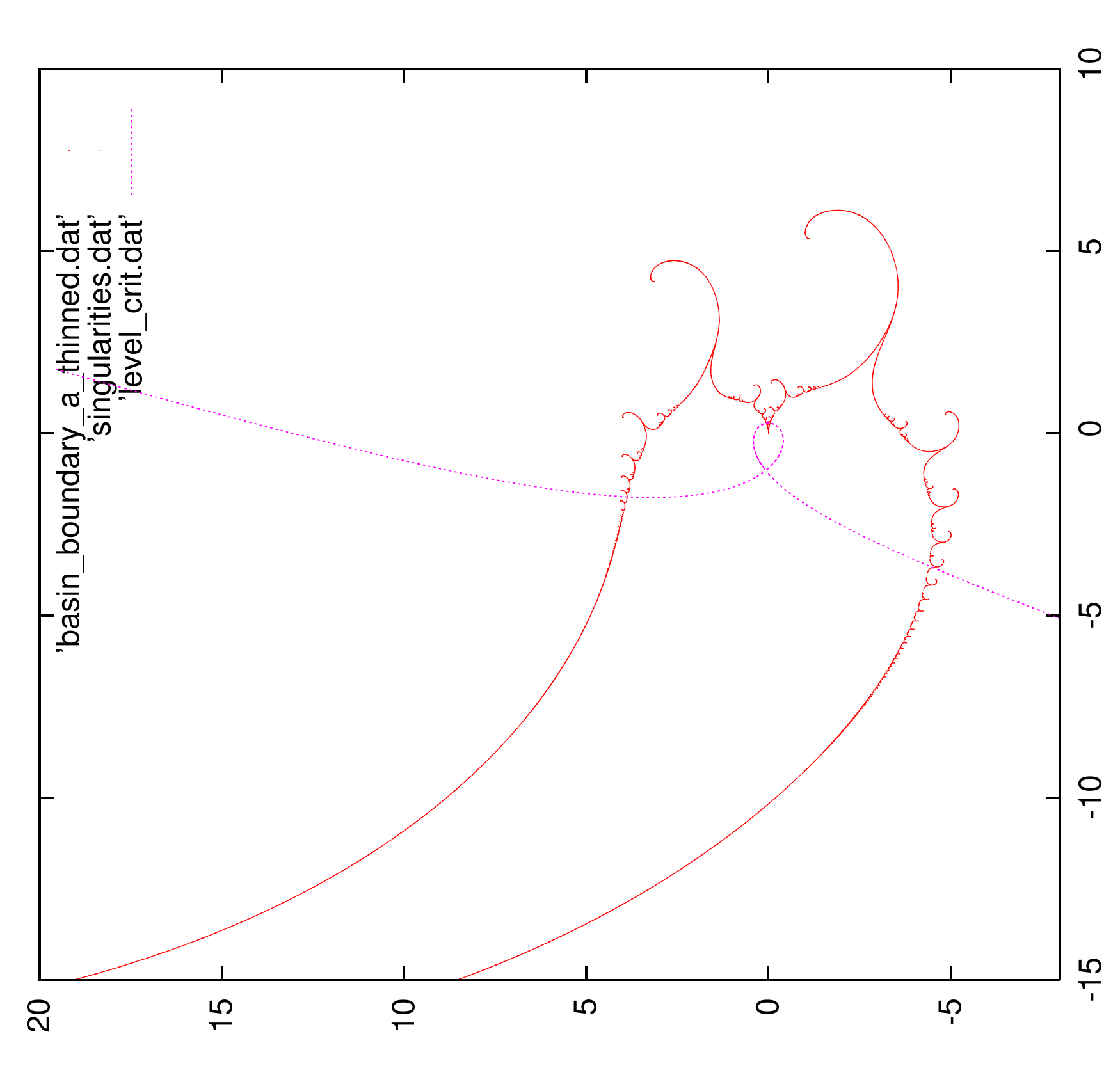}}
\caption{A blow-up of the boundary of the immediate basin of $f_*$ in the vicinity of the parabolic point.}
\end{figure}

\medskip
\noindent
{\bf The size of the domain of analyticity.}
To draw the pictures of the domain of analyticity of the fixed point of $f_*$ (Figures \ref{fig-fixed-pt} and \ref{fig-fixed-pt2})
we employed the following strategy.
First, a periodic orbit of period $2$ in $\partial B_0$ was identified. Its preimages give a rough outline of $\partial B_0$,
but become sparse near the ``tails'', which are not visible in this initial outline. At the next step, the large ``tail'' of $B_0$ is
computed as described above. Finally, its preimages are used to fill in the remaining gaps in $\partial B_0$.

As the final step, we calculate the boundary of $\operatorname{Dom}(f_*)$ as
$$\partial \operatorname{Dom}(f_*)=\ixp \circ \tlphi_R(\partial B_0\cap P_R).$$
An empirical estimate of the inner radius of $\text{Dom}(f_*)$ around the origin allows us to formulate the following observation
 (see \figref{fig-fixed-pt2}):

\begin{empirical}
\label{emp1}
The radius of convergence of the Taylor's expansion of $f_*$ at the origin is $R\approx 41.$
\end{empirical}

\appendix
\section{Example of a map with a simply-connected parabolic basin whose boundary is not locally connected}

\begin{thm}
\label{th:example}
There exists a quadratic rational map $R:\hat\CC\to\hat\CC$ of degree $2$ with the following properties:
\begin{itemize}
\item $R$ has a simple parabolic fixed point at $\infty$ with a proper immediate basin $B_0^R$ of degree $2$;
\item the boundary of the immediate basin $B_0^R$ is not locally connected.
\end{itemize}

\end{thm}

We begin by recalling:
\begin{prop}
\label{cremer}
There exist $\alpha\in \RR/\setminus\QQ$ such that no fixed point of multiplier $e^{2\pi i\alpha}$ for a rational function
of degree $d$ can be locally linearizable.
\end{prop}
The first proof of this result is due to Cremer [1927], who gave a sufficient condition for $\alpha$ (see \cite{Mil}, Theorem 11.2).

Let us fix $\alpha$ as in \propref{cremer} and set $\lambda=e^{2\pi i\alpha}$.

\begin{prop}
\label{defn-map}
There exists a quadratic rational map $R$ with  a simple parabolic point with multiplier $1$ at $\infty$, and
a Cremer point with multiplier $\lambda$ at $0$.
\end{prop}
\begin{proof}
The reader may find a detailed discussion of the dynamics of quadratic rational maps
in Milnor's paper \cite{Milnor-remarks}. Below we give a brief summary of some relevant
facts.
Every quadratic rational map $F$ has three fixed points, counted with multiplicity. Let $\mu_1$, $\mu_2$, $\mu_3$
denote the multipliers of the fixed points.
$$\sigma_1=\mu_1+\mu_2+\mu_3,\;\sigma_2=\mu_1\mu_2+\mu_1\mu_3+\mu_2\mu_3,\;\sigma_3=\mu_1\mu_2\mu_3$$
be the elementary symmetric functions of these multipliers. \vspace{0.12 in}\\
{\bf Proposition }(\cite{Milnor-remarks}, Lemma 3.1). {\it The numbers $\sigma_1$, $\sigma_2$, $\sigma_3$ determine 
$F$ up to a M\"{o}bius conjugacy, and are subject only to the restriction that
$$\sigma_3=\sigma_1-2.$$
Hence the moduli space of quadratic rational maps up to M{\"o}bius conjugacy
 is canonically isomorphic to $\CC^2$, with coordinates $\sigma_1$ and $\sigma_2$.}
\vspace{0.12 in}

Note that for any choice of $\mu_1$, $\mu_2$ with $\mu_1\mu_2\ne 1$
there exists a quadratic rational map $F$, unique up to a M\"obius conjugacy, which has distinct fixed points
with these multipliers. The third multiplier can be computed as $\mu_3=(2-\mu_1-\mu_2)/(1-\mu_1\mu_2)$.

Thus there exists a quadratic rational map $R$ with fixed points at $0, \infty$ such that $R'(0)=\lambda$ and $R'(\infty)=1$.
The map $R$ has only two critical orbits, and at least one of them has to contain the Cremer fixed point $0$ in its closure.
Thus the parabolic basin of $\infty$ can contain only one critical value, and hence $\infty$ is a simple parabolic point.

\end{proof}
Let us fix $R$ as in \propref{defn-map}.
The following is an immediate consequence of Montel's Theorem:
\begin{prop}
The Julia set 
$$J(R)=\partial B_0^R.$$
\end{prop}

Let us now argue by way of contradiction and assume that $\partial B_0^R$ is locally connected. It is elementary to see that
$B_0^R$ is simply-connected. Denote $D:\TT\to\TT$ the doubling map $D(x)=2x\mod 1$.

\begin{prop}
\label{quotient}
There exists a continuous surjective map $\gamma:\TT\to \partial B_0^R$ such that
$$\gamma\circ D=R\circ \gamma.$$
\end{prop}
\begin{proof}
We use the fact that $R$ is conformally conjugate to the Koebe function $K$ on the immediate basin $B_0^R$,
$$\psi\circ R\circ \psi^{-1}=K,$$
and apply Carath{\'e}odory Theorem to the conformal map $\psi$.
\end{proof}
Let us denote $W=\gamma^{-1}(0).$ We claim:
\begin{prop}
The set of angles $\gamma$ is finite.
\end{prop}
To prove this first note that $R$ is a local homeomorphism at $0$ and therefore it induces a homeomorphism on $W$.
The proposition now follows from the following general fact (see e.g. \cite{Mil}, Lemma 18.8):
\begin{lem}
Let $X$ be a compact metric space and let $h:X\to X$ be a homeomorphism. If $h$ is expanding then $X$ is finite.
\end{lem}
We conclude:
\begin{cor}
\label{periodic angle}
There exists a periodic angle $x\in B$.
\end{cor}

We are now set to prove \thmref{th:example}:
\begin{proof}
Let $\psi$ be as in \propref{quotient} and let 
$$\ell=\psi^{-1}(\RR)\subset B_0^R.$$
Let us denote $U^+$, $U^-$ the components of $B_0^R\setminus \ell$. To fix the ideas, assume that the angle $x$
from \propref{periodic angle} corresponds to a prime end $p$ in the Carath{\'e}odory completion of $U^+$, and has period $m$
under the doubling map.
Let $Q:U^+\to U^+$ be the branch of $R^{-m}$ whose Carath{\'e}odory extension fixes $p$. 
By Denjoy-Wolff Theorem, all orbits of $Q$ converge to 
$0$. Hence, $0$ has a non-empty basin under $Q$, and thus is either repelling or parabolic for $R$, which 
contradicts our assumptions.
\end{proof}

We finally note, that up to a conjugacy by $z\mapsto az$ with $a\in\CC^*$, our map $R$ has the form
$$R(z)=\frac{z^2+\lambda z}{z+1}.$$

\section{Example of a map in $\claa\setminus \cla$.}
We begin by showing:
\begin{thm}
\label{th-ex-2}
There exists a quadratic rational map $R$ with a simple parabolic fixed point whose immediate basin
$B_0^R$ is simply connected and has a locally connected boundary, which is not a Jordan curve.

\end{thm}

We find our example in the family of quadratic rational maps
$$R_a(z)=\frac{a}{z^2+2z}$$
which was considered in \cite{AY}. For $a\neq 0$ the map $R_a$ has two simple critical points $c_1=-1$ and $c_2=\infty.$ The latter
one is periodic with period 2:
$$\infty\mapsto 0\mapsto \infty.$$
Fix $a_0=32/27$ and let 
$R\equiv R_{a_0}$. The rational function $R$ has a parabolic fixed point $z_0=-\frac{4}{3}$ with multiplier $1$. Since
there is a single critical orbit of $R$ which may be attracted by $z_0$, the point $z_0$ is a simple parabolic fixed point.
Elementary considerations imply that $B_0^R$ is simply-connected, it is bounded since a neighborhood of $\infty$ belongs to the
basin of the super-attracting cycle of period $2$.

By Montel's Theorem $\partial B_0^R=J(R)$. The rational map $R$ can be described as a conformal mating of the basilica
$z\mapsto z^2-1$ and the parabolic map $z\mapsto z+z^2$. The Julia set of $R$ is locally connected by \cite{TY}. 

The other fixed point of $R$ is $\alpha=\frac{2}{3}$. It is repelling: $R'(\alpha)=-\frac{5}{4}$. It is elementary to verify that
the interval $[0,\alpha]$ is invariant under the second iterate of $R$, and that
$$R^2(x)<x\text{ for }x\in(0,\alpha).$$
Hence,  the interval $(0,\alpha)$ belongs to the super-attracting basin, and so does the interval $(\alpha,\infty)$. The point $\alpha$ is
on the boundary of the immediate basin $B_0^R$ and hence must be accessible from $B_0^R$. By real symmetry, $\alpha$ is bi-accessible
from $B_0^R$. Consider the parabolic renormalization $F\equiv \cP(R)$. Preimages of $\alpha$ are dense in $\partial B_0^R$, and therefore
the boundary $\cD(F)$ contains locally conformal preimages of a neighborhood of $\alpha$ in $J(R)$. In particular, 
$\partial \cD(F)$ contains bi-accessible points, and therefore is not Jordan. On the other hand, 
$\partial \cD(F)$ is a locally homeomorphic image of a locally connected set, and hence is locally connected.
Thus, $F\in \claa\setminus\cla$. Note that by \thmref{class-1c}, the parabolic renormalization $\cP(F)\in\cla.$

\end{document}